\documentclass[11pt]{article}
\usepackage{latexsym,amsfonts,amssymb,amsmath,amsthm}
\usepackage{graphicx}
\usepackage{subcaption}

\usepackage[usenames,dvipsnames]{color}
\usepackage{ulem}
\usepackage{caption}
\usepackage{float}
\usepackage{color}
\usepackage{bbm}
\usepackage{comment}
\begin{document}

\newtheorem{tm}{Theorem}[section]
\newtheorem{prop}[tm]{Proposition}
\newtheorem{defin}{Definition}[section]
\newtheorem{coro}[tm]{Corollary}
\newtheorem{lem}[tm]{Lemma}
\newtheorem{assumption}{Assumption}[section]
\newtheorem{rk}{Remark}[section]
\newtheorem{nota}{Notation}[section]
\numberwithin{equation}{section}

    \newcommand{\lb}{\label}
  \newcommand{\beq}{\begin{equation}}
    \newcommand{\eeq}{\end{equation}}

\newcommand{\stk}[2]{\stackrel{#1}{#2}}
\newcommand{\dwn}[1]{{\scriptstyle #1}\downarrow}
\newcommand{\upa}[1]{{\scriptstyle #1}\uparrow}
\newcommand{\nea}[1]{{\scriptstyle #1}\nearrow}
\newcommand{\sea}[1]{\searrow {\scriptstyle #1}}
\newcommand{\csti}[3]{(#1+1) (#2)^{1/ (#1+1)} (#1)^{- #1
 / (#1+1)} (#3)^{ #1 / (#1 +1)}}
\newcommand{\RR}[1]{\mathbb{#1}}

\newcommand{\rd}{{\mathbb R^d}}
\newcommand{\ep}{\varepsilon}
\newcommand{\rr}{{\mathbb R}}
\newcommand{\alert}[1]{\fbox{#1}}
\newcommand{\eqd}{\sim}

\def\un{\underline}
\def\ba{\overline}

\def\p{\partial}
\def\R{{\mathbb R}}
\def\N{{\mathbb N}}
\def\Q{{\mathbb Q}}
\def\C{{\mathbb C}}
\def\l{{\langle}}
\def\r{\rangle}
\def\t{\tau}
\def\k{\kappa}
\def\a{\alpha}
\def\la{\lambda}
\def\De{\Delta}
\def\de{\delta}
\def\ga{\gamma}
\def\Ga{\Gamma}
\def\ep{\varepsilon}
\def\eps{\varepsilon}
\def\si{\sigma}
\def\Re {{\rm Re}\,}
\def\Im {{\rm Im}\,}
\def\E{{\mathbb E}}
\def\P{{\mathbb P}}
\def\Z{{\mathbb Z}}
\def\D{{\mathbb D}}
\def\calS{\mathcal{S}}

\newcommand{\ceil}[1]{\lceil{#1}\rceil}

\newcommand{\bbR}{{\mathbb R}}

\newcommand{\calB}{{\mathcal B}}

\title{The spreading of global solutions of chemotaxis systems with logistic source and consumption on $\mathbb{R}^{N}$}

\author{
Zulaihat Hassan,  Wenxian Shen, and Yuming Paul Zhang  \\
Department of Mathematics and Statistics\\
Auburn University, Auburn, AL 36849\\
U.S.A. }

\date{}
\maketitle

\begin{abstract}

This paper investigates the spreading properties of globally defined bounded positive solutions of a chemotaxis system featuring a logistic source and consumption:
\[
\left\{
\begin{aligned}
&\partial_tu=\Delta u - \chi\nabla\cdot(u\nabla v)+ u(a-bu),\quad  &(t,x)\in {(0},\infty)\times\mathbb{R}^N, \\
&{\tau \partial_tv}=\Delta v-uv,\quad & (t,x)\in {(0},\infty)\times\mathbb{R}^N,
\end{aligned}
\right.
\]
where $u$ represents the population density of a biological species, and $v$ denotes the density of a chemical substance. We show that the presence of the chemical does not hinder the spreading of the species in general, and it does not accelerate the spreading speed under conditions that $v(0,\cdot)$ decays spatially or the parameters satisfy $0<-\chi\ll 1$ and $\tau=1$. In the proof, we establish a novel connection between $u$ and $v$ using the sup- and inf-convolution techniques for viscosity solutions. Additionally, our numerical simulations reveal a noteworthy phase transition in $\chi$: for $v(0, \cdot)$ uniformly distributed across space, the spreading speed accelerates only when $\chi$ surpasses a critical positive value.

\end{abstract}

\medskip


\medskip


\tableofcontents
\section{Introduction}

The current paper is devoted to the study of the asymptotic behavior of globally defined bounded solutions of  the following parabolic-parabolic chemotaxis system with logistic source  and consumption 
\begin{equation}\label{main-Eq}
\begin{cases}
u_{t}=\Delta u - \chi\nabla\cdot(u\nabla v)+ u(a-bu),\quad  & (t,x)\in {(0},\infty)\times\R^N, \\
{\tau v_t}=\Delta v-uv,\quad & (t,x)\in {(0},\infty)\times\R^N,
\end{cases}
\end{equation}
where $\chi$ is a constant that can be either positive or negative, and $a$, $b$, and $\tau$ are positive constants. In \eqref{main-Eq}, $u(t,x)$ denotes the population density of some biological species and $v(t,x)$ denotes the population density of  some  chemical substance at  time $t$ and space location $x$; $\chi$ is referred to as the chemotaxis sensitivity coefficient; the reaction term $u(a-bu)$ is referred to as a logistic  source; and $\tau$ is linked to the diffusion rate of the chemical substance.
The chemotaxis sensitivity coefficient  $\chi>0$ indicates that the  biological species is being attracted by  the chemical signal,  and $\chi<0$ indicates that the biological species is repelled by the chemical signal.

Chemotaxis is a phenomenon in which the movement of living organisms or cells is partially oriented along the gradient of certain chemicals. This phenomenon plays an important role in many biological processes, such as embryo formation and tumor development. The development of mathematical models for chemotaxis dates back to the pioneering works of Keller and Segel in the 1970s (\cite{KeSe0, KeSe1, KeSe2}).
Since these pioneering works, a large amount of research has been carried out on various chemotaxis models. See, for example, {\cite{HeVe1, HeVe2, JaLu, Nag, Tao, ZhLi}, etc. for the study of global existence and finite-time blow-up in the following minimal chemotaxis model on a bounded domain,
\begin{equation}
\label{eq1}
\begin{cases}
u_t=\Delta u-\chi \nabla \cdot (u\nabla v),\quad &x\in\Omega,
\cr
\tau v_t=\Delta u +g(u,v)-h(u,v)v,\quad & x\in\Omega,
\cr
\frac{\partial u}{\partial n}=\frac{\partial v}{\partial n}=0,\quad & x\in\partial \Omega,
\end{cases}
\end{equation}
where $\Omega\subset\R^N$ is a smooth bounded domain.
\medskip

Before diving into a more detailed discussion of this problem and the history of past developments, let us briefly describe the main theoretical results of the paper for the benefit of  the readers. Firstly, we show a lower bound for spreading speeds of solutions to  \eqref{main-Eq} with general non-negative initial data $u_0$ in Theorem \ref{speed-lower-bound-thm}. This lower bound equals the spreading speed when $\chi= 0$, that is, when chemotaxis is absent. Theorem \ref{speed-upper-bound-thm} states that the chemical substance does not induce infinitely fast spreading of $u$ and $v$. Finally, we prove that the upper bound of the spreading speed is the same as the lower bound, showing no acceleration due to chemotaxis, under certain conditions: when $v(0,\cdot)$ decays in space (see Theorem \ref{speed-thm}(1)), and when $0<-\chi\ll 1$ and $\tau=1$ (see Theorem \ref{speed-thm}(2)).

Let us highlight the proof of Theorem \ref{speed-thm}(2), where we uncover a novel connection between the population densities of the species $u$ and the chemical substance $v$. Specifically, we introduce the variables $\zeta:=1-v$ and
\[
w:=\zeta^p/u\qquad\text{ for some }p>1,
\]
and we show that, under the conditions of the theorem, $w$ is uniformly finite in both space and time after a short initial period.
For more details, we refer the readers to the discussion in Remark \ref{speed-rk} (4) and Lemma \ref{L.8.3}.

\medskip

Now, we discuss the existing literature on \eqref{eq1}.
In \cite{Tao, ZhLi},  \eqref{eq1} was considered with $\tau = 1$, $g(u, v) = 0$,  $h(u,v) = u$ and this model is referred to as chemotaxis model with consumption. It is proved in \cite[Theorem 1.1]{Tao} that this equation has a unique globally defined bounded classical solution with non-negative initial functions
$(u_0,v_0)\in (W^{1,p}(\Omega))^2$ for some $p>N$  provided
that
\begin{equation}
\label{bounded-domain-cond1}
0<\|v_0\|_{L^\infty(\Omega)}\cdot  \chi<\frac{1}{6(N+1)}.
\end{equation}
The authors in \cite{ZhLi} proved  that   any globally defined bounded positive classical solutions of  \eqref{eq1} converge to $(\frac{1}{|\Omega|}\int_\Omega u_0,0)$ as $t\to\infty$.
In  \cite{HeVe1,  HeVe2, JaLu,  Nag},  \eqref{eq1} is studied with $g(u, v) = \mu u,$ and $h(u,v) = \lambda$ ($\mu, \lambda >0
).$ It is known that finite-time blow-up does not occur when $N=1$ and may occur when $N\ge 2$ (see \cite{HeVe2}, \cite{Nag}, etc.).

We refer the reader to \cite{IsSh, LaWa, OsTsYa, OsYa1, WaKhKh, Win, ZhLiBaZo}, etc. for the study of the following chemotaxis model with a logistic source on a bounded domain,
\begin{equation}
\label{eq2}
\begin{cases}
u_t=\Delta u-\chi \nabla \cdot (u\nabla v)+u(a-bu),\quad &x\in\Omega,
\cr
\tau v_t=\Delta u +g(u,v)-h(u,v)v,\quad &x\in\Omega,
\cr
\frac{\partial u}{\partial n}=\frac{\partial v}{\partial n}=0,\quad &x\in\partial \Omega.
\end{cases}
\end{equation}
Assume $\tau = 1,$ $g(u, v) = 0,$ and $h(u,v) = u$. It is proved in \cite[Theorem 3.3]{WaKhKh} that
 \eqref{eq2} has a unique globally defined bounded classical solution with non-negative initial functions
$(u_0,v_0)\in (W^{1,p}(\Omega))^2$ for some $p>N$,  provided
that
 \eqref{bounded-domain-cond1}  holds.  It is also proved in \cite{LaWa}   that this equation has a unique globally defined bounded classical solution with  non-negative initial $(u_0,v_0)\in (C(\bar \Omega)\times C^1(\bar\Omega))$  provided that $\chi\|v_0\|_{L^\infty(\Omega)}$ is small relative to $b$ (see \cite[Theorem 1.1]{LaWa}), and that any positive bounded globally defined classical solution converges to $(\frac{a}{b},0)$ as $t\to\infty$ (see \cite[Theorem 1.2]{LaWa}).
Furthermore, if $\tau = 1, g(u, v) = \mu u,$ and $h(u,v) = \lambda$, it is known (see \cite{IsSh, Win}) that finite-time blow-up does not occur provided
\begin{equation*}
b>{N|\chi|\mu}/{4}.
\end{equation*}

The study of chemotaxis models  on the whole space can provide some deep insight into the influence of chemotaxis on the evolution of biological species from  some different angles. For instance, it can help determine whether chemotaxis would speed up or slow down the invasion of the biological species. Also, there have been numerous studies investigating chemotaxis models on the entire space. For instance, refer to \cite{Bra, QgChOt, HaHe, Hend, HeRe, LiPa, LiWa, SaSh, SaSh2, SaSh3, SaShXu,ShXu,WaOu}, etc. for investigations into the following chemotaxis model:
\begin{equation}\label{Eq3}
\begin{cases}
u_{t}=\Delta u - \chi\nabla\cdot(u\nabla v)+ u(a-bu),&x\in\R^N,\\ 
{\tau v_t}=\Delta v-\lambda v + \mu u,&x\in\R^N,\\
u(0, x) = u_0(x), \, \, v(0, x) = v_0(x)  &x\in\R^N.
\end{cases}
\end{equation}

Observe that, in the absence of chemotaxis (i.e., $\chi= 0$),  the dynamic of \eqref{Eq3} is governed by the following reaction-diffusion equation,
\begin{equation}\label{fisher-kpp}
u_t=\Delta u+u(a-bu),\quad x\in\R^N.
\end{equation}
The equation is also known as Fisher-KPP equation
due to the pioneering works of Fisher \cite{Fisher} and Kolmogorov, Petrowsky, Piskunov \cite{KPP} on traveling wave solutions and take-over properties.
The spreading properties of \eqref{fisher-kpp} are well-established.
Equation \eqref{fisher-kpp} has  traveling wave solutions $u(t,x)=\phi(x\cdot\xi-ct)$ ($\xi\in \mathbb{S}^{N-1}$)
connecting $\frac{a}{b}$ and $0$ $(\phi(-\infty)=\frac{a}{b},\phi(\infty)=0)$ for all speeds $c\geq 2\sqrt a$ and there is no such traveling wave solutions of slower speeds.  For any given bounded $u_0\in C(\R^N, \R^{+})$ with
nonempty compact support,
\begin{equation}
\label{FKPP-speed-eq1}
\lim_{t\to \infty}\sup_{|x| \le c't}\left|u(t, x)-\frac{a}{b}\right|=0\quad \forall\,  c'<2\sqrt a,
\end{equation}
and
\begin{equation}
\label{FKPP-speed-eq2}
\lim_{t\to\infty}\sup_{|x|\ge c''t}u(t, x)=0 \quad \forall \, c''>2\sqrt a
\end{equation}
{(see \cite{ArWe1,ArWe2})}.
Thanks to \eqref{FKPP-speed-eq1} and \eqref{FKPP-speed-eq2}, in literature, the minimal wave speed  $c^*:=2\sqrt a$ is  also called the {\it
spreading speed} of  \eqref{fisher-kpp}.

Considering \eqref{Eq3},  it is an important question whether the chemotaxis speeds up or slows down the spreading of the biological species.
When $\tau = 1$, it is proved in  \cite{ShXu} that  if $b>\frac{N\chi\mu}{4},$ chemotaxis neither speed up nor slow down the spreading speed of the biological species  in the sense that  \eqref{FKPP-speed-eq1} and \eqref{FKPP-speed-eq2} hold for any solution $(u(t,x),v(t,x))$ of \eqref{Eq3} with initial data  $(u_0(\cdot),v_0(\cdot))$ whose supports are nonempty and compact.  When $\tau = 0$, it is proved in \cite{SaShXu} that if $b>\chi\mu$ and $\big(1+\frac{(\sqrt{a}- \sqrt{\lambda})_+}{2(\sqrt{a}+ \sqrt{\lambda})}\big)\chi\mu\le b,$ then again chemotaxis neither slows down nor speeds up the spreading speed. 
Assuming $\tau = 0$ and $N=1$, the authors of \cite{QgChOt} proved that negative chemotaxis speeds up the spatial spreading of the biological species in the sense that the minimal wave speed of traveling waves of \eqref{Eq3} must be greater than $c^*=2\sqrt a$ {when $-\chi$ is sufficiently large}.

 Very recently, the authors of the current paper  studied the global existence of  the chemotaxis models \eqref{main-Eq} and \eqref{Eq3}   in a unified way.   Sufficient conditions are provided for  positive classical solutions to exist globally and stay bounded (see \cite[Theorems 1.2-1.4]{HSZ}.
In this paper, we focus on the study of asymptotic behavior
of globally defined bounded solutions of \eqref{main-Eq}, specially, spreading speeds of  globally defined bounded positive solutions of \eqref{main-Eq}.

In the following, we introduce some notations and definitions in subsection 1.1,  state our main theoretical results, {and discuss novel ideas and techniques}
in subsection 1.2, and present some numerical observations in subsection 1.3.

\subsection{Notations and definitions}

In this subsection, we introduce some standing notations and the definition of spreading speeds.
Let us define
\begin{equation*}
X:=C_{\rm unif}^b(\R^N)=\big\{u\in C(\R^N)\,|\, u(x)\,\,\text{is uniformly continuous in}\,\, x\in\R^N\,\, {\rm and}\,\, \sup_{x\in\R^N}|u(x)|<\infty\big\}
\end{equation*}
equipped with the norm $\|u\|_\infty:=\sup_{x\in\R^N}|u(x)|$. We also set
$$
X^+:=\{u\in X\,|\, u\ge 0\},
$$
and
$$
X_c^+:=\big\{u\in X^+\,|\, {\rm supp}(u)\not=\emptyset,\,\, {\rm supp}(u)\,\, \text{is compact}\big\}.
$$

Next, we let
$$
X_1:=C_{\rm unif}^{1,b}(\R^N)=\left\{u\in X\,\big|\, \frac{\partial  u}{\partial x_i} \in  X\,\,
{\rm for}\,\, i=1,2,\cdots, N\right\}
$$
equipped with the norm $\|u\|_{X_1}:=\|u\|_\infty+\sum_{i=1}^N  \|\frac{\partial  u}{\partial x_i}\|_\infty$. We write $X_1^+:=X_1\cap X^+$.
Finally, we use $C_{\rm unif}^{2+\alpha, b}(\R^N)$ 
to denote all $u\in X_1$ such that $\frac{\partial^2 u}{\partial x_i\partial x_j}$ are uniformly bounded and $\alpha$-H\"{o}lder continuous for $i,j=1,2,\ldots, N$,
and use the notation
$$
C_0(\R^N):=\left\{u\in C(\R^N)\,\big|\, \lim_{|x|\to\infty} u(x)=0\right\}.
$$

For any given $(u_0, v_0)\in X^+\times X_1^+$, we denote by $(u(t,x ;u_0, v_0),v(t, x ; u_0, v_0))$ the classical solution of \eqref{main-Eq} satisfying
 \begin{equation*}
 \lim_{t\to 0+}(\|u(t,\cdot;u_0,v_0)-u_0(\cdot)\|_{X}+\|v(t,\cdot;u_0,v_0)-v_0(\cdot)\|_{X_1})=0,
\end{equation*}
\begin{equation*}
 u(\cdot,\cdot;u_0, v_0) \in C([0, T_{\max} ), X )\cap C^1((0,T_{\max}), X),
    \end{equation*}
    \begin{equation*}
  v(\cdot,\cdot;u_0, v_0) \in C([0, T_{\max} ), X_1 )\cap C^1((0,T_{\max}), X_1).
    \end{equation*}
Note that, by the comparison principle for parabolic equations, for every $(u_0, v_0)\in X^{+}\times X_1^+$, it always holds that $u(t, x ;u_0,v_0)\geq 0$ and $v(t, x;u_0, v_0)\geq 0$ for all $t>0$ in the existence interval of the solution.

\smallskip

The objective of this paper is to study the asymptotic behavior of globally defined positive solutions. We introduce  the following definitions.
For given $(u_0, v_0)\in X^+\times X_1^+$,  assuming that $(u(t,x;u_0,v_0),v(t,x;u_0,v_0))$
exists for all $t>0$, let
$$
S_{\rm low}(u_0,v_0):=\{c\,|\, c>0,\,\, \liminf_{t\to\infty}\inf_{|x|\le  ct}u(t,x;u_0,v_0)>0\}
$$
and
$$
S_{\rm up}(u_0,v_0):=\{c\,|\, c>0,\,\, \limsup_{t\to\infty}\sup_{|x|\ge ct} u(t,x;u_0,v_0)=0\}.
$$
We define
\begin{equation*}
c_{\rm low}^*(u_0,v_0):=\sup \{c\,|\, c \in S_{\rm low}(u_0,v_0)\}
\end{equation*}
and
\begin{equation*}
c_{\rm up}^*(u_0,v_0):=\inf\{c|\, c\in S_{\rm up}(u_0,v_0)\},
\end{equation*}
where $c_{\rm low}^*(u_0,v_0)=0$ if $S_{\rm low}(u_0,v_0)=\emptyset$ and $c_{\rm up}^*(u_0,v_0)=\infty$ if $S_{\rm up}(u_0,v_0)=\emptyset$.
By the definition of $c_{\rm low}^*(u_0,v_0)$ and $c_{\rm up}^*(u_0,v_0)$,  if $c_{\rm low}^*(u_0,v_0)>0$, then
\begin{equation*}
\liminf_{t\to\infty}\inf_{|x|\le {c'} t}u(t,x;u_0,v_0)>0\quad \forall\, 0<{c'}<c^*_{\rm low}(u_0,v_0),
\end{equation*}
and if $c_{\rm up}^*(u_0,v_0)<\infty$, then
\begin{equation*}
\lim_{t\to\infty} \sup_{|x|\ge {c''}t}u(t,x;u_0,v_0)=0\quad \forall\, {c''}>c^*_{\rm up}(u_0,v_0).
\end{equation*}
If $c_{\rm low}^*(u_0,v_0)=c_{\rm up}^*(u_0,v_0)$, then $c^*(u_0,v_0):=c_{\rm low}^*(u_0,v_0)$ is called
the {\it spreading speed} of
 the solution $(u(t,x;u_0,v_0),v(t,x;u_0,v_0))$.
It is well known  that when $\chi=0$, we have $c_{\rm low}^*(u_0,v_0)=c_{\rm up}^*(u_0,v_0)=2\sqrt a$ for any
$u_0\in X_c^+$ and $v_0\in X_1^+$ (see \cite{ArWe1,ArWe2}).

Many interesting questions arise when $\chi\not =0$.
For example, whether $c_{\rm low}^*(u_0,v_0)$ is positive and $c_{\rm up}^*(u_0,v_0)$ is finite;
whether $c_{\rm low}^*(u_0,v_0)\ge 2\sqrt a$,  in other words, whether the chemotaxis does not
slow down the population's spreading; whether $c_{\rm up}^*(u_0,v_0)\le 2 \sqrt a$, that is, whether
the chemotaxis does not speed up the population's spreading;  whether $c_{\rm low}^*(u_0,v_0)=c_{\rm up}^*(u_0,v_0)$, that is, the population has a single spreading speed; in the case $c_{\rm low}^*(u_0,v_0)>0$, whether $u(t,x;u_0,v_0)$ converges to $\frac{a}{b}$ in the region $|x|<ct$ for any $0<c<c_{\rm low}^*(u_0,v_0)$, which is strongly related to the stability of the constant equilibrium $(\frac{a}{b},0)$; {and how $\chi$ affects the spreading properties; etc.
The goal of this paper is to answer some of these questions. }

\subsection{Main results, biological interpretations, and techniques}

In this subsection, we state our main results on
the asymptotic behavior of  globally defined positive solutions,  give some biological interpretations of the results,  and highlight the techniques employed/developed  in the proofs of the results.

Our first main result provides a lower bound of  the spatial spreading speeds of solutions to  \eqref{main-Eq} with general non-negative initial data $u_0$.

\begin{tm}[Lower bound of spreading speeds]
\label{speed-lower-bound-thm}
Suppose that $u_0\in X^+$ and $v_0\in X^+_1$, and  $(u(t,x;u_0,v_0), v(t,x;u_0,v_0))$ is a globally defined bounded solution of \eqref{main-Eq}.
If ${\rm supp}(u_0)={\rm cl}\{x\in\R^N\,|\, u_0(x)>0\}$ is nonempty,  then the following hold.
\begin{itemize}
\item[(1)] $c_{\rm low}^*(u_0,v_0)\ge 2\sqrt a$, or equivalently,
\begin{equation}
\label{lower-bound-eqq1}
\liminf_{t\to\infty}\inf_{|x|\le {c'} t}u(t,x;u_0,v_0)>0\quad \forall\, 0<{c'}<2\sqrt a.
\end{equation}

\item[(2)]
\begin{equation}
\label{lower-bound-eqq2}
\lim_{t\to\infty}\sup_{|x|\le {c'} t}\left|u(t,x;u_0,v_0)-\frac{a}{b}\right|=0\quad \forall\, 0<{c'}<c^*_{\rm low}(u_0,v_0),
\end{equation}
and
\begin{equation}
\label{lower-bound-eqq3}
\lim_{t\to\infty} \sup_{|x|\le {c'}t}v(t,x;u_0,v_0)=0\quad \forall\, 0< {c'}<c^*_{\rm low}(u_0,v_0).
\end{equation}
\end{itemize}
\end{tm}

\begin{rk}
{\rm
\begin{itemize}
\item[(1)]
As it is mentioned in the above, when $\chi=0$,   $c_{\rm low}^*(u_0,v_0)=2\sqrt a$ for any $u_0\in X_c^+$ and $v_0\in X_1^+$ (in this case, $u(t,x;u_0,v_0)$ is independent of $v_0$).
Note that, when $\chi<0$, the chemical substance is a chemorepellent, and when $\chi>0$, the chemical substance is a chemoattractant. Theorem \ref{speed-lower-bound-thm} reveals an important biological observation:   chemical
substance does not slow down the propagation of the biological species with nonzero initial distribution even when the chemical substance is a chemorepellent.

\item[(2)] \eqref{lower-bound-eqq1} is proved by nontrivially modified arguments of \cite[Theorem 1.2(1)]{ShXu}.
The key idea is to show that, for any given $0<c'<2\sqrt a$,  $u(t,x+c\xi t;u_0,v_0)$
is bounded away from zero uniformly in $0<c<c'$, $\xi\in \mathbb{S}^{N-1}$, $|x|\le l$ for some $l>0$, and $t\gg 1$ (see Lemmas \ref{asymptotic2-lem2}, \ref{lem-3}, and \ref{lem-4}).
In the proof, we make use of special Harnack inequalities for bounded solutions of \eqref{main-Eq} established in Lemmas \ref {asymptotic-lm1} and \ref{asymptotic-lm2}. Specifically, the local-in-time Harnack inequality, as stated in Lemma \ref {asymptotic-lm1}, is a slight generalization of \cite[Theorem 1.2]{BHR} and \cite[Lemma 2.2]{HaHe}, which address advection Fisher-KPP equations.
Additionally, our approach employs the principal eigenvalue and eigenfunction of some linearized operator for $u$, and the comparison principal for parabolic equations.
\end{itemize}
}
\end{rk}

The following theorem is on upper bound of   the spatial spreading speeds of solutions to  \eqref{main-Eq} with compactly supported $u_0$.

\begin{tm}[Upper bound of spreading speeds]
\lb{speed-upper-bound-thm}
{Suppose that ${ u_0\in X_c^+}$ and ${v_0\in X^+_1}$}, and  $(u(t,x;u_0,v_0), v(t,x;u_0,v_0))$ is a globally defined bounded solution of \eqref{main-Eq}. Then we have
\begin{itemize}
\item[(1)] $ c_{\rm up}^*(u_0,v_0)<\infty$. Moreover,  for any $c''>c_{\rm up}^*(u_0,v_0)$, there is $M>0$ such that
\begin{equation}
\label{upper-bound-eqq1}
u(t,x;u_0,v_0)\le M e^{-\sqrt{a} (|x|-c'' t)}\quad \forall\, t>0,\quad x\in\R^N.
\end{equation}

\item[(2)] There exist $C,\gamma>0$ depending on $c{''}-c^*_{up}(u_0,v_0)$ such that
\begin{equation*}
\sup_{|x|\ge {c''}t}|v(t,x;u_0,v_0)-V(t,x;v_0)|\leq Ce^{-\gamma t}.
\quad \forall\, t>0,\,\, {c''}
>c^*_{\rm up}(u_0,v_0),
\end{equation*}
where $V(t,x):=V(t,x;v_0)$  is the solution of
\beq\lb{1.19}
\begin{cases}
\tau  V_t=\Delta V,\quad  &x\in\R^N,\,t>0\cr
V(0,x)=v_0(x),\quad &x\in\R^N.
\end{cases}
\eeq
\end{itemize}
\end{tm}

\begin{rk}
{\rm
\begin{itemize}
\item[(1)]
Theorem \ref{speed-upper-bound-thm} implies that the chemical substance does not drive the biological species spreads infinitely fast.

\item[(2)]  Theorem \ref{speed-upper-bound-thm}(1) is proved by the application of   special Harnack inequalities for bounded solutions of \eqref{main-Eq} established in Lemmas \ref {asymptotic-lm1} and \ref{asymptotic-lm2}. The exponential decay property \eqref{upper-bound-eqq1} for $u(t,x;u_0,v_0)$ and   the representation of $v(t,x;u_0,v_0)$ via the Duhamel's principle are the key ingredients  in the proof of   Theorem \ref{speed-upper-bound-thm}(2).

\end{itemize}
}
\end{rk}

We  point out that the methods developed in  the proofs of Theorems \ref{speed-lower-bound-thm} and \ref{speed-upper-bound-thm}
 can be applied to the study of the asymptotic dynamics of the following modified chemotaxis model for $\sigma>0$ (\eqref{main-Eq} corresponds to the case when $\sigma=1$),
\begin{equation}
\label{main-Eq1}
\begin{cases}
u_{t}=\Delta u - \chi\nabla\cdot  (u  \nabla v)+ u(a-bu^\sigma),\quad  &(t,x)\in[0,\infty)\times \R^N, \\
 \tau v_t= \Delta v-u v,\quad &(t,x)\in[0,\infty)\times \R^N.
\end{cases}
\end{equation}
Theorem \ref{speed-lower-bound-thm} and Theorem \ref{speed-upper-bound-thm} with $\frac{a}{b}$ being replaced by $(\frac{a}{b})^{{1}/{\sigma}}$ hold for globally defined bounded positive classical solutions of \eqref{main-Eq1}.

\smallskip

Our last two theorems discuss various sufficient conditions  for
 the existence of spreading speed of globally defined bounded solutions of \eqref{main-Eq}
and \eqref{main-Eq1}.

\begin{tm} [Existence of spreading speeds]
\label{speed-thm}
Suppose that $u_0\in X^+_c$ and $v_0\in X^+_1$, and $(u(t,x;u_0,v_0), v(t,x;u_0,v_0))$ is a globally defined bounded solution of \eqref{main-Eq}.
\begin{itemize}
\item[(1)]
If $v_0\in C_0(\R^N)$ or $v_0\in L^p(\R^N)$ for some $p\ge 1$, then
$$
c^*_{\rm low}(u_0,v_0)=c^*_{\rm up}(u_0,v_0)=2\sqrt a.
$$

\item[(2)] Further assume $\tau=1$, and that $u_0\in X^+_1$, $v_0\in C_{\rm unif}^{2+\alpha, b}(\R^N)$ for some $\alpha>0$, and $1-v_0\in X_c^+$.   Then there exists  $\chi_0>0$ such that
for any $-\chi_0<\chi<0$, we have
$$
c^*_{\rm low}(u_0,v_0)=c^*_{\rm up}(u_0,v_0)=2\sqrt a.
$$
\end{itemize}

\end{tm}

\begin{tm}
\label{speed-thm1}
Suppose that $(u(t,x;u_0,v_0), v(t,x;u_0,v_0))$ is a globally defined bounded solution
of \eqref{main-Eq1}. Assume that $\tau=1$,  $\sigma\in (0,1)$,  $1-v_0$, $u_0\in X_c^+\cap X_1$, and $v_0\in C_{\rm unif}^{2+\alpha, b}(\R^N)$ for some $\alpha>0$.  Then there exists  $\chi_0>0$ such that if $|\chi|<\chi_0$, the conclusion of Theorem \ref{speed-thm} (2) holds the same.
\end{tm}

\begin{rk}
\label{speed-rk}
{\rm
\begin{itemize}
\item[(1)] Note that the biological species gradually consumes the chemical substance over time. Consequently, if initially, the chemical substance does not occupy the entire space to a certain extent, as indicated by conditions such as $v_0\in L^p(\R^N)$ for some $p\ge 1$ or $v_0\in C_0(\R^N)$, it intuitively follows that the spreading of the species should be unaffected. This intuition aligns with the findings of Theorem \ref{speed-thm} (1).  

\item[(2)] Recall that when $\chi<0$ (resp. $\chi>0$), the chemical substance acts as a chemorepellent (resp. chemoattractant). If $v_0$ is not negligible as $|x|\to\infty$, it is reasonable to expect that negative chemotaxis wouldn't enhance the spread of the biological species, while a chemoattractant might accelerate it. This expectation is partly supported by our Theorem \ref{speed-thm} (2) when $\chi<0$ but $|\chi|\ll 1$. However, the situation regarding chemoattractants is more complicated. Indeed, as confirmed by numerical simulations (refer to Section \ref{simulation}), a phase transition emerges: there exists a critical threshold, denoted as $\chi^* > 0$, such that the spread accelerates when $\chi>\chi^*$, while it maintains a constant speed of $2\sqrt{a}$ when $\chi<\chi^*$. Although Theorem \ref{speed-thm1} provides another evidence, the problem remains open. 



\item[(3)] The additional smooth assumption $u_0\in  X^+_1$ and $v_0\in C_{\rm unif}^{2+\alpha, b}(\R^N)$  guarantees the boundedness of $\|\nabla u(t,\cdot;u_0,v_0)\|_\infty$ and $\|\Delta v(t,\cdot;u_0,v_0)\|_\infty$ up to $t=0$, which is a technical assumption used in our proof. 

\item[(4)] To establish Theorem \ref{speed-thm}(2) and Theorem \ref{speed-thm1}, we explore a novel connection between the population densities of the species $u$ and the chemical substance $v$. Specifically, we introduce the variables $\zeta:=1-v$ and
\[
w:=\zeta^p/u\qquad\text{ for some }p>1,
\]
and $w$ turns out to satisfy a favorable parabolic equation, see the proof of Lemma \ref{L.8.3}.
Leveraging the smoothing effect inherent in the parabolic equation, we are able to show $\zeta^p\lesssim u$ which leads to the estimates of $\zeta$ and $\nabla \zeta$ in terms of $u^{{1}/{p}}$.
However, a significant difficulty arises in this approach: since $u_0$ and $\zeta_0:=1-v_0$ are known to have compact support only, $w$ is not well-defined at time $0$. To circumvent this obstacle, we prove that $\zeta^p\lesssim u$ at some small positive times by using the comparison principle and employing the inf- and sup- convolution technique for viscosity solutions.
Having $w$ be finite over at least a small positive time interval, we propagate this property to all subsequent times through the parabolic equation satisfied by $w$. Linking the two density variables $u$ and $v$ is essential in the analysis of spreading properties.

\end{itemize}
}
\end{rk}

\subsection{Numerical simulations and biological  indications}

As mentioned before, numerical experiments indicate a phase transition, that is there exists a positive critical value of $\chi$, beyond which the spreading speed accelerates, while it remains to be $2\sqrt{a}$ when $\chi<\chi^*$. Theorem \ref{speed-thm}(2) confirmed this expectation for $\chi<0$ but $|\chi|\ll 1$. In this subsection, we present the numerical experiments that explore the influence of chemotaxis on the spread of the biological species in \eqref{main-Eq} when $N=1$ and {$a=b=1$}.
The equation becomes
\begin{equation*}
\begin{cases}
u_{t}= u_{xx} - \chi(uv_x)_x+ u(1-u),\quad  &x\in \R \\
{\tau}{v_t}=v_{xx}-uv,\quad  & x\in \R.
\end{cases}
\end{equation*}

For given    $u_0\in X^+$ and $v_0\in X^+_1$, in order to see the behavior of $u(t,x;u_0,v_0)$ near $(t,ct)$,  we  consider
$ (\tilde u(t, x),\tilde v(t,x)):= (u(t, x+ct;u_0,v_0)$, $v(t,x+ct;u_0,v_0))$, which solves
\begin{equation}\label{num02}
\begin{cases}
\tilde u_{t}=\tilde u_{xx} + c\tilde u_x - \chi(\tilde u\tilde v_x)_x+ \tilde u(1-\tilde u),\quad & x\in \R, \\
{\tau}{ \tilde v_t}=\tilde v_{xx} + c\tau\tilde v_x-\tilde u\tilde v,\quad &x\in \R,\\
\tilde u(0, x) = u_0(x),\, \, \tilde v(0, x) = v_0(x), \quad & x \in \R.
\end{cases}
\end{equation}
For the numerical simulations, we use the following cut-off system of \eqref{num02}  on $(-L,L)$,
\begin{equation}\label{num03}
\begin{cases}
\tilde u_{t}=\tilde u_{xx} + c\tilde u_x - \chi(\tilde u\tilde v_x)_x+ \tilde u(1-\tilde u),\quad & x\in (-L, L), \\
{\tau}{\tilde v_t}=\tilde v_{xx} + c{\tau}\tilde v_x-\tilde u\tilde v,\quad& x\in (-L, L),\\
\tilde u(0, x) = u_0(x),\, \, \tilde v(0, x) = v_0(x), \quad& x \in (-L, L)
\end{cases}
\end{equation}
 complemented with the following boundary conditions:
\begin{equation}\label{BC}
    \tilde u(t, \pm L) = 0, \qquad \frac{\partial\tilde v}{\partial x}(t, \pm L)= 0.
\end{equation}

If $\chi=0$, it suffices to solve for the following Fisher-KPP equation with convection,
\begin{equation}
\label{reduced-fisher-kpp}
\begin{cases}
\tilde u_t=\tilde u_{xx}+c\tilde u_x +\tilde u(1-\tilde u),\quad -L<x<L\cr
\tilde{u}(0,x)=u_0,\,\, \tilde u(t,-L)=\tilde u(t,L)=0.
\end{cases}
\end{equation}
Following from the arguments of \cite[Theorem 2.2]{FrZh},
{we have the following dichotomy about the asymptotic dynamics of \eqref{reduced-fisher-kpp}: for fixed $c\in\mathbb{R}$ and $L>0$, either $\tilde u(t,x)\to 0$ as $t\to \infty$ uniformly in $x\in[-L,L]$ for any $u_0\in C([-L,L])$ and $u_0>0$, or \eqref{reduced-fisher-kpp} has a unique positive stationary solution $u^*(x)$ and
$u(t,x)\to u^*(x)$  as $t\to \infty$ uniformly for all $x\in[-L,L]$ and $u_0\in C([-L,L])$ with $u_0>0$. The former occurs when $\lambda(c,L)\le 0$ 
and
the latter occurs when  $\lambda(c,L)>0$, where $\lambda(c,L)$ is the principal eigenvalue of
$$
\begin{cases}
\tilde u_{xx}+c\tilde u_x+\tilde u=\lambda u,\quad -L<x<L\cr
\tilde u(-L)=\tilde u(L)=0.
\end{cases}
$$
One can check that if $c>2$ and $L\gg 1$, 
we have $\lambda (c,L)< 0$; and when $0<c<2$, we have $\lambda(c,L)>0$ if $L\gg1$. Thus, this confirms that the spreading speed for Fisher-KPP is $2\sqrt{a}$, which is $2$ here. 

When $\chi\not =0$, if  for some  $c>2$ and $L$ sufficiently large, $\tilde u(t,x)\not \to 0$ as $t\to\infty$, we can conclude that the chemotaxis speeds up the spreading of the species.}

We  choose the following initial functions  $u_0$ and $v_0$,
\begin{equation}\label{u0-v0}
u_0(x) = \begin{cases}
0,\quad  & x\le-1, \\
e^{\frac{1}{x^2 - 1}},\quad & x\in (-1, 1),\\
0, \quad & x\ge 1
\end{cases} \quad v_0(x) = 1,
\end{equation}
and $L=20$.
We compute the numerical solution  of \eqref{num03}+\eqref{BC} using the finite difference method (see Section \ref{simulation} for more detail). The following scenarios are observed numerically for ${\tau =1}$.

\begin{itemize}

\item[(i)] When $\chi>0$ and is not  large, for any $c>2$, $\tilde u(t,x)\to 0$ as $t\to\infty$, which indicates that small positive chemotaxis does not speed up the spreading of the biological species (see
the numerical experiments in subsection \ref{simulation-1}).

\item[(ii)] When $\chi>0$ and is large, $\tilde u(t,x)\not \to 0$ as $t\to\infty$ for some $c>2$, which indicates that large positive chemotaxis speeds up the spreading of the biological species (see the numerical experiments in subsection \ref{simulation-2}).

\item[(iii)] When $\chi<0$, for any $c>2$,  $\tilde u(t,x)\to 0$ as $t\to\infty$,
which indicates that negative chemotaxis does not speed up the spreading of the biological species
(see the numerical experiments in subsection \ref{simulation-3}).
\end{itemize}

We further compare the behavior of solutions for different values of $\tau$. We choose $\tau = 0.5$ and $\tau = 4.$ The result of the simulation shows that for large positive $\chi$, chemotaxis speeds up the spreading of the biological species for these $\tau$. 

The rest of the paper is organized as follows:
In Section 2, we present some preliminary materials for use in later sections, including a review of the global existence of classical solutions of \eqref{main-Eq}; special Harnack inequalities for bounded solutions of \eqref{main-Eq}; the comparison principle for viscosity solutions to general parabolic type equations; and convergence of globally defined bounded solutions of \eqref{main-Eq} to the constant solution $(\frac{a}{b},0)$. Section 3 is devoted to the investigation of the lower bounds of spreading speeds of global bounded solutions of \eqref{main-Eq} and the proof of Theorem \ref{speed-lower-bound-thm}. In Section 4, we study the upper bounds of spreading speeds of global bounded solutions to \eqref{main-Eq} and prove Theorem \ref{speed-upper-bound-thm}. The existence of spreading speeds is discussed in Section 5. Theorems \ref{speed-thm} and \ref{speed-thm1} are proved in this section. In Section 6, we present our numerical experiments. Finally, we give a proof of Lemma \ref{asymptotic-lm1} in the appendix.

\section{Preliminary}

In this section, we recall some results from \cite{HSZ} about global existence and $C^{1+\theta,2+\nu}$--boundedness  of
classical solutions of \eqref{main-Eq}, present a Harnack type inequality for \eqref{main-Eq}, and introduce the concept of viscosity solutions of general parabolic type equations and recall the comparison principle for viscosity solutions.

\subsection{Global existence  {and $C^{1+\theta, 2+\nu}$-boundedness} of classical solutions}


We first introduce the following notations. For given $0<\nu<1$, let
\begin{equation*}
C^{\nu, b}_{\rm unif}(\R^N)=\Big\{u\in X\,|\, \sup_{x,y\in\R^N,x\not =y}\frac{|u(x)-u(y)|}{|x-y|^\nu}<\infty\Big\}
\end{equation*}
with the norm $\|u\|_{\infty,\nu}:=\|u\|_\infty +\sup_{x,y\in\R^N,x\not =y}\frac{|u(x)-u(y)|}{|x-y|^\nu}$.
For given  $0<\theta<1$ and an interval $I\subset\R$, let
\begin{align*}
  C_{\rm unif} ^{\theta,\nu}(I\times \R^N)=\Big\{ u(\cdot,\cdot)\in C_{\rm unif}^b (I\times \R^N)\,|\,
\quad  \|u\|_{C_{\rm unif}^{\theta,\nu}(I\times \R^N)}<\infty\Big\}
\end{align*}
where the norm is given by
\[
\|u\|_{C_{\rm unif}^{\theta,\nu}(I\times \R^N)}:=\sup_{(t,x)\in I\times\R^N}|u(t,x)|+\sup_{(t,x),(s,y)\in I\times\R^N, (t,x)\not =(s,y)}\frac{|u(t,x)-u(s,y)|}{|t-s|^\theta+|x-y|^\nu}.
\]

Now, we recall the global existence of classical solutions of \eqref{main-Eq}  
and 
present a lemma on the $C^{1+\theta,2+\nu}$-boundedness of globally defined bounded classical solutions of \eqref{main-Eq}.

\begin{prop}
\label{global-existence-prop}  {\rm (Global existence)}
For any given  $u_0\in X^+$ and $v_0\in X_1^+$,  if
\begin{equation*}
|\chi|\cdot \|v_0\|_\infty <\max\left\{ {D^*_{\tau,N}, \,  b \cdot {C^*_{N}} }\right\} ,
\end{equation*}
then  the classical solution
$(u(t,x;u_0,v_0)$, $v(t,x;u_0,v_0))$ of \eqref{main-Eq}  exists for all $t>0$,  and
\[
\|u(t,\cdot;u_0,v_0)\|_\infty\quad\text{and}\quad\|\nabla v(t,\cdot;u_0,v_0)\|_\infty \text{ stays bounded as $t\to\infty$},
\]
where
\begin{equation*}
D^*_{\tau,N}:=\begin{cases}  \frac{2}{\tau N^* }\left(2\sqrt{\frac{(\tau^*)^2}4+\frac1{\tau N^*}}+j|\tau^*|\right)^{-1}\quad &{\rm if}\quad \sqrt{\frac{(\tau^*)^2}4+\frac1{\tau N^*}}>-j|\tau^*|\cr\cr
\frac{2}{ \tau N^*}\left(\sqrt{\frac{(\tau^*)^2}4+\frac1{\tau N^*}}\right)^{-1}\quad &{\rm if}\quad \sqrt{\frac{(\tau^*)^2}4+\frac1{\tau N^*}}\leq -j|\tau^*|
\end{cases}
\end{equation*}
with $N^*:=\max\{1,\frac{N}{2}\}$, $
\tau^*:=\frac{1}{\tau}-1$, and $j:=\text{\rm Sign}(\chi\tau^*)$ is the sign of $\chi\tau^*$,
and
\begin{equation*}
{{ C^*_{N}}:=\sup_{\gamma>\max\{1,N/2\}}\frac{\gamma}{\gamma-1}\left({ C_{\gamma+1,N}}\right)^{-\frac{1}{\gamma+1}},}
\end{equation*}
where $C_{\gamma+1,N}$ is  a constant associated with the maximal regularity for  the following parabolic equation
\begin{equation*}
\begin{cases}
\tau v_t =\Delta v -\alpha  v + g,\quad &x\in \R^N,\,\,  0<t<T\cr
v(0,x) = v_0(x),\quad  &x\in\R^N
\end{cases}
\end{equation*}
and $C_{\gamma+1,N}\leq ({C\gamma^2})^\gamma({\gamma-1})^{-\gamma-1} 2^{9n(\gamma-1)}$ with $C$ an absolute constant
(see \cite[Lemma 2.3, Theorem A.1]{HSZ}).
\end{prop}

The above proposition  follows from \cite[Theorem 1.2]{HSZ}. We comment that the estimates on $u$ and $v$ remain uniform as $\chi\to 0$, meaning that they depend only on an upper bound of $|\chi|$.

\begin{lem}
\label{derivative-boundedness-lm}
 {\rm ($C^{1+\theta,2+\nu}$-boundedness)}  Suppose  that $u_0\in X^+$ and $v_0\in X^+_1$, and  $(u(t,x;u_0,v_0)$, $ v(t,x;u_0,v_0))$ is a globally defined bounded solution \eqref{main-Eq}. Then  for any given $0<\nu\ll 1$, $0<\theta\ll 1$,  and $t_0>0$, there is $C>0$ depending on the system parameters, $\|v_0\|_{X_1}$,
$\sup_{t\geq 0}\|u(t,\cdot;u_0,v_0)\|_\infty$, $\nu$, $\theta$ and $t_0$ such that
\begin{equation}
\label{bound-on-holder-norm-eq}
\|w\|_{C_{\rm unif}^{\theta,\nu}([t_0,\infty)\times \R^N)}\le  C,
\end{equation}
where $w(t,x)$ is any one  of the following functions:  $u(t,x;u_0,v_0)$, $\partial_{x_i}u(t,x;u_0,v_0)$, $\partial_t u(t,x;u_0,v_0)$, $\partial^2_{x_i x_j}u(t,x;u_0,v_0)$, $v(t,x;u_0,v_0)$, $\partial_{x_i}v(t,x;u_0,v_0)$,
$\partial_t v(t,x$; $u_0,v_0)$ or $\partial^2_{x_ix_j}v(t,x;u_0,v_0)$ for $1\le i,j\le N$.

Moreover, if $u_0\in X_1^+$ and $v_0\in C_{\rm unif}^{2+\alpha, b}(\R^N)$ for some $\alpha>0$, then \eqref{bound-on-holder-norm-eq} holds for $t_0=0$ with $w$ being $v(t,x;u_0,v_0)$, $\partial_{x_i}v(t,x;u_0,v_0)$,
$\partial_t v(t,x$; $u_0,v_0)$ or $\partial^2_{x_ix_j}v(t,x;u_0,v_0)$ for $1\le i,j\le N$.
\end{lem}

The proof  of the above lemma follows from the outline of the proof of \cite[Proposition 2.1]{HSZ}. For the reader's convenience,  we provide a proof in the appendix.

\begin{rk}
\label{lm2.2-rk} 
It follows from Proposition \ref{global-existence-prop} that the classical solution $(u(t,x;u_0,v_0)$, $v(t,x;u_0,v_0))$  to \eqref{main-Eq}  exists globally when $|\chi|$ is sufficiently small. Moreover, from its proof, we have that
$\sup_{t\geq 0}\|u(t,\cdot;u_0,v_0)\|_\infty$ stays uniformly finite for all small $|\chi|$. Due to this, it follows from the proof of Lemma \ref{derivative-boundedness-lm} that the constant $C$ in \eqref{bound-on-holder-norm-eq} is uniform for all $|\chi|$ sufficiently small as well.
\end{rk}

We recall \cite[Proposition 2.1(1)]{HSZ} about local existence of solutions:   For any  given  $u_0\in X$ and $v_0\in  X_1$,  there is $T_{\max}:=T_{\max}(u_0,v_0)\in (0,\infty]$ such that \eqref{main-Eq}  has a unique classical solution
$(u(t,x;u_0,v_0),v(t,x;u_0,v_0))$ on $(0,T_{\max}(u_0,v_0))$ with $u(0,x;u_0,v_0)=u_0(x)$ and $v(0,x;u_0,v_0)=v_0(x)$. Moreover,
if $u_0,v_0\ge 0$, then $u(t,x;u_0,v_0)\ge 0$ and $v (t,x;u_0,v_0)\ge 0$ for $t\in [0,T_{\max} (u_0,v_0))$ and $x\in\R^N$.
We end up this subsection with  some  remarks on these solutions.
\begin{rk}
{\rm
\begin{itemize}
\item[(1)]  When $N=1$ or $2$,  for any $u_0\in X^+$ and $v_0\in X_1^+$, $T_{\max}(u_0,v_0)=\infty$ (see \cite[Remark 1.3]{HSZ}).

\item[(2)] Following the arguments of { Lemma \ref{derivative-boundedness-lm}}, for any given $u_0\in X^+$,  $v_0\in X^+_1$, $0<\nu\ll 1$, $0<\theta\ll 1$,  and $[t_0,T]\subset (0, T_{\max}(u_0,v_0))$, there is $C>0$ depending on $\chi$,  $\|v_0\|_{X_1}$,
$\sup_{t\in[0,T]}\|u(t,\cdot;u_0,v_0)\|_\infty$, $\nu$, $\theta$, and $t_0$ such that \eqref{bound-on-holder-norm-eq} holds with $C_{\rm unif}^{\theta,\nu}([t_0,T]\times \R^N)$ in place of $C_{\rm unif}^{\theta,\nu}([t_0,\infty)\times \R^N)$.


\item[(3)] Considering \eqref{main-Eq1},   by similar  arguments of \cite[Proposition 2.1(1)]{HSZ}, for any given $u_0\in  { X^+}$ and $v_0\in {X_1^+}$, there is  $T_{\max}^\sigma(u_0,v_0)\in (0,\infty]$ such that
\eqref{main-Eq1} has a unique classical solution $(u^\sigma(t,x;u_0,v_0),v^\sigma(t,x;u_0,v_0))$ on $(0,T_{\max}^\sigma(u_0,v_0))$
with $u^\sigma(0,x;u_0,v_0)=u_0(x)$ and $v^\sigma(0,x;u_0,v_0)=v_0(x)$. 

\item[(4)]
 We refer $u(a-bu^\sigma)$ as a weak (resp. regular, strong)  logistic source if
$0<\sigma<1$ (resp. $\sigma=1$, $\sigma>1$).  When $\sigma>1$, it can be proved that
$T_{\max}^\sigma(u_0,v_0)=\infty$ for any $u_0\in X^+$ and $v_0\in X^+_1$.
The global existence of classical solutions of \eqref{main-Eq1} with $0<\sigma<1$ will not be studied in this paper.
\end{itemize}
}
\end{rk}

Throughout the rest of the paper, for given  $\sigma > 0$  and $(u_0,v_0)\in X^+\times X^+_1$,  if $(u(t,x;u_0,v_0)$,  $v(t,x;u_0,v_0))$ is a   globally defined bounded solution of \eqref{main-Eq1}, we put
$$\|u\|_\infty=\sup_{t\ge 0}\|u(t,\cdot;u_0,v_0)\|_\infty.
$$

\subsection{Special Harnack inequality}

In this subsection,  we present two lemmas, which  will be used frequently in the proofs of the main results. The first lemma is a Harnack type inequality, which is a generalization \cite[Lemma 2.2]{HaHe}.
The constants $C$'s might depend on the parameters $a$, $b$, $\sigma$, $\tau$ in the equation, as well as the dimension $N$, without further notification.

\begin{lem}
\label{asymptotic-lm1}
For given  $\sigma > 0$, $u_0\in X^+$ and $v_0\in X^+_1$,
assume
  that $(u(t,x;u_0,v_0)$,  $v(t,x;u_0,v_0))$ is a   globally defined bounded solution of \eqref{main-Eq1}.
Then for any $s_0 \ge 0$,
$R>0$,  and $p\in (1,\infty)$,  there exists a constant $ C=C(s_0,R,p,\chi,\|v_0\|_{X_1}, \|u\|_\infty)$ such that if $t\ge 1$,
$s\in [0,s_0]$,
 and $|x-y|\le  R$, then
\begin{equation}
\label{asymptotic-eq1}
u(t,x;u_0,v_0)\le  C u^{\frac{1}{p}}(t+s, y;u_0,v_0)
\end{equation}
and
\begin{equation}
\label{asymptotic-eq2}
|\nabla u(t,x;u_0,v_0)|\le C u^{\frac{1}{p}}(t, y;u_0,v_0).
\end{equation}
The constant $C$ is uniform as $\chi\to 0$.
\end{lem}
The proof is similar to the one of \cite[Lemma 2.2]{HaHe}. For the reader's convenience,  we outline the  proof in the appendix.


The following estimates on $v$ can be obtained as a corollary of Lemma \ref{asymptotic-lm1}.

\begin{lem}
\label{asymptotic-lm2}
Let  $(u(t,x;u_0,v_0), v(t,x;u_0,v_0))$ be as in Lemma \ref{asymptotic-lm1}. For any ${\eps}>0$ and $p\in (1,\infty)$, there are $ T(\eps,\chi, {\|v_0\|_{X_1}})\geq 1$ and $C=C(\eps, p, \chi, \|v_0\|_{X_1},\|u\|_\infty)>0$ such that
\begin{equation}
\label{asymptotic-eq3}
|\nabla v(t,x;u_0,v_0)|\le {\eps}+ C u^{\frac{1}{p}}(t,x;u_0,v_0),\quad \forall\, t\ge T,\,\, x\in\R^N
\end{equation}
and
\begin{equation}
\label{asymptotic-eq4}
|\Delta v(t,x;u_0,v_0)|\le {\eps}+ C  u^{\frac{1}{p}}(t,x;u_0,v_0),\quad \forall\, t\ge T,\,\, x\in\R^N.
\end{equation}
The constants $T,C$ are uniform as $\chi\to 0$.
\end{lem}


\begin{proof}
For simplicity of notations, we drop $u_0,v_0$ from the notations of $u(t,x;u_0,v_0)$, $v(t,x;u_0,v_0)$.

First, note that, for any $t>t_1\ge 0$ and $x\in\R^N$, we have
\beq\label{asymptotic-proof-eq0}
\begin{aligned}
v(t,x)&=\left(\frac{\tau}{4\pi (t-t_1)}\right)^{N/2}\int_{\R^N} e^{-\frac{\tau|x-y|^2}{4(t-t_1)}}v(t_1,y)dy \\
&\quad
-\frac{1}{\tau}\int_{t_1}^t \left(\frac{\tau}{4\pi (t-s)}\right)^{N/2}\int_{\R^N} e^{-\frac{\tau|x-y|^2}{4(t-s)}}u(s,y) v(s,y)dy ds \\
&=\left(\frac{\tau}{4\pi (t-t_1)}\right)^{N/2}\int_{\R^N} e^{-\frac{\tau|x-y|^2}{4(t-t_1)}}v(t_1,y)dy \\
&\quad  -\frac{1}{\tau}\int_{t_1}^t \left(\frac{\tau}{\pi}\right)^{N/2}\int_{\R^N} e^{-\tau|z|^2}u(s,x+2\sqrt{t-s} \, z) v(s,x+2\sqrt{t-s} \, z)dz ds.
\end{aligned}
\eeq
Since $u,v\geq 0$, this with $t_1=0$ implies that $\|v\|_\infty\leq \|v_0\|_\infty$. By \cite[Proposition 3.3]{HSZ}, there is $C=C(\|v_0\|_{X_1},\|u\|_\infty)$  such that
\begin{equation}
\label{new-v-est-eq1}
v(t,x) + |\nabla v(t,x)|\le C\quad \forall\, t\ge 0,\,\, x\in\R^N.
\end{equation}
By Lemma \ref{derivative-boundedness-lm}, there is $C=C(\chi,\|v_0\|_{X_1},\|u\|_\infty)>0$ such that
\begin{equation*}
|\nabla v(t,x)|+|\Delta v(t,x)|\le C\quad \forall\, t\ge 1,\,\, x\in\R^N.
\end{equation*}
This yields that  for any $\eps>0$, there is $T(\eps,\chi, \|v_0\|_{X_1},\|u\|_\infty)$ such that 
\begin{equation}
\label{asymptotic-proof-eq1}
\begin{cases}
\left|\nabla \frac{\tau^{N/2}}{(4\pi T)^{N/2}}\int_{\R^N} e^{-\frac{\tau|x-y|^2}{4T}}v(t_1,y)dy\right|
= \left| \frac{\tau^{N/2}}{(4\pi T)^{N/2}}\int_{\R^N} e^{-\frac{\tau|x-y|^2}{4T}}\nabla v(t_1,y)dy\right|\cr
 \qquad\qquad\qquad\qquad\qquad\qquad\qquad\quad \le {\eps/3}\quad \forall\, t_1\ge 1,\,\, x\in\R^N,\cr\cr
\left|\Delta \frac{\t^{N/2}}{(4\pi T)^{N/2}}\int_{\R^N} e^{-\frac{\tau|x-y|^2}{4T}}v(t_1,y)dy\right|   = \left| \frac{\t^{N/2}}{(4\pi T)^{N/2}}\int_{\R^N} e^{-\frac{\tau|x-y|^2}{4T}}\Delta v(t_1,y)dy\right|\cr
\qquad\qquad\qquad\qquad\qquad\qquad\qquad\quad \le {\eps/3}\quad \forall\, t_1\ge 1,\,\,  x\in\R^N.
\end{cases}
\end{equation}
Moreover,  \eqref{new-v-est-eq1} implies  that there is 
$R=R(\eps, T,{\|v_0\|_{X_1}, \|u\|_\infty})>0$
  such that for all $t_1\le t\le t_1+T$ and $x\in\R^N$,
\begin{equation}
\label{asymptotic-proof-eq2}
\begin{cases}
\left|\int_{t_1}^t \left(\frac{\tau}{\pi}\right)^{N/2}\int_{|z|>R} e^{-\tau|z|^2}|\nabla_x u(s,x+2\sqrt{t-s} \, z)| v(s,x+2\sqrt{t-s} \, z)dz ds\right|\le \eps\tau/{3}, \cr
\cr
\left|\int_{t_1}^t \left(\frac{\tau}{\pi}\right)^{N/2}\int_{|z|>R} e^{-\tau|z|^2}u(s,x+2\sqrt{t-s} \, z) |\nabla_x v(s,x+2\sqrt{t-s} \, z)|dz ds\right|\le \eps\tau/{3},\cr
\cr
\left|\int_{t_1}^t \left(\frac{\tau}{\pi}\right)^{N/2}\frac{1}{\sqrt{t-s}}\int_{|z|>R} e^{-\tau|z|^2} |z\cdot \nabla_x  u(s,x+2\sqrt{t-s} \, z)| v(s,x+2\sqrt{t-s} \, z)dz ds\right|\le {\eps}/{3},\cr
\cr
\left|\int_{t_1}^t \left(\frac{\tau}{\pi}\right)^{N/2}\frac{1}{\sqrt{t-s}}\int_{|z|>R} e^{-\tau|z|^2}u(s,x+2\sqrt{t-s} \, z)  |z\cdot \nabla_x  v(s,x+2\sqrt{t-s} \, z)|dz ds\right|\le {\eps}/{3}.
\end{cases}
\end{equation}

Next, note that, for  any $t\ge 1+T$, by \eqref{asymptotic-proof-eq0} with $t_1=t-T$,   we have
\beq\label{asymptotic-proof-eq3}
\begin{aligned}
\nabla v(t,x)
&=\nabla \frac{\tau^{N/2}}{(4\pi T)^{N/2}}\int_{\R^N} e^{-\frac{\tau|x-y|^2}{4T}}v(t-T,y)dy\\
&\quad  -{\frac{1}{\tau}} \int_{t-T}^t \left(\frac{\tau}{\pi}\right)^{N/2}\int_{\R^N} e^{-\tau|z|^2}\nabla_x u(s,x+2\sqrt{t-s} \, z) v(s,x+2\sqrt{t-s} \, z)dz ds\\
&\quad - {\frac{1}{\tau}} \int_{t-T}^t \left(\frac{\tau}{\pi}\right)^{N/2}\int_{\R^N} e^{-\tau|z|^2} u(s,x+2\sqrt{t-s} \, z) \nabla_x v(s,x+2\sqrt{t-s} \, z)dz ds.
\end{aligned}
\eeq
By \eqref{new-v-est-eq1}--\eqref{asymptotic-proof-eq3}, we have
\[
\begin{aligned}
&|\nabla v(t,x)|\\
&\le {{\eps}}+{\frac{1}{\tau}} \int_{t-T}^t  \left(\frac{\tau}{\pi}\right)^{N/2}\int_{|z|\le R} e^{-\tau|z|^2}|\nabla_x u(s,x+2\sqrt{t-s} \, z)| v(s,x+2\sqrt{t-s} \, z)dz ds\\
&\quad +{\frac{1}{\tau}} \int_{t-T}^t \left(\frac{\tau}{\pi}\right)^{N/2}\int_{|z|\le R} e^{-\tau|z|^2}u(s,x+2\sqrt{t-s} \, z)| \nabla_x v(s,x+2\sqrt{t-s} \, z)|dz ds\\
&\le \eps+ {\frac{C}{\tau}} \int_{t-T}^t  \left(\frac{\tau}{\pi}\right)^{N/2}\int_{|z|\le R} e^{-\tau|z|^2}|\nabla_x u(s,x+2\sqrt{t-s} \, z)| dz ds\\
&\quad+{\frac{C}{\tau}} \int_{t-T}^t \left(\frac{\tau}{\pi}\right)^{N/2}\int_{|z|\le R} e^{-\tau|z|^2}u(s,x+2\sqrt{t-s} \, z)| dz ds.
\end{aligned}
\] 
This, together with $2\sqrt{t-s}|z|\le 2\sqrt T R$ for $|z|\le R$ and  Lemma \ref{asymptotic-lm1}, implies that there is {$C(\eps, p, \chi, T, \|v_0\|_\infty, \|u\|_\infty)>0$ } such that
$$
|\nabla v(t,x)|\le {\eps} +C u^{1/p}(t,x)\quad \forall\, t\ge 1+T.
$$
This proves \eqref{asymptotic-eq3}. 

Now, note that,  for any $t\ge 1+T$, by \eqref{asymptotic-proof-eq0} with $t_1=t-T$,  \eqref{asymptotic-proof-eq1},  and  \eqref{asymptotic-proof-eq2}, we have
\begin{align*}
|\Delta v(t,x)|&\le |\Delta  \left(\frac{\tau}{4\pi T}\right)^{N/2}\int_{\R^N} e^{-\frac{\tau|x-y|^2}{4 T}}v(t-T,y)dy |  \\
&\quad  +|\int_{t-T}^t \left(\frac{\tau}{\pi}\right)^{N/2}\frac{1}{\sqrt{t-s}}\int_{\R^N} e^{-\tau|z|^2} \left[ z\cdot \nabla_x u(s,x+2\sqrt{t-s}\, z) \right] v(s,x+2\sqrt{t-s}\, z)dz ds| \\
&\quad +|\int_{t-T}^t \left(\frac{\tau}{\pi}\right)^{N/2}\frac{1}{\sqrt{t-s}}\int_{\R^N} e^{-\tau|z|^2} u(s,x+2\sqrt{t-s}\, z)\left[ z\cdot \nabla_x v(s,x+2\sqrt{t-s}\,z )\right]dz ds|.
\end{align*}
This, together with \eqref{new-v-est-eq1}, \eqref{asymptotic-proof-eq1}, \eqref{asymptotic-proof-eq2} and  Lemma \ref{asymptotic-lm1}  implies that there is $C(\eps, p, \chi, T, \|v_0\|_\infty, \|u\|_\infty)>0$ such that
$$
|\Delta  v(t,x)|\le  {\eps} +C u^{1/p}(t,x)\quad \forall\, t\ge 1+T,
$$
which proves \eqref{asymptotic-eq4}.
\end{proof}

\subsection{Viscosity solutions}

In this subsection, we briefly recall viscosity solutions. We refer readers to \cite{user} for more details. This notion of solutions as well as the comparison principle will be one of the main tools we use in Section \ref{S.6.2}.

Consider the following parabolic type equation:
\beq\lb{2.1}
u_t+F(t,x,u,\nabla u,D^2u)=0.
\eeq
Let $\calS^N$ denote the set of $N \times N$ symmetric matrices with the spectral norm. We say that $F$ is uniformly elliptic, if there exists $\Lambda>0$ such that for any positive semi-definite matrix $P\in \calS^N$, and any $(t,x,u,p,X)\in [0,\infty)\times \bbR^N\times\bbR\times\bbR^N\times \calS^N$,
\[
\Lambda \text{ \rm Tr}(P)   \leq  F(t,x,u,p,X)-F(t,x,u,p,X+P).
\]
We assume $F$ to be continuous and uniformly elliptic.

Now we recall the definition of viscosity solutions. Let $\Omega\subseteq\bbR^N$ be open and $T>0$.
\begin{enumerate}
\item[(i)]
We say that an upper semicontinuous (resp. lower semicontinuous) function $u:(0,T)\times\Omega\to \mathbb{R}$ is a (viscosity) subsolution (resp. (viscosity) supersolution) to \eqref{2.1} if the following holds:
for any smooth function $\phi$ in $\Omega$ such that $u-\phi$ has a local maximum (resp. minimum) at $(t_0,x_0)\in (0,T)\times\Omega$, we have
\[
\partial_t u(t_0,x_0)+{F}(t_0,x_0, u(t_0,x_0),\nabla \phi(t_0,x_0),D^2 \phi(t_0,x_0))\leq 0
\]
\[
\left(\text{resp. }\partial_t u(t_0,x_0)+{F}(t_0,x_0, u(t_0,x_0),\nabla \phi(t_0,x_0),D^2 \phi(t_0,x_0))\geq 0\right).
\]
\item[(ii)]
We say that a continuous function $u:(0,T)\times \Omega\to \mathbb{R}$ is a (viscosity) solution to \eqref{2.1} if it is both a subsolution and a supersolution.
\end{enumerate}
It is easy to see that a classical solution is a viscosity solution.

\medskip

For the purpose of the paper, we take
\beq\lb{2.2}
F(t,x,u,\nabla u, D^2 u)=-\Delta u+f(t,x) \cdot\nabla u+g(t,x,u)
\eeq
where $f,g$ are uniformly continuous and bounded functions. It is easy to check that the operator satisfies the condition (3.14) in \cite{user}.
Consequently, we have the following comparison principle.
\begin{lem}
Let $T>0$ and let $F$ be given in \eqref{2.2}. Let   $ u^+ :(0,T)\times\bbR^N\to\bbR$ and $ u^-:(0,T)\times\bbR^N \to\bbR$ be, respectively, a supersolution and a subsolution to \eqref{2.1}. If $ u^+(0,\cdot)\geq u^-(0,\cdot)$ and $\inf_{t\in (0,T)}\liminf_{|x|\to\infty} (u^+(t,x)-u^-(t,x))\geq 0$, then
\[
u^+(t,x)\geq u^-(t,x)\quad\text{ for all }(t,x)\in (0,T)\times\bbR^N.
\]
\end{lem}

We refer readers to Sections 5D and 8 \cite{user} for the proof and for more general cases.

\subsection{Convergence to the constant equilibrium}

In this subsection, we  prove the convergence of globally defined bounded solutions with strictly positive initial data to the constant solution $(\frac{a}{b},0)$. 
The result implies that there are no other positive stationary solutions
$(u(x),v(x))$ of \eqref{main-Eq} with $\inf_{x\in\R^N} u(x)>0$  rather than $(\frac{a}b,0)$.

\begin{prop}[Convergence to constant equilibrium]
\label{asymptotic-behavior-prop}
Suppose  that $u_0\in X^+$ and $v_0\in X^+_1$, and  $(u(t,x;u_0,v_0), v(t,x;u_0,v_0))$ is a globally defined bounded solution \eqref{main-Eq}.
If $\inf_{x\in\R^N}u_0(x)>0$, then
$$
\lim_{t\to\infty}u(t,x;u_0,v_0)=\frac{a}{b},\quad \lim_{t\to\infty}v(t,x;u_0,v_0)=0\quad \text{uniformly in}\,\, x\in\R^N.
$$
\end{prop}


\begin{proof}
For any fixed $\eps \in (0, \frac{a}{|\chi|})$, by Lemma \ref{asymptotic-lm2}, there are  $T=T(\eps,{\|v_0\|_\infty})\geq 1$ and  $C= C(\eps, {\chi})>0$ such that for any $t\ge T,$ we have
\begin{equation}
\label{new-u-est-eq0}
     u_t \ge \Delta u -\chi\nabla u \cdot \nabla v -| \chi|  u(\eps + Cu^{1/2}) + au - bu^2.
\end{equation}
We first claim that
$$
\delta_T:=\inf_{x\in \R^N}u(T,x;u_0,v_0)>0.
$$

It follows from \cite[Proposition 2.1]{HSZ} that $u(t,\cdot;u_0,v_0)$ is uniformly continuous in $L^\infty$-norm as $t\to 0$.
Since $\inf_{x\in\R^N}u_0(x)>0$, there is $0<t_1<T$ such that
$$
\delta_1:=\inf_{x\in\R^N}u(t_1,x;u_0,v_0)>0.
$$
It follows from  {Lemma  \ref{derivative-boundedness-lm}}  that
$$
{M}:=\max\left\{\sup_{t\in [t_1,T],x\in\R^N} |\Delta v(t,x;u_0,v_0)|, \sup_{t_1\le t\le T,x\in\R^N}u(t,x;u_0,v_0)\right\}<\infty.
$$
We then have
\begin{equation}
\label{new-u-est-eq1}
\begin{cases}
u_t\ge \Delta u-\chi\nabla v\cdot\nabla u -|\chi| {M} u +a u- b{M}  u, \quad t_1<t<T,\,\, x\in\R^N\cr
u(t_1,x)\ge \delta_1,\quad x\in\R^N.
\end{cases}
\end{equation}
Note that $u$ is a classical solution, and $|\nabla u|$ and $|\nabla v|$ are uniformly bounded and continuous on $[t_1,T]\times \R^N$. Thus, by viewing $\nabla v(t,x;u_0,v_0)$ as a given function, the comparison principle  for parabolic equations (see e.g., \cite[Proposition 52.10]{quittner})  yields
$$
u(t,x;u_0,v_0)\ge e^{(-|\chi|{M} +a-b{M})(t-t_1)}\delta_1\quad \forall\, t_1\le t\le T,\,\, x\in\R^N.
$$
This implies that $\delta_T\ge e^{(-|\chi|{M}+a-b{M})(T-t_1) }\delta_1>0$. The claim then follows.

Next, let $\underline u(t)$ be the solution to the ODE
\begin{equation}
\label{ODE1}
\begin{cases}
\underline{u}_t= -|\chi| \underline{u}(\eps + C\underline{u}^{1/2})+\underline{u}(a-b\underline{u})\cr
\underline{u}(T)=\delta_T.
\end{cases}
\end{equation}
Note that  $ -|\chi| \eps + a >0.$  Hence
 $\underline u=0$ is an unstable solution of \eqref{ODE1}, and  there  exists $\eps_0 >0$ such that
$$
\underline{u}(t;\delta_T)\ge \eps_0, \qquad \forall \, t>T.
$$
Then, by the comparison principle, we have
\begin{equation}\label{lowerbound}
  u(t, x; u_0,v_0)\ge \underline{u}(t;\delta_T)\ge\eps_0,\quad \forall\, t\ge T.
\end{equation}
This shows that  $\tau v_t\leq \Delta v-\eps_0v$ for all $t\geq T$.
Thus, by comparing $v$ with $\bar{v}(t)$ which solves the ODE $\tau \bar{v}_t= -\eps_0 \bar{v}$ with $\bar{v}(T)=\|v(T,\cdot;u_0,v_0)\|_{\infty}\leq \|v_0\|_\infty$,
we obtain
\begin{equation}
\label{new-v-est-eq0}
    {0\le v(t, x) \le \bar{v}(t) \le \|v_0\|_{\infty} e^{-\frac{\eps_0}{\tau}(t-T)}, \qquad \forall t\ge T.}
\end{equation}
In particular, we proved
\begin{equation}
\label{v-to-0-eq}\lim_{t\to\infty}v(t,x;u_0,v_0)=0\quad \text{uniformly in}\,\, x\in\R^N.
\end{equation}

\smallskip

Now, we prove  that
\begin{equation}
\label{u-to-a-over-b-eq}
\lim_{t\to\infty}u(t,x;u_0,v_0)=a/b \quad \text{uniformly in} \,\,  x\in \R^N.
\end{equation}
Suppose for contradiction that
 there exist $\delta_0>0$, a sequence $t_n\to\infty$,  and  $x_n\in \R^N$ such that
\begin{equation}\label{seq2}
    |u(t_n,x_n; u_0,v_0)-a/b| > \delta_0.
\end{equation}
Define
$$
u_n(t,x): = u(t+t_n, x+x_n;u_0, v_0)\quad \text{and} \quad v_n(t,x) = v(t+t_n, x+x_n;u_0, v_0), \quad t\ge -t_n.
$$
By \eqref{v-to-0-eq},
\begin{equation}
\label{v-to-0-eq1}
\lim_{n\to\infty} v_n(t,x)=0\quad \text{locally uniformly in}\, t\in\R,\, \text{and uniformly in}\, x\in\R^N.
\end{equation}
By Lemma \ref{derivative-boundedness-lm}, for any bounded subset $I\subset \R$, there is $n_0\ge 1$ such that  $\{w_n(t,x)\}_{n\ge n_0}$ is   uniformly bounded and equi-continuous on $I\times \R^N$, where 
$w_n(t,x)=u(t+t_n,x+x_n)$, $\partial_{x_i}u(t+t_n,x+x_n;u_0,v_0)$, $\partial_t u(t+t_n,x+x_n;u_0,v_0)$, $\partial^2_{x_i x_j}u(t+t_n,x+x_n;u_0,v_0)$, $v(t+t_n,x+x_n;u_0,v_0)$, $\partial_{x_i}v(t+t_n,x+x_n;u_0,v_0)$,
$\partial_t v(t+t_n,x+x_n$; $u_0,v_0)$ or $\partial^2_{x_ix_j}v(t+t_n,x+x_n;u_0,v_0)$ for $1\le i,j\le N$. Then, by   Arzel\`{a}-Ascoli theorem, and \eqref{v-to-0-eq1},  there is 
a subsequence  $(u_{n_j}, v_{n_j})$ and a smooth function $\tilde u(t,x)$  such that
\beq\lb{new-conv-eq1}
\lim_{j\to\infty} u_{n_j}(t,x)=\tilde u(t,x),\quad \lim_{j\to \infty}v_{n_j}(t,x)=0\quad \forall\, t\in\R,\,\, x\in\R^N.
\eeq
Due to the uniform regularity, $\partial_t u_{n_j}$ converges and the limit is equal to $\partial_t\tilde u$. Similarly, and after applying a diagonal argument, we can assume that $Du_{n_j}$, $D^2u_{n_j}$, $Dv_{n_j}$ and $D^2v_{n_j}$ converge to $D\tilde u$, $D^2\tilde u$, $0$ and $0$, respectively.
Therefore,  $\tilde u$ satisfies
\begin{equation*}
    \tilde{u}_t = \Delta \tilde{u} + a \tilde{u} - b\tilde{u}^2, \qquad t\in \R, \, x\in\R^N.
\end{equation*} 
By  \eqref{lowerbound} and the global boundedness  of
$(u(t,x;u_0,v_0),v(t,x;u_0,v_0))$, there exists $K> \eps_0$ such that
$$0<\eps_0\le \tilde{u}(t,x)\le K, \quad \forall \, t\in \R, \,\, x\in\R^N.$$
Set $\underline{u}_{0}:=\inf_{(t,x)\in\R^{N+1}}\tilde{u}$ and $\overline{u}_{0}:=\sup_{(t,x)\in\R^{N+1}}\tilde{u}.$
 For every $t_{0}\in \R,$ let $\underline{u}(\cdot;t_{0})$ and $\overline{u}(\cdot;t_{0})$ be the solutions of
 $$
 \begin{cases}
 \frac{d}{dt}\overline{u}=\overline{u}(a-b\overline{u}),\ \ t>t_0\\
 \overline{u}(t_{0};t_{0})=\overline{u}_{0}
 \end{cases}
 $$
 and
 $$
 \begin{cases}
 \frac{d}{dt}\underline{u}=\underline{u}(a-b\underline{u}),\ \ t>t_0\\
 \underline{u}(t_{0};t_{0})=\underline{u}_{0},
 \end{cases}
 $$
respectively. Then for every $t_0\in\R$,
\beq\label{Q3}
 \lim_{t\to \infty}\overline{u}(t;t_{0})=\lim_{t\to\infty} \underline{u}(t;t_{0})=\frac{a}b.
\eeq
 Since $0<\underline{u}_{0}\leq \tilde{u}(t, x)\leq \overline{u}_{0}$ for every $(t, x)\in\R^{N+1},$ we obtain that
 \begin{equation*}
 \underline{u}(t-t_0;0)=\underline{u}(t;t_{0})\leq \tilde{u}(t, x)\leq\overline{u}(t;t_{0})=\overline{u}(t-t_0;0)\ \ \forall\ x\in\R^N,\ t\geq t_{0}.
 \end{equation*}
 Taking limit as $t_0\to -\infty$ on both sides and using \eqref{Q3} imply
$\tilde u(t,x) \equiv a/b,$
which contradicts with \eqref{seq2}.
Therefore, \eqref{u-to-a-over-b-eq} holds, and
Proposition \ref{asymptotic-behavior-prop} is proved.
\end{proof}

\section{Lower bound of spreading speeds}

In this section, we study the lower bounds of spreading speeds and prove Theorem \ref{speed-lower-bound-thm}.   Throughout this section, we fix $(u_0,v_0)\in X^+\times X_1^+$ and  suppose  that $(u(t,x;u_0,v_0), v(t,x;u_0,v_0))$ is a globally defined bounded and non-negative solution of \eqref{main-Eq}. {We will often omit the dependence of constants on $a,b,\tau$ and $N$.}

\subsection{Proof of Theorem \ref{speed-lower-bound-thm}(1)}


The proof follows the line of those in \cite{ShXu}, which studied \eqref{Eq3}. For us, nontrivial modifications are needed due to the difference of the system.

Let us fix any
\beq\lb{4.1}
0<{c'}<2\sqrt a\quad\text{and}\quad 0<\delta_0<\min\left\{1, 2\sqrt a  -c', (a^{1/4}(2N)^{-1})^N\right\}.
\eeq
Since $2\sqrt{a}-\delta_0>c'$, for any $R>0$,
$$
\inf_{|x|\le {c'}t}u(t,x;u_0,v_0)\ge \inf_{-2\sqrt{a}+\delta_0\le  c\le 2\sqrt{a}-\delta_0,\xi\in\mathbb{S}^{N-1}}\left(\inf_{|x|\le R}u(t,x+ct\xi;u_0,v_0)\right).
$$
To prove Theorem \ref{speed-lower-bound-thm}(1) (i.e. \eqref{lower-bound-eqq1}), it then suffices to prove   that there is $R_0>0$ such that
\begin{equation}
\label{lower-bound-eqq1-0}
\liminf_{t\to\infty} \inf_{-2\sqrt{a}+\delta_0\le  c\le 2\sqrt{a}-\delta_0,\xi\in \mathbb{S}^{N-1}}\left(\inf_{|x|\le R_0}u(t,x+ct\xi;u_0,v_0)\right)>0.
\end{equation}

For  any $\xi\in \mathbb{S}^{N-1}$ and  $c\in\R$, let $\tilde u(t, x;\xi, c):=u(t,x+ct\xi;u_0,v_0)$ and $\tilde v(t, x;\xi, c):=v(t, x+ct\xi;u_0,v_0)$. 
Then
$(\tilde u(t, x;\xi, c),\tilde v(t,x;\xi, c))$ satisfies
\begin{equation}\label{tilde-u-v-eq}
\begin{cases}
\tilde u_{t}=\Delta\tilde u+c\xi\cdot\nabla \tilde u- \chi \nabla\cdot (\tilde u\nabla \tilde v) + \tilde u(a-b\tilde u)\quad & x\in\R^N, \\
\tau \tilde v_{t}=\Delta \tilde v+\tau c\xi\cdot\nabla\tilde v - \tilde u \tilde v,\quad & x\in\R^N.
\end{cases}
\end{equation}
To prove \eqref{lower-bound-eqq1-0}, it  is equivalent to prove
 that,  for any fixed $0<c'<2\sqrt a$ and $0<\delta_0<2\sqrt a-c'$, there is ${R_0}>0$ such that
\begin{equation}
\label{lower-bound-eqq1-1}
\liminf_{t\to\infty} \inf_{-2\sqrt{a}+\delta_0\le  c\le 2\sqrt{a}-\delta_0,\xi\in \mathbb{S}^{N-1}}\left(\inf_{|x|\le {R_0}} \tilde u(t,x;\xi,c)\right)>0.
\end{equation}
To do so, we first prove some lemmas.

In the following,
for any $R>0$, let
$$
B_R=\{x\in\R^N\,|\, |x|<R\}.
$$
For any
$0<\eps<1$, let
\begin{equation}\lb{4.4}
T_\eps:={1}/{\eps},\quad R_\eps:=  {2}{\eps^{-(1/2 + 1/N)}}\geq 2.
\end{equation}
Our first lemma  shows that if supremum of $\tilde u$ is small for $x$ in a ball of radius $2R$ and $t$ within a small time interval,  then both $\nabla \tilde v$ and $\Delta \tilde v$ are  small  for $x$ in a ball of radius $R$ and $t$ within a smaller time interval.

\begin{lem}\label{asymptotic2-lem2}
There exists $\tilde M = \tilde M(\chi, \|v_0\|_\infty, \|u\|_\infty)>0$ 
such that
for any  $0<\eps<1$, $\xi\in \mathbb{S}^{N-1}$,  $c\in\R$,   and  $1<t_1<t_1+T_\eps<t_2 \le \infty$, if
\begin{equation}\label{ulowerbound-est}
   \sup_{x\in B_{2R_\eps}}  \tilde u(t, x;\xi,c)\le \eps,\quad \forall\, t_1\le t< t_2,
\end{equation}
then
\begin{equation}
\label{v-dv-est-eq1}
\sup_{x\in B_{R_\eps}}|\nabla \tilde v(t, x;\xi,c)|\leq \tilde M\eps^{1/2}\quad   \forall\, t_1+T_\eps \le t <t_2,
\end{equation}
and
\begin{equation}
\label{v-dv-est-eq2}
 \sup_{x\in B_{R_\eps}}|\Delta \tilde v(t, x;\xi,c)|\leq \tilde M\eps^{1/4}\quad  \forall\, t_1+T_\eps\le t <t_2.
\end{equation}
\end{lem}

\begin{proof}
First, we prove \eqref{v-dv-est-eq1}.  If no confusion occurs, we may drop $\xi,c$ in the notations of $\tilde u(t,x;\xi,c)$ and $\tilde v(t,x;\xi,c)$.  By the definition of   $\tilde v(t,x) $, we have for $t>t_1$,
\begin{align*}
        \tilde v(t,x) &= v(t, x+ct\xi)=\left(\frac{\tau}{4\pi (t-t_1)}\right)^{N/2}\int_{\R^N} e^{-\frac{\tau|x+ct\xi-y|^2}{4(t-t_1)}}v(t_1,y)dy\\
        &\qquad -\frac{1}{\tau}\int_{t_1}^t \left(\frac{\tau}{4\pi (t-s)}\right)^{N/2}\int_{\R^N} e^{-\frac{\tau|x+ct\xi-y|^2}{4(t-s)}}u(s,y) v(s,y)dy ds.
\end{align*}
Writing $z=\frac{y-x-ct\xi}{2\sqrt{t-t_1}}$, this implies that
\begin{align}
\label{nabla-v-eq0}
       \nabla \tilde v(t,x)& =-\left(\frac{\tau}{\pi}\right)^{N/2}\int_{\R^N} \frac{z\tau}{\sqrt{t-t_1}}e^{-\tau|z|^2}\tilde v(t_1,x+2z\sqrt{t-t_1})dz\nonumber \\
        &\quad -\int_{t_1}^t \left(\frac{\tau}{\pi}\right)^{N/2}\int_{|z| \le \frac{R_\eps}{2\sqrt{T_\eps}}} \frac{{z}}{\sqrt{t-s}}e^{-\tau|z|^2} \tilde u(s,x+ 2z\sqrt{t-s} ) \tilde v(s,x+2z\sqrt{t-s} )dz ds\nonumber \\
        &\quad -\int_{t_1}^t \left(\frac{\tau}{\pi}\right)^{N/2}\int_{|z| > \frac{R_\eps}{2\sqrt{T_\eps}}} \frac{{z}}{\sqrt{t-s}}e^{-\tau|z|^2} \tilde u(s,x+ 2z\sqrt{t-s} ) \tilde v(s,x+2z\sqrt{t-s})dz ds.
    \end{align}
In the following, we estimate each term in \eqref{nabla-v-eq0}.

Note that  $(\frac{\tau}{\pi})^{N/2}\int_{\R^N} |z|e^{-\tau|z|^2}dz\leq C_1\tau^{-1/2}$ for some $C_1>0$, which is because
\begin{align*}
    \left(\frac{\tau}{\pi}\right)^{N/2}\int_{\R^N} |z|e^{-\tau|z|^2}dz &= \left(\frac{\tau}{\pi}\right)^{N/2}\int_0^{\infty}\int_{\partial B(0, r)} re^{-\tau r^2}dS(z) dr= \left(\frac{\tau}{\pi}\right)^{N/2}\int_0^{\infty} \frac{2\pi^{\frac{N}{2}}}{\Gamma(\frac{N}{2})}r^Ne^{-\tau r^2} dr\\
    &=  \left(\frac{\tau}{\pi}\right)^{N/2}\frac{\pi^{\frac{N}{2}}}{\tau^{N/2 +1/2}\,\Gamma(\frac{N}{2})}\int_0^{\infty}R^{\frac{N-1}{2}}e^{-R} dR = \frac{\Gamma(\frac{N+1}{2})}{\tau^{1/2}\Gamma(\frac{N}{2})}.
\end{align*}
Thus, for  the  first integral in \eqref{nabla-v-eq0} we have for all $t\ge  t_1+T_\eps$ and $x\in\R^N$,
\beq\lb{nabla-v-eq1}
\begin{aligned}
\left|\left(\frac{\tau}{\pi}\right)^{N/2}\int_{\R^N} \frac{z \tau}{\sqrt{t-t_1}}e^{-\tau|z|^2}\tilde v(t_1,x+2z\sqrt{t-t_1})dz\right|  \le 
C_1\tau^{\frac{1}{2}}\eps^{\frac{1}{2}}\|v_0\|_{\infty}.
\end{aligned}
\eeq
For  the second integral in \eqref{nabla-v-eq0}, using \eqref{ulowerbound-est}, we have for any $ t_1+T_\eps\le t\le \min\{t_1 + 2T_\eps, t_2\}$,
\begin{align}
\label{nabla-v-eq2}
    &\int_{t_1}^t \left(\frac{\tau}{\pi}\right)^{N/2}\int_{|z|\le \frac{R_\eps}{2\sqrt{T_\eps}}} \frac{|z|}{\sqrt{t-s}}e^{-\tau|z|^2} \tilde u(s,x+ 2z\sqrt{t-s}) \tilde v(s,x+2z\sqrt{t-s})drds \nonumber\\
    &\qquad \le  C_1\tau^{-{\frac{1}{2}}}\|v_0\|_\infty \Big[\sup_{\substack{s\in (t_1,t),\nonumber\\
    y\in B(x,R_\eps)}}\tilde{u}(s,y)\Big] \int_{t_1}^t \frac {1}{\sqrt{t-s}}ds  \\
&\qquad \le  2\sqrt 2 C_1 \tau^{-\frac{1}{2}}\eps^{\frac12} \|v_0\|_\infty\|u\|
    \quad \forall\, |x|\le R_\eps.
\end{align}
Then, choose a $C_2=C_2(N)>0$ such that $e^{-|\tau^{1/2}z|^2} \le C_2|\tau^{1/2}z|^{-2N-1}$ for $|z|\ge \frac{R_\eps}{2\sqrt T_\eps}=\eps^{-1/N}$. With this and similarly as before, we obtain for $t_1 +T_\eps < t<\min\{t_1 + 2T_\eps, t_2\}$:
\begin{align}
\label{nabla-v-eq3}
    &\int_{t_1}^t \left(\frac{\tau}{\pi}\right)^{N/2}\int_{|z|> \frac{R_\eps}{2\sqrt{T_\eps}}} \frac{{|z|}}{\sqrt{t-s}}e^{-\tau|z|^2} \tilde u(s,x+ 2\sqrt{t-s} \, z)\tilde v(s,x+2\sqrt{t-s} \, z)dz ds \nonumber \\
    &\qquad\le C_2\left(\frac{\tau}{\pi}\right)^{N/2} \|u\|_{\infty}\|v_0\|_\infty \int_{t_1}^t \int_{|z|> \frac{R_\eps}{2\sqrt{T_\eps}}} \frac{\tau^{-N-\frac12}|z|^{-2N}}{\sqrt{t-s}}dzds\nonumber \\
 &\qquad \le C_2\tau^\frac{-1-N}{2}\|u\|_{\infty}\|v_0\|_\infty \sqrt{T_\eps}\left({2\sqrt{T_\eps}}/{R_\eps}\right)^N\leq C\eps^\frac12\tau^\frac{-1-N}{2}\|u\|_{\infty}\|v_0\|_\infty .
\end{align}

Combining \eqref{nabla-v-eq0}--\eqref{nabla-v-eq3},  there is  $\tilde M =\tilde M(N, \tau, \|v_0\|_\infty,\|u\|_\infty)>0$ such that
\begin{equation*}
    |\nabla \tilde v(t,x)| \le \tilde M\eps^{1/2}, \qquad \forall\,  t_1 + T_\eps\le t \le \min\{t_1+ 2T_\eps, t_2\},\,\, |x|\le R_\eps.
\end{equation*}
Identical argument with $t_3\in [t_1,t_2)$ in place of $t_1$ yields the same estimate  for $t_3+T_\eps < t\le \min\{t_3 + 2T_\eps, t_2\}$. Therefore, we conclude that there is $\tilde M = \tilde M(N, \tau, \|v_0\|_\infty, \|u\|_\infty)>0$ such that
\begin{equation*}
    |\nabla \tilde v(t,x)| \le \tilde M\eps^{1/2}, \qquad \forall\,  t_1 + T_\eps\le t \le t_2,\,\, |x|\leq R_\eps.
\end{equation*}

Next, we prove \eqref{v-dv-est-eq2}.
By \eqref{nabla-v-eq0},  we have
\begin{align}
\label{delta-v-eq0}
           & \Delta \tilde v(t,x) =-\left(\frac{\tau}{\pi}\right)^{N/2}\int_{\R^N} \frac{z\tau}{\sqrt{t-t_1}}e^{-\tau|z|^2}\cdot \nabla \tilde v(t_1,x+2z\sqrt{t-t_1})dz \nonumber\\
        &\qquad -\int_{t_1}^t \left(\frac{\tau}{\pi}\right)^{N/2}\int_{|z| \le \frac{R_\eps}{2\sqrt{T_\eps}}} \frac{z}{\sqrt{t-s}}e^{-\tau|z|^2}\cdot \nabla  \tilde u(s,x+ 2\sqrt{t-s} \, z) \tilde v(s,x+2\sqrt{t-s} \, z)dz ds\nonumber \\
 &\qquad -\int_{t_1}^t \left(\frac{\tau}{\pi}\right)^{N/2}\int_{|z| \le \frac{R_\eps}{2\sqrt{T_\eps}}} \frac{z}{\sqrt{t-s}}e^{-\tau|z|^2}   \tilde u(s,x+ 2\sqrt{t-s} \, z)\cdot \nabla  \tilde v(s,x+2\sqrt{t-s} \, z)dz ds\nonumber \\
        &\qquad -\int_{t_1}^t \left(\frac{\tau}{\pi}\right)^{N/2}\int_{|z| > \frac{R_\eps}{2\sqrt{T_\eps}}} \frac{z}{\sqrt{t-s}}e^{-\tau|z|^2} \cdot\nabla \tilde u(s,x+ 2\sqrt{t-s} \, z) \tilde v(s,x+2\sqrt{t-s} \, z)dz ds\nonumber \\
&\qquad -\int_{t_1}^t \left(\frac{\tau}{\pi}\right)^{N/2}\int_{|z| > \frac{R_\eps}{2\sqrt{T_\eps}}} \frac{z}{\sqrt{t-s}}e^{-\tau|z|^2} \tilde u(s,x+ 2\sqrt{t-s} \, z) \cdot\nabla \tilde v(s,x+2\sqrt{t-s} \, z)dz ds
    \end{align}
for all $t\ge t_1$ and $x\in\R^N$.
In the following, we estimate each term in \eqref{delta-v-eq0}.

First, since \eqref{ulowerbound-est} and $t_1\geq 1$, by Lemma \ref{asymptotic-lm1} (with $s_0,R,p$ in the lemma being $0,1,\frac43$), there is
$C_0=C_0({\chi},\|v_0\|_\infty,\|u\|_\infty)>0$
independent of $\eps$ such that
\begin{equation}
\label{nabla-u-eq0}
|\nabla \tilde u(t,x)|\le C_0\, \tilde u(t,x)^{\frac{3}{4}}\le C_0\, \eps^{\frac{3}{4}}\quad  \forall\,  t_1\le t<t_2,\,\, |x|\le 2 R_\eps.
\end{equation}
{Notice that $\nabla \tilde v(t,x)$ is uniformly bounded for $t\geq 1$ by the classical parabolic regularity theory.} So, similarly as done in \eqref{nabla-v-eq1}, there is $C=C(\chi, \|v_0\|_\infty,\|u\|)>0$  such that
\begin{equation}
\label{delta-v-eq1}
\left|\left(\frac{\tau}{\pi}\right)^{N/2}\int_{\R^N} \frac{{z}}{\sqrt{t-t_1}}e^{-\tau|z|^2}\cdot \nabla \tilde v(t_1,x+2z\sqrt{t-t_1})dz\right|\le C{\eps^{\frac{1}{2}}}\quad \forall\, t\ge t_1+T_\eps,\,\, x\in\R^N.
\end{equation}
By \eqref{nabla-u-eq0} and the arguments of \eqref{nabla-v-eq2}, there is $C=C(\chi,\|v_0\|_\infty,\|u\|)>0$ such that for $t_1+T_\eps\le t\le \min\{t_2,t_1+2T_\eps\}$ and $|x|\le R_\eps$,
\begin{equation*}
 \left|\int_{t_1}^t \left(\frac{\tau}{\pi}\right)^{N/2}\int_{|z| \le \frac{R_\eps}{2\sqrt{T_\eps}}} \frac{z}{\sqrt{t-s}}e^{-\tau|z|^2}\cdot \nabla  \tilde u(s,x+ 2\sqrt{t-s} \, z) \tilde v(s,x+2\sqrt{t-s} \, z)dz ds\right|\le C \eps^{\frac{1}{4}}
\end{equation*}
and
\begin{equation*}
\left|\int_{t_1}^t \left(\frac{\tau}{\pi}\right)^{N/2}\int_{|z| \le \frac{R_\eps}{2\sqrt{T_\eps}}} \frac{{z}}{\sqrt{t-s}}e^{-\tau|z|^2}\tilde u(s,x+ 2\sqrt{t-s} \, z)\cdot\nabla  \tilde v(s,x+2\sqrt{t-s} \, z)dz ds\right|\le C \eps^{\frac{1}{2}}.
\end{equation*}
By \eqref{nabla-u-eq0} and the arguments of \eqref{nabla-v-eq3}, there is $C=C(\chi,\|v_0\|_\infty,\|u\|)>0$ such that for $t_1+T_\eps\le t\le \min\{t_2, t_1+2T_\eps\}$ and $|x|\leq R_{\eps}$,
\begin{equation*}
\left|\int_{t_1}^t \left(\frac{\tau}{\pi}\right)^{N/2}\int_{|z| > \frac{R_\eps}{2\sqrt{T_\eps}}} \frac{{z}}{\sqrt{t-s}}e^{-\tau|z|^2} \cdot\nabla \tilde u(s,x+ 2\sqrt{t-s} \, z) \tilde v(s,x+2\sqrt{t-s} \, z)dz ds\right|\le  C \eps^{\frac{1}{4}}
\end{equation*}
and
\begin{equation}
\label{delta-v-eq3-2}
\left|\int_{t_1}^t \left(\frac{\tau}{\pi}\right)^{N/2}\int_{|z| > \frac{R_\eps}{2\sqrt{T_\eps}}} \frac{z}{\sqrt{t-s}}e^{-\tau|z|^2}  \tilde u(s,x+ 2\sqrt{t-s} \, z)\cdot\nabla  \tilde v(s,x+2\sqrt{t-s} \, z)dz ds\right|\le C \eps^{\frac{1}{2}}.
\end{equation}
By \eqref{delta-v-eq0}
and \eqref{delta-v-eq1}--\eqref{delta-v-eq3-2},
there is  $\tilde M = \tilde M( \chi,  \|v_0\|_\infty,\|u\|_\infty)>0$ such that
\begin{equation*}
    |\Delta  \tilde v(t,x)| \le \tilde M\eps^{1/4}, \qquad \forall\,  t_1 + T_\eps\le t \le \min\{t_1+ 2T_\eps, t_2\},\,\, |x|\le R_\eps.
\end{equation*}

After replacing $t_1$ by any $t_3\in [t_1,t_2)$, we can get the estimate for $t_3+T_\eps < t\le \min\{t_3 + 2T_\eps, t_2\}$. Therefore, we conclude that there exists $\tilde M = \tilde M(\chi,  \|v_0\|_\infty,\|u\|_\infty)>0$ such that
\begin{equation*}
    |\Delta \tilde v(t,x)| \le \tilde M\eps^{1/4}, \qquad \forall\,  t_1 + T_\eps\le t \le t_2,\,\, |x|\leq R_\eps.
\end{equation*}
\end{proof}


Our second  lemma shows that if we can bound $\tilde u$ below at some given time $t_0$ in a ball $B_{2R},$ then we can bound it below up to some time $t_1>t_0$ in this ball.

\begin{lem}\label{lem-3} 
Fix $0<\eps< 1$.
For any $\eta >0$, there is $\delta_{\eta}>0$ such that for any $\xi\in \mathbb{S}^{N-1}$, any $ c\in [-2\sqrt{a}, 2\sqrt{a}]$,  and any  $t_0\ge 2$,  if
$$
\sup_{x\in B_{2R_\eps}} \tilde u(t_0, x;\xi,c)\ge \eta,
$$
then
$$
\inf_{x\in B_{2R_\eps}} \tilde u(t, x;\xi,c)\ge \delta_{\eta},\quad \forall\, t_0\le t\le t_0+T_\eps+1.
$$
\end{lem}

\begin{proof} 
Suppose for contradiction that there exist $\eta_0 >0$,
 $\xi_n \in \mathbb{S}^{N-1}$, $-2\sqrt{a}\leq c_n\leq 2\sqrt{a}$, $t_{0n}\ge 2$,  $x_n, y_n\in B_{2R}$, and   $t_n\in [t_{0n}, t_{0n}+T_\eps+1]$ such that
\begin{equation}\label{4.11}
\lim_{n\to\infty} \tilde u(t_{0n}, x_n;\xi_{n}, c_n)\geq \eta_0,
\end{equation}
and
\begin{equation}\label{4.12}
\lim_{n\to\infty} \tilde u(t_{n}, y_{n};\xi_{n}, c_n)=0.
\end{equation}
Let $\tilde u_n(t, x)=\tilde u(t+t_{0n}-1, x+x_{n};\xi_{n},c_n)$, $\tilde v_n(t, x)=\tilde v( t+t_{0n}-1, x+x_{n};\xi_{n}, c_n)$.
Since $x_n,  y_n$ and $t_n - t_{0n} +1 $ are bounded sequences, without loss of generality, we may assume that
$$
\xi_{n}\to \xi^{*},\quad c_n\to c^*, \quad x_n  \to x^*, \quad y_n\to y^*,  \quad t_n - t_{0n} +1\to t^*\ge 1\qquad \text{ as } n\to \infty,
$$
for some $\xi^{*}\in \mathbb{S}^{N-1}$, $-2\sqrt{a}\leq c^{*}\leq 2\sqrt{a}$, $x^*,y^*\in B_{2R}$ and $t^*\in [1,T_\eps+2]$.
By Lemma \ref{derivative-boundedness-lm} and   Arzel\`a-Ascoli theorem, after passing to a subsequence, we can assume that there is $(u^*(t,x),v^*(t,x))$ such that
\begin{equation*}
(\tilde u_n(t, x),\tilde v_n(t, x))\to (u^*(t, x),v^*(t, x))\quad\text{ as $n\to\infty$}
\end{equation*}
locally uniformly in $(t,x)\in [0,\infty)\times \R^{N}$, and     $(u^*, v^*)$ is a solution of \eqref{tilde-u-v-eq} with $\xi$ and $c$ being replaced by $\xi^{*}$ and $c^{*}$ for $t\ge 0$. By \eqref{4.11},
$$u^{*}( 1, 0)\geq \eta_{0}.$$
It then follows from Lemma \ref{asymptotic-lm1}
that $u^{*}(t^*, \cdot)>0$ in $B_{4R}$, which contradicts with $u^{*} (t^{*}, y^*-x^{*})=0$ by \eqref{4.12}.
The lemma is thus proved.
\end{proof}

To proceed, since $(2\sqrt{a}-\delta_0)^2+\delta_0\sqrt{a}<4a$}, take $\bar a\in (0,a)$   such that
\begin{equation}
\label{bar-r-eq}
4\bar a -c^2\ge \delta_0\sqrt{a} \quad \text{for any}\quad c\in [ -2\sqrt{a}+\delta_0 ,  2\sqrt {a}-\delta_0].
\end{equation}
For $\xi\in\mathbb{S}^{N-1}$, let $\lambda (c,\bar a)$ be the principal eigenvalue of
\begin{equation*}
\begin{cases}
\Delta\phi+ c\xi\cdot\nabla\phi+\bar a \phi=\lambda \phi,\quad &x\in B_{R_0}\cr
\phi(x)=0, \quad  & x\in \partial B_{R_0},
\end{cases}
\end{equation*}
and $\phi(x;\xi,c,\bar a)$ be the corresponding positive eigenfunction with $\| \phi(\cdot;\xi,c,\bar a)\|_\infty=1$. By symmetry, $\lambda(c,\bar{a})$ is independent of $\xi$.
We claim that there are $R_0,\lambda_0>0$ such that
\begin{equation}
\label{lambda-lambda-eq}
\lambda(c,\bar a)\ge \lambda_0>0\quad \forall\, c\in [-2\sqrt a+\delta_0,2\sqrt a-\delta_0],\,\, \xi\in \mathbb{S}^{N-1}.
\end{equation}

The proof for the claim is easy. Indeed, let
\begin{equation}
\label{L-R-eq}
l_0=l_0(\delta_0):={2\pi\sqrt{N}}{(\delta_0\sqrt{a})^{-\frac{1}{2}}},  \quad R_0:=\sqrt N l_0,
\end{equation}
and set
$$
D_{l_0}:=\{x\in\R^N \,\ | \,\ |x_{i}|<l_0\,\ {\rm for}\,\ i=1,2,\cdot\cdot\cdot N\}.
$$
Then it is direct to check that for $\xi\in\mathbb{S}^{N-1}$, we have
\begin{equation*}
\tilde\lambda(c,\bar a):=\bar a-\frac{c^2}4-\frac{N\pi^2}{4l_0^2}
\quad\text{and}\quad
\tilde\phi(x;\xi, c,\bar a):=e^{-\frac{c}{2}\xi\cdot x}\prod_{i=1}^{N}{ \cos (\frac{\pi}{2l_0}x_{i}})
\end{equation*}
satisfy
\begin{equation*}
\begin{cases}
\Delta\tilde\phi+ c\xi\cdot\nabla\tilde\phi+\bar a \tilde\phi=\tilde \lambda \tilde\phi,\quad & x\in D_{l_0}\cr
\tilde\phi(x)=0, \quad & x\in \partial D_{l_0},
\end{cases}
\end{equation*}
and
\begin{equation*}
\lambda_0:=\tilde\lambda(2\sqrt{a}-{\delta_0},\bar a)=\min_{-2\sqrt{a}+{\delta_0}\leq c\leq 2\sqrt{a}-{\delta_0}}\tilde\lambda(c,\bar a)>0.
\end{equation*}
Since $D_{l_0}\subseteq B_{R_0}$, the domain monotonicity for the Dirichlet  principal eigenvalues yields the claim \eqref{lambda-lambda-eq}.

\smallskip

The next lemma shows that if  the supremum of $\tilde u$ is small on some interval $(t_1,t_2)\subset (2,\infty)$,  we can obtain a lower bound  for $\tilde u$ on that interval.

\begin{lem}\label{lem-4}
 {Recall the notations of \eqref{4.4} and $\delta_0$ from \eqref{4.1}.} There is  $\eps_1>0$ such that for  any $0< \eta\le \eps_1$, there is $\delta_\eta>0$ such that
for  any $\xi\in \mathbb{S}^{N-1}$, any $c\in [-2\sqrt{a}+\delta_0, 2\sqrt{a}-\delta_0]$, and  any $t_1,t_2$ with $2\le t_1< t_2\le \infty$,
if
\beq\lb{4.8}
\sup_{x\in B_{2R_{\eps_1}}}\tilde u(t_1, x;\xi, c)=\eta,\, \,\, \sup_{x\in B_{2R_{\eps_1}}} \tilde u(t,x;\xi, c)\le \eta,\qquad \forall\, t_1< t<t_2,
\eeq
then
$$
\inf_{x\in B_{R_{0}}} \tilde u(t,x;\xi, c)\ge \tilde\delta_\eta\qquad \forall\, t_1< t< t_2.
$$
\end{lem}

\begin{proof}
First of all, we give a construction of $\eps_1$. Let $\bar{a}$ from \eqref{bar-r-eq} and set
\begin{equation}\lb{T0}
T_0:=\max\left\{1,\lambda_0^{-1}\ln 4\right\}.
\end{equation}
Note that
$ u^*(t,x;\xi,c):=  e^{\lambda (c, \bar a) t}\phi(x;\xi,c,\bar a)$ is the solution of
\begin{equation}
\label{new-cut-off-eq1}
\begin{cases}
 u_t=\Delta u+c\xi \cdot \nabla u+\bar a u, \quad & x\in B_{R_0}, \,\, t>0,\\
 u(t,x)=0, \quad & x\in \p{B_{R_0}}, \,\,  t>0,\\
 u(0,x)= \phi(x;\xi,c,\bar a), \quad &x\in B_{R_0}.
\end{cases}
\end{equation}
It follows from \eqref{T0} and \eqref{lambda-lambda-eq} that
\begin{equation}
\label{cut-off-linear-solu1}
u^*(T_0,x;\xi,c)\ge 4 \phi(x;\xi,c,\bar a)\quad \forall\, x\in B_{R_0},\,\, c\in[-2\sqrt a -\delta_0,2\sqrt a-\delta_0],\,\, \xi\in \mathbb{S}^{N-1}.
\end{equation}

For  a given $C^1$ function $q(t,x)$, consider
\begin{equation}
\label{cut-off-linear-eq1}
\begin{cases}
 u_t=\Delta u+c\xi \cdot \nabla u+\nabla q\cdot\nabla u+\bar a u, \quad & x\in B_{R_0}, \,\, t>0,\\
 u(t,x)=0, \quad & x\in \p{B_{R_0}}, \,\,  t>0,\\
 u(0,x)= \phi(x;\xi,c,\bar a), \quad & x\in B_{R_0},
\end{cases}
\end{equation}
where $R_0$ is given in \eqref{L-R-eq}.
Let
$u_q(t,x;\xi,c,\phi)$ be the solution of \eqref{cut-off-linear-eq1}.
We claim that there is $\eps_2>0$ such that for any $C^1$-function $q(t,x)$, if
$$\|\nabla q\|_\infty:=\|\nabla q\|_{C([0,T_0]\times \bar B_{R_0})}<|\chi|\tilde M \eps_2\quad \text{with  $\tilde M$ from Lemma \ref{asymptotic2-lem2}},
$$
then
\begin{equation}
\label{cut-off-linear-solu2}
u_q(T_0,x;\xi,c,\phi)\ge 2\phi(x;\xi,c,\bar a)\quad \forall\, x\in B_{R_0},\,\, c\in [-2\sqrt a+\delta_0,2\sqrt a-\delta_0],\,\, \xi\in \mathbb{S}^{N-1}.
\end{equation}

In fact, recall that $u^*$ is a classical solution of \eqref{new-cut-off-eq1},
$u^*(t,x)>0$ for $t>0$ and $x\in B_{R_0}$, and $u^*(t,x)=0$ for $t>0$ and $x\in\partial B_{R_0}$.  By \cite[Theorem 2]{Friedman-1} and the continuity of  $\frac{\partial u^*(T_0,x,\xi,c)}{\partial\nu}$ in $x\in\partial B_{R_0}$, $|c|\le 2\sqrt a-\delta_0$, and $\xi\in \mathbb{S}^{N-1}$,  there holds 
\begin{equation}
\label{hopf-eq}
\inf_{x\in \partial B_{R_0},|c|\le 2\sqrt a-\delta_0,\xi\in \mathbb{S}^{N-1}}\frac{\partial u^*(T_0,x;\xi,c)}{\partial\nu}<0
\end{equation}
where $\partial/\partial \nu$ denotes the outer   normal derivative.
By \cite[Theorem 3.4.1]{Hen},
\begin{equation}
\label{continuity-eq1}
\lim_{\|\nabla  q\|_\infty\to 0} \|u_q(T_0,\cdot;\xi,c, \phi)-u^*(T_0,\cdot;\xi,c)\|_{C^1(\bar B_{R_0})}=0
\end{equation}
uniformly in $c\in[-2\sqrt a-\delta_0,2\sqrt a+\delta_0]$ and $\xi\in \mathbb{S}^{N-1}$.
Note that
$$
\frac{\partial u^*(T_0,x;\xi,c)}{\partial\nu}=\frac{1}{R_0}\nabla u^*(t,x;\xi,c)\cdot x,\quad \forall\, x\in \partial B_{R_0}.
$$
Thus, by   \eqref{hopf-eq}  and \eqref{continuity-eq1},  there is $r_1>0$ such that
\begin{equation}
\label{continuity-eq2}
\inf_{x\in B_{R_0}\setminus B_{R_0-r_1},|c|\le 2\sqrt a-\delta_0,\xi\in S^{N_1}}\nabla u^*(T_0,x;\xi,c)\cdot x<0.
\end{equation}
While, away from the boundary, we have
$$
\inf_{x\in B_{R_0-r_1},|c|\le 2\sqrt a-\delta_0,\xi\in \mathbb{S}^{N-1}}\phi(x;\xi,c,\bar a)>0.
$$
Then, by  \eqref{cut-off-linear-solu1}, \eqref{continuity-eq1} and \eqref{continuity-eq2}, there is
$\eps_2>0$ such that for any $C^1$ function $q$ satisfying that
 $$\|\nabla q\|_{(C[0,T_0]\times\bar B_{R_0})}<|\chi|\tilde M\eps_2,
$$
 we have for all $|c|\le 2\sqrt a-\delta_0$ and $\xi\in \mathbb{S}^{N-1}$,
\begin{equation}
\label{continuity-eq3}
u_q(T_0,x;\xi,c,\phi)\ge 2 \phi(x;\xi,c,\bar a),\quad\text{ in $B_{R_0-r_1}$},
\end{equation}
and for all  $ x\in B_{R_0}\setminus B_{R_0-r_1}$,
\begin{equation*}
-\nabla u_q(T_0,x;\xi,c,\phi)\cdot x\ge  - \frac{1}{2}\nabla u^*(T_0,x; c,\xi)\cdot x.
\end{equation*}
This,  together with   \eqref{cut-off-linear-solu1}, yields for $x\in B_{R_0}\setminus B_{R_0-r_1},$
\beq\label{continuity-eq4}
\begin{aligned}
u_q(T_0,x;\xi,c,\phi)
&=-\left(\frac{R_0}{|x|}-1\right)\int_0^1 \nabla u_q(T_0,sx+(1-s)R_0x/|x|;\xi,c,\phi)\cdot  x \, ds \\
&\ge -\frac{1}{2} \left(\frac{R_0}{|x|}-1\right)\int_0^1 \nabla u^*(T_0,sx+(1-s)R_0x/|x|;\xi,c)\cdot x\, ds\\
&=\frac{1}{2} u^*(T_0,x;\xi,c)\ge 2 \phi(x;\xi,c,\bar a) \qquad \forall\, |c|\le 2\sqrt a-\delta_0,\,\, \xi\in \mathbb{S}^{N-1}.
\end{aligned}
\eeq
The  claim \eqref{cut-off-linear-solu2} then follows from \eqref{continuity-eq3} and \eqref{continuity-eq4}.

Finally, we let
\begin{equation*}
\eps_1=\min\left\{\frac12,\, \left(\eps_2\right)^2,\, \left (\frac{a-\bar a}{|\chi|\tilde M+b}\right)^4,\,  \left(\frac{2}{R_0}\right)^{\frac{2N}{2+N}}\right\}.
\end{equation*}
Recall $R_{\eps_1}={2}{\eps_1^{-\frac{2+N}{2N}}}$ and so $
R_{\eps_1}\ge R_0.$

\smallskip

Next, we prove that the lemma holds with the above $\eps_1$. By the assumption \eqref{4.8}  and
 Lemma \ref{lem-3},  there is  $\delta_1=\delta_1(\eta)>0$ such that
\begin{equation*}
    \inf_{x\in B_{2R_{\eps_1}}} \tilde u(t,x;\xi, c)\ge \delta_1, \quad \forall\, t_1\le t\le t_1+T_{\eps_1}+1,\,\, |c|\le 2\sqrt a-\delta_0,\,\, \xi\in \mathbb{S}^{N-1}.
\end{equation*}
Thus, the proof is finished if $t_2\le t_1+T_{\eps_1}+1$. 

In the following, we assume that $t_2>t_1+T_{\eps_1}+1$.  
By the assumption and  Lemma \ref{asymptotic2-lem2}, we have
$$
|\nabla\tilde  v(t,x;\xi,c)|\le  \tilde M \eps_1^{\frac{1}{2}}\quad\text{and}\quad |\Delta \tilde v(t,x;\xi,c)|\le \tilde M \eps_1^{\frac{1}{4}}\quad \forall\, t_1+T_{\eps_1}\le t<t_2, \, x\in B_{R_{\eps_1}}.
$$
By the definition of $\eps_1$, for any  $\eta\leq \eps_1$, we have
$$
a-|\chi| \tilde M \eps_1^\frac{1}{4}-b\eta\ge a-\left(|\chi|\tilde M +b\right) \eps_1^\frac{1}{4}\ge \bar a.
$$
This implies that
\begin{align*}
\tilde{u}_t&\geq \Delta \tilde{u}+c\xi \cdot \nabla \tilde{u}- \chi \nabla \cdot(\tilde{u}  \nabla\tilde  v) + a \tilde{u}-b\eta\tilde{u}\\
&\geq \Delta \tilde{u}+c\xi \cdot \nabla \tilde{u}- \chi \nabla \tilde{u} \cdot \nabla\tilde  v + \bar a \tilde{u}, \quad \forall\,  t\le t_1+T_{\eps_1}+1\le t<t_2,\, x\in B_{R_{\eps_1}}.
\end{align*}
Also using that $\tilde u(t_1+T_{\eps_1}+1,x;\xi,c)\ge \delta_1\geq  \delta_1 u_q(0,x;\xi,c,\phi)$ and that $\tilde{u}$ is non-negative in the whole domain,
it follows from the comparison principle that
$$
\tilde u(t_1+T_{\eps_1}+1+t,x;\xi,c)\ge  \delta_1 u_q(t,x;\xi,c,\phi),\quad 0\le t<t_2-t_1-T_{\eps_1}-1,\quad x\in B_{R_0},
$$
where  $q(t,x)=-\chi \tilde v(t+t_1+T_{\eps_1}+1,x;\xi,c)$.

Let $n_0\ge 0$ be  such that
$$t_1+T_{\eps_1}+1+ n_0 T_0<t_2\quad {\rm and}\quad t_1+T_{\eps_1}+1+(n_0+1) T_0\ge t_2.$$
Since
$$
|\nabla q(t,x)|=|\chi||\nabla \tilde v(t,x;\xi,c)|\le |\chi|\tilde M \eps_1^{\frac{1}{2}}\le |\chi|\tilde  M_1 \eps_2,
$$
by \eqref{cut-off-linear-solu2}, we get
\begin{align*}
\tilde u(t_1+T_{\eps_1}+1+k T_0,x;\xi,c)&\ge \delta_1 u_q(kT_0,x;\xi,c,\phi)\\
&\ge 2^{k} \delta_1 \phi(x;\xi,c,\bar a)\quad \forall\, x\in B_{R_0}, \,\, k=1,2,\cdots,n_0.
\end{align*}
Applying Lemma \ref{lem-3} implies that there is $0<\tilde\delta_\eta\leq \delta_1$ such that for any $-2\sqrt{a}+{\delta_0}\leq c\leq 2\sqrt{a}-{\delta_0}$, any $\xi\in \mathbb{S}^{N-1}$,
$$
\inf_{x\in B_{R_{0}}}\tilde u(t, x;\xi, c)\ge \tilde\delta_\eta\qquad \forall \, t_1\le t<t_2.
$$

\end{proof}

Now, we prove  Theorem \ref{speed-lower-bound-thm}(1).

\begin{proof}[Proof of Theorem \ref{speed-lower-bound-thm}(1)]
As it is pointed out in the above, to prove  Theorem \ref{speed-lower-bound-thm}(1), it suffices to prove \eqref{lower-bound-eqq1-1}.

{Let $\eps_0,T_0, R_0, \eps_1,T_{\eps_1}$  be as in the above.   Let
\begin{equation*}
     \delta_*:=\inf\left\{ \tilde u(T_{ \eps_1}+1, x;\xi, c)\,|\, x\in B_{R_0},\,\, \xi\in \mathbb{S}^{N-1},\,\, |c|\le 2\sqrt a-\delta_0\right\}.
\end{equation*}}
Since $u_0(x)\ge 0$ has nonempty support,  by Lemma \ref{asymptotic-lm1}, $\delta_*>0$. Let
$$
k_*:=\inf\{k\in \Z^+\,|\, 2^{k}{\delta_*}\ge   \eps_1\}
\quad {\rm and}\quad
T_*:=T_{\eps_1}+1+k_*   T_0.
$$
We claim that there is a $\delta >0$  independent of $c$ and $\xi$ such that,
\begin{equation}\label{4.18}
    \inf_{x\in B_{R_0}}\tilde u(t,x;\xi,c)\ge \delta\qquad \forall t> \, T_*.
\end{equation}
{\bf Case 1: }  Suppose that for any $t\geq T_{\eps_1}\geq2 $,
\begin{equation}\label{lbound}
  \sup_{x\in B_{2R_{\eps_1}}} \tilde u(t,x;\xi,c)\ge \eps_1.
\end{equation}
In this case,  by applying Lemma \ref{lem-3} repeatedly, we get
\begin{equation*}
 \inf_{x\in B_{2R_{\eps_1}}} \tilde u(t,x;\xi,c)\ge \delta_{\eps_1}\quad \forall\,  t\ge T_{\eps_1}.
\end{equation*}
Hence, \eqref{4.18} holds with $\delta=\delta_{\eps_1}$.

\smallskip

\noindent {\bf Case 2:}
Suppose that there exists $t > T_{\eps_1}$ such that \eqref{lbound} is not true. Then the set $\{t>T_{\eps_1}\,|\, \sup_{x\in B_{2R_{\eps_1}}}\tilde u(t,x;\xi,c)< \eps_1 \}$ is non-empty, and the set is open by continuity.  This means we can write it as union of some disjoint open intervals i.e.,
\[
\Big\{t>T_{\eps_1} : \sup_{x\in B_{2R_{\eps_1}}}\tilde u(t,x;\xi,c)< \eps_1 \Big\} = \bigcup _{i\in I}(t_i, s_i),\quad\text{ for some }t_i, s_i \in [T_{\eps_1},\infty), t_i<s_i<t_{i+1}.
\]


\noindent{\bf Case 2.1:} Suppose that   $t_i> T_{\eps_1}$ for all $i$. Then
$$\sup_{x\in B_{2R_{\eps_1}}}\tilde u(t_i,x;\xi, c)=\eps_1,\, \,\, \sup_{x\in B_{2R_{\eps_1}}} \tilde u(t,x;\xi, c)< \eps_1,\,\, \forall\,\ t_i<t<s_i.$$
Then by Lemma \ref{lem-4}, there exists a $\tilde\delta_{{\eps_1}} >0$ independent of $\xi$ and $c$  such that
$$
\inf_{x\in B_{R_{0}}} \tilde u(t,x;\xi, c)\ge  \tilde\delta_{{\eps_1}}\quad \forall\,\ t_i\le t< s_i.
$$
For any $t>T_*$ and $t\not \in \bigcap_{i\in I}(t_i,s_i)$, we have
$$
\sup_{x\in B_{2R_{\eps_1}}} \tilde u(t,x;\xi,c)\ge \eps_1.
$$
Then Lemma \ref{lem-3} yields
$$
 \inf_{x\in B_{2R_{\eps_1}}} \tilde u(s,x;\xi,c)\ge \delta_{\eps_1}\quad \forall\,  t\le s\le t+ T_{\eps_1} +1.
$$
Hence
\eqref{4.18} holds with $\delta=\min\{\delta_{\eps_1},\tilde  \delta_{{\eps_1}}\}$.

\smallskip

\noindent {\bf Case 2.2:}   There is $i_0$ such that  $t_{i_0} = T_{\eps_1} $.  Note that
\begin{equation}\label{con1}
    \sup_{x\in B_{2R_{\eps_1}}} \tilde u(t,x;\xi, c)<{\eps}_1, \qquad \forall\, \, t_{i_0}<t<s_{i_0}.
\end{equation}
We claim that $s_{i_0}\le T_*$,  and then by the arguments in Case 2.1,
\eqref{4.18} holds with $\delta=\min\{\delta_{\eps_1},\tilde  \delta_{{\eps_1}}\}$.
In fact, assuming  $s_{i_0}>T_*$, then
$$
\inf_{x\in B_{R_0}}\tilde u(T_{\eps_1}+1,x;\xi,c)\ge \delta_*.
$$
By comparing $\tilde{u}(T_{\eps_1}+1+\cdot, \cdot;\xi, c)$ with $\delta_* u_q(\cdot,\cdot;\xi,c,\phi)$ in $[0,k_*T_0]\times B_{R_0}$, with $u_q$ from Lemma \ref{lem-4}, we get
\begin{align*}
\tilde u(T_{\eps_1}+1+k  T_0, x;\xi, c)\geq  2^{k} \delta_*  \phi(x;\xi, c,\bar a) \qquad \forall\,x\in B_{R_0}
\end{align*}
for $k=0,1,2,\cdots, k_*$. In particular, this implies that
\begin{align*}
\sup_{x\in B_{2R_{\eps_1}}}\tilde u(T_*, x;\xi, c)\geq \sup_{x\in B_{R_0}}\tilde u(T_{\eps_1}+1+k T_0, x;\xi, c) \geq  \eps_1,
\end{align*}
which contradicts with \eqref{con1}, and hence $s_{i_0}\le T_*$.

Overall, we conclude that for all $t\geq T_*$,
$$\inf_{\xi\in \mathbb{S}^{N-1},\,|c|\leq 2\sqrt{a}-\delta_0}\left(\inf_{|x|\le R_0} \tilde u(t,x;\xi,c)\right)\geq  \min\{\delta_{{\eps_1}}, \tilde\delta_{\eps_1}\}.$$
Theorem \ref{speed-lower-bound-thm}(1) is thus proved.
\end{proof}

\subsection{Proof of  Theorem  \ref{speed-lower-bound-thm}(2)}

In this subsection, we prove Theorem \ref{speed-lower-bound-thm}(2).

\begin{proof}[Proof of Theorem\ref{speed-lower-bound-thm}(2)]
We first prove \eqref{lower-bound-eqq2}.
Suppose for contradiction that there are  $0< c_1< c_{\rm low}^*(u_0,v_0)$,  $\delta>0$ and  $\{(t_n, x_{n})\} \subset \R^{N+1}$ such that
$t_{n}\rightarrow \infty,\, |x_{n}|\leq c_1t_{n} $,  and
\begin{equation}\label{aux-eq1}
\left|u(t_{n},x_{n})-\frac{a}{b}\right|\geq \delta, \quad \ \ \forall\ n\geq 1.
\end{equation}
Consider
\begin{equation*}
u_{n}(t,x)=u(t+t_{n}, x+x_{n}) \quad \text{and} \quad v_{n}(t, x)=v(t+t_{n}, x+x_{n})\quad \forall\, t\ge -t_{n},\,\, x\in \R^N.
\end{equation*}
By Lemma  \ref{derivative-boundedness-lm}, 
there is a function $(\hat {u},\hat {v})\in C^{2,1}(\R\times \R^N)$ and  a subsequence $\{(u_{n_k},v_{n_k})\}$ of $\{(u_{n},v_{n})\}$  such that
$$
\lim_{k\to\infty} (u_{n_k}(t,x),v_{n_k}(t,x))=(\hat u(t,x),\hat v(t,x))\quad \text{locally uniformly in}\,\, \R\times\R^N.
$$
Moreover, $(\hat u(t,x),\hat v(t,x))$ is an entire solution of \eqref{main-Eq}.

\smallskip

Choose $c'$ such that $c_1<c'<c_{\rm low}^*(u_0,v_0)$.   Then, for every $x\in\R^N$ and $t\in\R$, let us select $k$ such that   $t_{n_k}\geq \frac{|x|-  c't}{c'-c_1}$. This yields
\begin{eqnarray*}
|x+x_{n_k}| \leq    |x|+ c_1t_{n_k}  \leq   {c'}(t_{n_k}+t)
\end{eqnarray*}
which then implies that
$$
\hat {u}(t, x)=\lim_{k\rightarrow \infty}u(t+t_{n_k}, x+ x_{n_k})\geq \liminf_{s\rightarrow \infty}\inf_{|y|\leq {c'}s}u(s,y)\quad \forall\, (t,x)\in\R\times\R^N.
$$
By \eqref{lower-bound-eqq1},
\begin{equation}
\label{aux-new-thm1-eq1}
\inf_{(t,x)\in\R\times \R^N}\hat{u}(t, x)\ge \liminf_{s\rightarrow \infty}\inf_{|y|\leq {c'}s}u(s,y)=:\eps_0>0.
\end{equation}

Similarly to the proof of Proposition \ref{asymptotic-behavior-prop}, since $(\hat u(t,x),\hat v(t,x))$ is an entire solution of \eqref{main-Eq}, we get $\tau \hat v_t\leq \Delta\hat v-\eps_0\hat v$, which implies
$$
    {0\le \hat v(t, x) \le \|\hat v(-T,\cdot)\|_{\infty} e^{-\frac{\eps_0}{\tau}(t+T)}\le \sup_{s\in\R} \|\hat v(s,\cdot)\|_\infty e^{-\frac{\eps_0}{\tau}(t+T)}  \qquad \forall t\ge -T.}
$$ 
Letting $T\to -\infty$ yields
$
\hat v(t,x)\equiv 0.
$
Then, $\hat u(t,x)$ satisfies
\begin{equation*}
\hat u_t=\Delta \hat u+\hat u(a-b\hat u),\quad \forall t\in\R.
\end{equation*}
Moreover, since \eqref{aux-new-thm1-eq1} is for all $t\in\R$, we get 
$\hat u(t,x)\equiv\frac{a}{b}$  (also see the proof for \eqref{u-to-a-over-b-eq}). However, this  contradicts with \eqref{aux-eq1}. Hence,  \eqref{lower-bound-eqq2}  holds.

\smallskip

Next, we prove \eqref{lower-bound-eqq3}. We also prove it by contradiction.
Assume for contradiction that there are  $0< c_1< c_{\rm low}^*(u_0,v_0)$,  $\delta>0$ and  $\{( t_n,  x_{n})\} \subset \R^{N+1}$ such that
$ t_{n}\rightarrow \infty,\, | x_{n}|\leq c_1 t_{n} $,  and
\begin{equation}
\label{aux-eq1-1}
v( t_{n}, x_{n})\geq \delta, \quad \ \ \forall\ n\geq 1.
\end{equation}
Let
\begin{equation*}
 u_{n}(t,x)=u(t+ t_{n}, x+ x_{n}) , \quad \text{and} \quad  v_{n}(t, x)=v(t+t_{n}, x+x_{n})\quad \forall\, t\ge - t_{n},\,\, x\in \R^N.
\end{equation*}
Then, similarly as done in the above,  there is a function $(\tilde {u},\tilde {v})\in C^{2,1}(\R\times \R^N)$ and  a subsequence $\{( u_{n_k}, v_{n_k})\}$ of $\{( u_{n}, v_{n})\}$  such that
$$
\lim_{k\to\infty} (u_{n_k}(t,x),v_{n_k}(t,x))=(\tilde u(t,x),\tilde v(t,x))\quad \text{locally uniformly in}\,\, \R\times\R^N,
$$
and, moreover,
$
\tilde u(t,x)\equiv \frac{a}{b}$ and $\tilde v(t,x)\equiv 0.
$
However, this contradicts with \eqref{aux-eq1-1}, and so we proved \eqref{lower-bound-eqq3}.
\end{proof}

\section{Upper bound of spreading speeds}

In this section, we provide an upper bound for the spreading speed and prove Theorem \ref{speed-upper-bound-thm}.    Throughout the section, we fix $(u_0,v_0)\in X_c^+\times X_1^+$,  and  suppose  that $(u(t,x;u_0,v_0)$, $v(t,x;u_0,v_0))$ is a globally defined bounded solution of \eqref{main-Eq}.

\subsection{Proof of Theorem \ref{speed-upper-bound-thm}(1)}


Let us drop $u_0,v_0$ from the notations of $u(t,x;u_0,v_0)$ and $v(t,x;u_0,v_0)$.
 We first prove that  there is ${c_1}>0$ such that
\begin{equation}
\label{aux-new-thm3-eq2}
\lim_{t\to\infty} \sup_{|x|\ge {c_1}t}u(t,x )=0.
\end{equation}
This will imply that
$
c_{\rm up}^*(u_0,v_0)\leq c_1<\infty.
$

To prove \eqref{aux-new-thm3-eq2}, we first claim that  there are $t_0>0$ and   $M_0>0$ such that
\begin{equation}
\label{exponential-bound-u-eq}
u(t,x)\le M_0 e^{-\sqrt a |x|}\quad \forall\, 0\le t \le t_0,\,\, x\in\R^N.
\end{equation}
In fact, let $\tilde u(t,x)=e^{\sqrt{a(1+|x|^2)}} u(t,x)$ and $\tilde v(t,x)=v(t,x)$.
Writing $h:=e^{\sqrt{a(1+|x|^2)}}$, we have
\begin{align*}
h\nabla u&=\nabla \tilde u-(h^{-1}\nabla h)\tilde u,\\
h \Delta u&=\Delta \tilde u -\left(h^{-1}\Delta  h\right)\tilde u +2|h^{-1}\nabla h|^2\tilde u-2h^{-1}\nabla h\cdot\nabla\tilde u=\Delta \tilde u +\left(h^{-1}\Delta  h\right)\tilde u -2\nabla (\tilde uh^{-1}\nabla h ).    
\end{align*}
Thus, 
$(\tilde u(t,x),\tilde v(t,x))$  is a global classical solution of 
\begin{equation*}
\begin{cases}
\tilde u_{t}= \Delta \tilde u - \chi\nabla\cdot(\tilde u \nabla\tilde v)+ \tilde u(a-
bh^{-1}\tilde u) \\
\qquad\quad + \left(h^{-1}\Delta  h\right)\tilde u -2\nabla (\tilde uh^{-1}\nabla h ) +\chi \tilde u h^{-1}\nabla h\cdot\nabla \tilde v,\quad & \text{in } (0,\infty)\times\R^N,\\
{\tau \tilde v_t}=\Delta \tilde v-h^{-1}\tilde u\tilde v,\quad & \text{in } (0,\infty)\times\R^N,\\
\tilde u(0,x)=h(x) u_0(x),\quad \tilde v(0,x)=v_0(x), & x\in \R^N.
\end{cases}
\end{equation*}
Note that $\tilde u(0,\cdot)$ has a compact support, and $h^{-1}$, $h^{-1}\nabla h$ and $h^{-1}\Delta h$ are uniformly bounded. By the arguments of \cite[Proposition 2.1(1)]{HSZ} about local existence of solutions of \eqref{main-Eq}, there is
$\tilde T_{\max}\in (0,\infty]$ such that the classical solution $ (\tilde u,\tilde v)$ is unique in $(0,\tilde T_{\max})$ and it satisfies for any  $t_0\in (0,\tilde T_{\max})$,
$$
M_0=\sup_{t\in [0,t_0], x\in\R^N} \tilde u(t,x)<\infty.
$$
This implies that
$$
u(t,x)=e^{-\sqrt {a(1+|x|^2)}} \tilde u (t,x)\le M_0 e^{-\sqrt a|x|},\quad \forall\, t\in [0,t_0], \,\, x\in\R^N,
$$
which yields that \eqref{exponential-bound-u-eq} holds.

 Fix a $t_0>0$ such that \eqref{exponential-bound-u-eq} holds. By Lemma \ref{derivative-boundedness-lm}, we have
$$A: =\sup_{t\ge t_0}\left(\|\nabla v(t)\|_{\infty}+\|\Delta v(t)\|_{\infty}\right)<\infty.$$
Then, for some $c_1>0$ to be determined and for each $\xi\in \mathbb{S}^{N-1}$, let
\beq\lb{5.1}
u_\xi (t,x):=M_0 e^{- \sqrt{a}(x\cdot \xi-c_1 t)}.
\eeq
It is direct to see that
\begin{align*}
&\p_t u_\xi -\Delta u_\xi+\chi\nabla\cdot ( u_\xi \nabla v)-au_\xi+bu_\xi^2\\
&\qquad \geq c_1\sqrt{a} u_\xi-a u_\xi-|\chi|A ( \sqrt{a}+1)u_\xi-a u_\xi=(c_1\sqrt{a}-2a-|\chi|A( \sqrt{a}+1))u_\xi.
\end{align*}
Let us pick $c_1:=2\sqrt{a}+|\chi|A( 1+1/\sqrt{a})$, and thus $u_\xi$ is a supersolution to the equation satisfied by $u$. By \eqref{exponential-bound-u-eq},
\[
u(t_0,x)\leq M_0e^{-\sqrt{a} x\cdot\xi}\quad\text{for all $\xi\in \mathbb{S}^{N-1}$}.
\]
It then follows from the comparison principle that $u\leq u_\xi$ in $[t_0,\infty)\times\R^N$ for all $\xi\in\mathbb{S}^{N-1}$. We obtain
\beq\lb{5.3}
u(t,x )\le \min_{\xi\in\mathbb{S}^{N-1}}u_\xi(t,x)= M_1 e ^{-\sqrt{a} (|x|-c_1t)}\quad \forall\, t>t_0,\,\, x\in\R^N.
\eeq
This implies \eqref{aux-new-thm3-eq2} with ${c_1}=2\sqrt{a}+|\chi|A( 1+1/\sqrt{a})$.

\smallskip

Next, we prove \eqref{upper-bound-eqq1}. Fix a $c_2>c_{\rm up}^*(u_0,v_0)\geq c_{\rm low}^*(u_0,v_0)\geq 2\sqrt{a}$. Let $\delta\in (0,1)$ and $\eps\in (0,\delta)$  to be determined. In view of the definition of $c_{\rm up}^*(u_0,v_0)$, there is $T_\eps>0$ such that
\begin{equation*}
u(t,x )\leq \eps,\quad \forall\,\, t\ge T_\eps,\,\, |x|\ge c_2 t.
\end{equation*}
Then, by Lemma \ref{asymptotic-lm2}, there is $C_\delta>0$ such that
\begin{equation}
\label{upper-bound-eqq3}
|\nabla v(t,x )|,\,\, |\Delta v(t,x )|\le \delta +C_\delta u(t,x )^\frac12\leq \delta +C_\delta \eps^\frac12\quad \forall\, t\ge T_\eps,\,\, |x|\ge c_2 t.
\end{equation}

Let $w_\xi$ be defined similarly as in \eqref{5.1},
that is $w_\xi (t,x):=M_2 e^{- \sqrt{a}(x\cdot \xi-c_2 t)}$, but with $M_2$ given by
\beq\lb{5.4}
M_2:=\max \left\{
   M_1\exp({\sqrt{a}T_\eps(c_1-c_2))},\|u\|_\infty \right\}.
\eeq
By  \eqref{upper-bound-eqq3} and direct computations, in the region of $\{t\ge T_\eps, |x|\ge c_2 t\} $, we have
\begin{align*}
&\p_t w_\xi -\Delta w_\xi+\chi\nabla\cdot ( w_\xi \nabla v)-aw_\xi+bw_\xi^2\\
&\qquad \geq \left(c_2\sqrt{a}-2a-|\chi|(\delta+C_\delta\sqrt{\eps})\right)\left(\sqrt{a}+1\right)w_\xi.
\end{align*}
To have the above $\geq 0$ (then $w_\xi$ is a supersolution), we need
\beq\lb{5.2}
c_2\geq 2\sqrt{a}+|\chi|(\delta+C_\delta\sqrt{\eps})(1+1/{\sqrt{a}}).
\eeq
Since $c_2>\sqrt{2a}$, we can now fix $\delta\in(0,1)$ and then $\eps\in (0,\delta)$ to be sufficiently small such that \eqref{5.2} holds.


{To use the comparison principle to conclude with $w_\xi\geq u$ for all $t\ge T_\eps, |x|\ge c_2 t$, it remains to show that $w_\xi(T_\eps,x)\geq u(T_\eps,x)$ and $|x|\geq c_2T_\eps$, and $w_\xi(t,x)\geq u(t,x)$ with $t\geq T_\eps$ and $|x|=c_2t$.} Indeed, it follows from \eqref{5.3} and \eqref{5.4} that on the bottom boundary,
\[
u(T_\eps,x)\leq M_1 e^{\sqrt{a}c_1T_\eps}e^{-\sqrt{a}|x|}\leq M_2e^{\sqrt{a}c_2T_\eps}e^{-\sqrt{a}x\cdot\xi}.
\]
On the lateral boundary of $|x|=c_2t$, we have
\[
u(t,x)\leq \|u\|_\infty\leq M_2\leq w_\xi(t,x).
\]
Overall, we can conclude that for all $\xi\in\mathbb{S}^{N-1}$,
$$
u(t,x)\le M_2e^{-\sqrt{a} (|x|-c_2t)}\quad \forall \, t\ge T_\eps,\,\, |x|\ge c_2 t.
$$
This, together with \eqref{exponential-bound-u-eq} and \eqref{5.3},  finishes the proof of \eqref{upper-bound-eqq1}.

\subsection{Proof of Theorem \ref{speed-upper-bound-thm}(2)}


Since $V$ is the solution to \eqref{1.19}, we have
\[
v(t,x)=V(t,x)
-\frac{1}{\tau}\int_{0}^t \left(\frac{\tau}{4\pi (t-s)}\right)^{N/2}\int_{\R^N} e^{-\frac{\tau |x-y|^2}{4(t-s)}}u(s,y) v(s,y)dy ds.
\]
Because $u,v\geq 0$, it is direct to see that  {$v\le  V$}, and so to prove the conclusion it suffices to estimate the following from above
\[
\int_{0}^t \left(\frac{\tau}{4\pi (t-s)}\right)^{N/2}\int_{\R^N} e^{-\frac{\tau |x-y|^2}{4(t-s)}}u(s,y) v(s,y)dy ds.
\]

 Fix $c''>c_{\rm up}^*(u_0,v_0)$ and
fix $x$ such that $|x|\geq c''t$, and take $c_1:=\frac1{2}{(c''+c_{\rm up}^*(u_0,v_0))}$ (so $c''>c_1>c_{\rm up}^*(u_0,v_0)$). We decompose the  double integral into two terms
\beq\lb{5.5}
\begin{aligned}
&\int_{0}^t \left(\frac{\tau}{4\pi (t-s)}\right)^{N/2}\int_{|y|\geq c_1s} e^{-\frac{\tau |x-y|^2}{4(t-s)}}u(s,y) v(s,y)dy ds\\
&\qquad+\int_{0}^t \left(\frac{\tau}{4\pi (t-s)}\right)^{N/2}\int_{|y|\leq c_1s} e^{-\frac{\tau|x-y|^2}{4(t-s)}}u(s,y) v(s,y)dy ds=:Y^++Y^-.
\end{aligned}
\eeq

First, we estimate $Y^-$. Since $u,v$ are uniformly bounded,
\begin{align*}
Y^-&\leq C\left(\frac{\tau}\pi\right)^{N/2}\int_{0}^{t} (4 (t-s))^{-N/2}\int_{|y|\leq c_1s} e^{-\frac{{\tau}|x-y|^2}{4(t-s)}}dy ds\\
&= C\left(\frac{\tau}\pi\right)^{N/2}\int_{0}^{t} \int_{|z|\leq \frac{c_1s}{\sqrt{4(t-s)}}} e^{-{\tau}\left|\frac{x}{\sqrt{4(t-s)}}-z\right|^2} dzds.
\end{align*}
Using that $|a-z|^2\geq (|a|-|z|)^2$, $|x|\geq c''t$ and $c''>c_1$, we get
\begin{align*}
\tau^{-N/2}Y^-&\leq C\int_{0}^{t} \int_{|z|\leq \frac{c_1s}{\sqrt{4(t-s)}}} e^{-\tau\left(\frac{|x|}{\sqrt{4(t-s)}}-|z|\right)^2} dzds\leq C\int_{0}^{t} \int_{|z|\leq \frac{c_1s}{\sqrt{4(t-s)}}} e^{-{{\tau}}\left(\frac{c''t-c_1s}{\sqrt{4(t-s)}}\right)^2} dzds\\
&\leq C\int_{0}^{t} s^N(t-s)^{-\frac{N}2} e^{-{{\tau}}\frac{\left((c''-c_1)t+c_1(t-s)\right)^2}{{4(t-s)}}} ds\leq C\int_{0}^{t} s^N(t-s)^{-\frac{N}2} e^{-{{\tau}}\frac{(c''-c_1)^2t^2}{4(t-s)}} ds . 
\end{align*}
Note that there exist $C, \gamma>0$ such that  for all $s\in (0,t)$,
\[
\tau^{-\frac{N}2}(t-s)^{-\frac{N}2} e^{-{{\tau}}\frac{(c''-c_1)^2t^2}{4(t-s)}}\leq C e^{-\frac{2\tau\gamma t^2}{t-s}}\quad\text{and}\quad (\tau s)^N\leq Ce^{\gamma
\tau s}\leq Ce^{\frac{\gamma \tau t^2}{t-s}}.
\]
Thus we get
\begin{equation}
\lb{Xl}
Y^- \leq C\int_0^t e^{-\frac{\gamma \tau t^2}{t-s}}ds
\leq C\int_0^t e^{-\gamma\tau t}ds= Cte^{-\gamma\tau t}
\end{equation}
which converges to $0$ as $t\to \infty$.

Next, we estimate $Y^+$. By Theorem \ref{speed-upper-bound-thm}(1), there exist $C\geq 1$ and $\delta(=\sqrt{a})>0$ such that
\[
u(s,y)\leq Ce^{-\delta(|y|-(c_1+c_{\rm up}^*)s/2)}\quad\text{ for }|y|\geq c_1s.
\]
Thus, also using that $v$ is uniformly bounded, we obtain
\begin{align*}
Y^+&=\int_{0}^t \Big(\frac{\tau}{4\pi (t-s)}\Big)^{N/2}\int_{|y|\geq c_1s} e^{-\frac{\tau |x-y|^2}{4(t-s)}}u(s,y) v(s,y)dy ds\\
&\leq C\int_{0}^{t} \Big(\frac{\tau}{4\pi (t-s)}\Big)^{N/2}\int_{|y|\geq c_1s} e^{-\frac{\tau |x-y|^2}{4(t-s)}}e^{-\delta(|y|-(c_1+c_{\rm up}^*)s/2)}dy ds\\
&\leq C\iint_{(s\geq \frac{t}2 \text{ or }|y|\geq \frac{c''t}2 )\bigcap D}  \Big(\frac{\tau}{4\pi (t-s)}\Big)^{N/2} e^{-\frac{\tau |x-y|^2}{4(t-s)}}e^{-\delta(|y|-(c_1+c_{\rm up}^*)s/2)}dy ds\\
&\qquad + C\iint_{(s\leq \frac{t}2\text{ and }|y|\leq \frac{c''t}2 )\bigcap D} \Big(\frac{\tau}{4\pi (t-s)}\Big)^{N/2} e^{-\frac{\tau |x-y|^2}{4(t-s)}}e^{-\delta(|y|-(c_1+c_{\rm up}^*)s/2)}dy ds\\
&=:Y^+_1+Y^+_2,
\end{align*}
where $D:=\{(s,y)\,|\, s\in (0,t),|y|\geq c_1s\}$.
We first show that $Y^+_1$ is small. If $s\geq \frac{t}2 $, since $c_1>c_{\rm up}^*$, then in $D$ we have for some $\gamma>0$,
\[
e^{-\delta(|y|-(c_1+c_{\rm up}^*)s/2)}\leq e^{-\delta(c_1-c_{\rm up}^*)s/2}\leq e^{-\gamma t}.
\]
If $|y|\geq \frac{c''t}2 $, also using that $|y|\geq  c_1s$ and $c_1>c_{\rm up}^*$, there is $\gamma>0$ such that
\[
e^{-\delta(|y|-(c_1+c_{\rm up}^*)s/2)}\leq e^{-\delta\left(|y|-\frac{c_1+c_{\rm up}^*}{2c_1}|y|\right)}\leq e^{-\gamma t}.
\]
Applying these into the definition of $Y^+_1$ yields
\beq\lb{Xg1}
\begin{aligned}
Y^+_1&\leq Ce^{-\gamma t}\iint_{(s\geq \frac{t}2 \text{ or }|y|\geq \frac{c''t}2 )\bigcap D}  \Big(\frac{\tau}{4\pi (t-s)}\Big)^{N/2} e^{-\frac{\tau |x-y|^2}{4(t-s)}}dy ds\\
&\leq Ce^{-\gamma t}\int_{0}^t\int_{\R^N}  \Big(\frac{\tau}{4\pi (t-s)}\Big)^{N/2} e^{-\frac{\tau |x-y|^2}{4(t-s)}}dy ds\leq Cte^{-\gamma t}.
\end{aligned}
\eeq

Now we estimate $Y^+_2$. Recall that $c_1=\frac1{2}{(c''+c_{\rm up}^*)}$, so
\begin{align*}
Y^+_2&= C\int_{0}^{t/2}\int_{c_1s\leq|y|\leq \frac{c''t}{2}} \left(\frac{\tau}{4(t-s)}\right)^{N/2} e^{-\frac{ \tau|x-y|^2}{4(t-s)}}e^{-\delta(|y|-(c_1+c_{\rm up}^*)s/2)}dy ds\\
&\leq C\tau^\frac{N}2\int_{0}^{t/2} \int_{\frac{c_1s}{\sqrt{4(t-s)}}\leq |z|\leq \frac{c''t}{2\sqrt{4(t-s)}}} e^{-{\tau}\left(\frac{|x|}{\sqrt{4(t-s)}}-|z|\right)^2} dzds.
\end{align*}
Using $|x|\geq c''t$ and $|z|\leq \frac{c''t}{2\sqrt{4(t-s)}}$ yields $\frac{|x|}{\sqrt{4(t-s)}}-|z|\geq \frac{c''t}{2\sqrt{4(t-s)}}$. We get that for some $C,\gamma>0$,
\beq\lb{Xg2}
\begin{aligned}
Y^+_2&\leq C\tau^\frac{N}2\int_{0}^{t/2} \int_{\frac{c_1s}{\sqrt{4(t-s)}}\leq |z|\leq \frac{c''t}{2\sqrt{4(t-s)}}} e^{-{ \tau}\left(\frac{c''t}{2\sqrt{4(t-s)}}\right)^2} dzds\\
&\leq C\tau^\frac{N}2\int_{0}^{t/2} \int_{ |z|\leq \frac{c''\sqrt{t}}{2\sqrt{2}}} e^{-{ \tau}\left({c''\sqrt{t}}/{4}\right)^2} dzds\leq C\tau^\frac{N}2t^{1+\frac{N}2}e^{-\gamma\tau t}.
\end{aligned}
\eeq

Overall, plugging \eqref{Xl}, \eqref{Xg1} and \eqref{Xg2} into \eqref{5.5}, it follows that
\[
\sup_{|x|\geq c't}|v(t,x)-V(t,x)|\leq Cte^{-\gamma \tau t}+Cte^{-\gamma t}+C\tau^{\frac{N}2}t^{1+\frac{N}2}e^{-\gamma \tau t}\to 0 \quad\text{ as }t\to \infty.
\]
We conclude the proof.

\section{Existence of spreading speeds}

In this section, we prove the existence of spreading speeds under certain  natural conditions on $v_0$
and prove Theorems \ref{speed-thm} and \ref{speed-thm1}.

\subsection{Proof of Theorem \ref{speed-thm}(1)}

In this subsection,
we study the existence of spreading speeds under the condition that $v_0(x)$ is small for $|x|\gg 1$ in certain sense and prove Theorem \ref{speed-thm}(1). We first prove a lemma.

\begin{lem}
\label{speed-upper-bound-lm1}
Under the assumptions of Theorem \ref{speed-thm}, assume that $v_0\in C_0(\R^N)$ or $v_0\in L^p(\R^N)$ for some $p\ge 1$. Then
\begin{equation}
\label{speed-upper-bound-eq1}
\lim_{t\to \infty} \|v(t,\cdot;u_0,v_0)\|_\infty=0,\quad \lim_{t\to\infty}\|\nabla v(t,\cdot;u_0,v_0)\|_\infty=0,\quad  \lim_{t\to\infty}\|\Delta v(t,\cdot;u_0,v_0)\|_\infty=0.
\end{equation}
\end{lem}

\begin{proof}
We divide the proof into two steps.

\smallskip

\noindent {\bf Step 1.} In this step, we prove the lemma for the case $v_0\in C_0(\R^n)$.
\smallskip

First, note that, for any ${\eps}>0$, there is $R(\eps)>0$  such that
$$
\Big(\frac{\tau}{\pi}\Big)^{N/2}\int_{|z|>R(\eps)} e^{-{\tau}|z|^2} \|v_0\|_\infty dz<{\eps},
$$
and
$$
|v_0(y)|<{\eps}\quad {\rm for\,\, all}\quad |y|\ge R(\eps).
$$
 Note  also that
$$
|x+2\sqrt t z|\ge R(\eps) \quad {\rm for\,\, all}\,\,  |x|\ge 2\sqrt t R(\eps) +R(\eps), \,\, |z|\le R(\eps).
$$
 Hence, for any $t>0$ and $|x|>2\sqrt t R(\eps)+R(\eps)$,
\begin{align}
\label{speed-upper-bound-eq2}
v(t,x;u_0,v_0)&\le {\Big(\frac{\tau}{\pi}\Big)^{N/2}}\int_{\R^N} e^{-{\tau}|z|^2} v_0(x+2\sqrt t z)dz\nonumber\\
&\le {\Big(\frac{\tau}{\pi}\Big)^{N/2}}\int_{|z|>R(\eps)} e^{-{\tau}|z|^2}\|v_0\|_\infty dz
+{\Big(\frac{\tau}{\pi}\Big)^{N/2}}\int_{|z|\le R(\eps)} e^{-{\tau}|z|^2} v_0(x+2\sqrt t z)dz\nonumber\\
&\le {\eps}+ {\eps} {\Big(\frac{\tau}{\pi}\Big)^{N/2}}\int_{|z|\le R(\eps)} e^{-{\tau}|z|^2} dz.
\end{align}

Next, by Theorem \ref{speed-lower-bound-thm}(2), we have
 $$
\limsup_{t\to\infty}\sup_{|x|\le ct} v(t,x;u_0,v_0)=0\quad \forall\, 0<c<c^*:=2\sqrt a.
$$
Therefore, for any ${\eps}>0$, there is $T(\eps)>0$ such that
\begin{equation*}
2\sqrt t R(\eps)+R(\eps) <\sqrt a t\quad \forall\, t\ge T(\eps)
\end{equation*}
and
\begin{equation}
\label{speed-upper-bound-eq4}
v(t,x;u_0,v_0)<{\eps}\quad \forall\, t>T(\eps),\,\, |x|\le \sqrt a t.
\end{equation}
By \eqref{speed-upper-bound-eq2}--\eqref{speed-upper-bound-eq4},
$$
v(t,x;u_0,v_0)<\left(1+\Big(\frac{\tau}{\pi}\Big)^{N/2}\int_{\R^N} e^{-{\tau}|z|^2}dz\right){\eps} \quad \forall\, t>T(\eps).
$$
This implies that
\begin{equation}
\label{new-conv-eq3}
\lim_{t\to\infty} \|v(t,\cdot;u_0,v_0)\|_\infty=0.
\end{equation}

Now, we assume for contradiction that there are $\eps_0>0$,  $t_n\to\infty$, and $x_n\in\R^N$ such that 
\beq\lb{3333}
|\nabla v(t_n,x_n;u_0,v_0)|+|\Delta v(t_n,x_n;u_0,v_0)|\ge \eps_0\quad\forall\, n\ge 1.
\eeq
For $t\ge -t_n$ and $x\in\R^N$, consider
$$
u_n(t,x)=u(t+t_n,x+x_n;u_0,v_0) \quad {\rm and}\quad v_n(t,x)=v(t+t_n,x+x_n;u_0,v_0)
.
$$
By the arguments for \eqref{new-conv-eq1} and statement below it,  without loss of generality, we may assume that 
$$
\nabla v_n(t,x)\to 0,\quad \Delta v_n(t,x)\to 0\quad\text{as $n\to \infty$, }
$$
locally uniformly in $(t,x)\in\R\times\R^N$.
In particular,
$$
|\nabla v_n(0,0)|+|\Delta v_n(0,0)|=|\nabla v(t_n,x_n)|+|\Delta v(t_n,x_n)|=0\quad\text{as $n\to \infty$, }
$$
which contradicts with \eqref{3333}. Hence, we proved \eqref{speed-upper-bound-eq1} for the case when $v_0\in C_0(\R^N)$.

\smallskip

\noindent {\bf Step 2.} In this step, we prove the theorem for the case $v_0\in L^p(\R^N)$ for some $p\ge 1$.

\smallskip

If $v_0\equiv 0$, nothing needs to be proved. In the following, we assume that $v_0\not\equiv 0$.
Note that
\begin{equation}
\label{v-lp-eq1}
\frac{{\tau}}{p}{\frac{d}{dt}}\int_{\R^N} v^p(t,x;u_0,v_0)dx=-(p-1)\int_{\R^N}v^{p-2}|\nabla v|^2dx-\int_{\R^N} u v^{p}dx \le 0\quad \forall\, t>0.
\end{equation}
Hence, $\int_{\R^N} v^p(t,x;u_0,v_0)dx$ is non-increasing as $t$ increases. We claim that
$$
\lim_{t\to\infty}\|v(t,x;u_0,v_0)\|_\infty=0.
$$
In fact, assume this is not true and then there are ${\eps}_0>0$, $t_n$ strictly increasing to $\infty$, and
$x_n\in\R^N$ such that
\begin{equation}
\label{assumption-eq1}
v(t_n,x_n;u_0,v_0)\ge {\eps}_0.
\end{equation}
{ 
By the similar arguments of \eqref{new-conv-eq1} again}, we may assume that there are $u^*(t,x)$ and  $v^*(t,x)$ such that
$$
\lim_{t\to\infty} u(t+t_n,x+x_n;u_0,v_0)=u^*(t,x),\quad
\lim_{n\to\infty} v(t+t_n,x;u_0,v_0)=v^*(t,x)
$$
locally uniformly in $t\in\R$ and $x\in\R^N$, and
$v^*(t,x)$ satisfies
$$
\tau v^*_t=\Delta v^*- u^* v^*,\quad t\in\R,\,\, x\in\R^N.
$$

Notice that $v$ is bounded above by the solution to the heat equation with initial data $v_0$. Then, by the dominated convergence theorem and
the monotonicity of $\int_{\R^N}v^p(t,x;u_0,v_0)dx$, we have
\begin{equation}
\label{v-lp-eq2}
\int_{\R^N} (v^*(t,x))^pdx=
\lim_{n\to\infty}\int_{\R^N} v^p(t+t_n,x;u_0,v_0)dx\le \int_{R^N}v_0^p(x)dx\quad \forall\, t\in\R.
\end{equation}
On the other hand, for any $m\ge 1 $, {there is $m'\ge m$ such  that}  $t_m+t_n\le t_{n+m'}$ for all $n=1,2,\cdots$. Then by \eqref{v-lp-eq1} and \eqref{v-lp-eq2},
\begin{align*}\int_{\R^N}(v^*(t_m, x))^p dx&=\lim_{n\to\infty}  \int_{\R^N} v^p(t_m+t_n;x,u_0,v_0)dx\\
&\ge \lim_{n\to\infty} \int_{R^N}v^p(t_{n+m'},x;u_0,v_0)dx=\int_{\R^N} (v^*(0,x))^pdx.
\end{align*}
This implies that
\begin{equation}
\label{v-lp-eq3}
\int_\Omega ( v^*(0,x))^pdx\le \int_\Omega (v^*(t_m,x))^pdx\quad \forall\, m=1,2,\cdots.
\end{equation}
 By \eqref{v-lp-eq1} with $v$ and $u$ being replaced $v^*$ and $u^*$, respectively, we have that
$\int_\Omega (v^*(t,x))^pdx$ is non-increasing.
This, together with \eqref{v-lp-eq3}, yields that
$$\int_\Omega (v^*(t,x))^pdx=\int_\Omega (v^*(0,x))^pdx\quad \forall\, t>0.
$$

By \eqref{v-lp-eq1} with $v$ and $u$ being replaced $v^*$ and $u^*$ again, we find
\begin{equation*}
0=\frac{{\tau}}{p}{\frac{d}{dt}}\int_{\R^N}( v^*(t,x))^pdx=-(p-1)\int_{\R^N}(v^*)^{p-2}|\nabla v^*|^2dx-\int_{\R^N} (u^* (v^*))^{p}dx.
\end{equation*}
Thus, we must have
$v^*(t,x)\equiv {\rm constant}$, and since $\int_{\R^N}(v^*(t,x))^pdx<\infty$, we have
$v^*(t,x)\equiv 0$.
This clearly contradicts with \eqref{assumption-eq1}.
Hence, the claim holds. By the arguments in Step 1, \eqref{speed-upper-bound-eq1} holds for the case when $v_0\in L^p(\R^N)$ for some $p\ge 1$. 
\end{proof}

We now prove Theorem \ref{speed-thm}(1).

\begin{proof}[Proof of Theorem\ref{speed-thm}(1)]
First, by Lemma  \ref{speed-upper-bound-lm1} for any fixed $\eps>0$, there exists a $T_\eps>1$ such that
\begin{equation}
\label{aux-new-thm4-eq1}
|\nabla v(t,x;u_0,v_0)|\le \eps,\,\, |\Delta v(t,x;u_0,v_0)|\le \eps\qquad \forall\, t\ge T_\eps,\,\,
x\in\R^N.
\end{equation}
%

Next, {take $c:=2\sqrt {a}+|\chi|\eps (1 +1/\sqrt{a})$.
The rest of the proof is similar to the one of Theorem \ref{speed-upper-bound-thm}(1).
Indeed, let $c_1=2\sqrt{a}+|\chi|A(1+1/\sqrt{a})$ and $M_1$ from \eqref{5.1}, and for any $\xi\in\mathbb{S}^{N-1}$,
define
\[
u_\xi(t,x):=M_1e^{-\sqrt{a} (x\cdot \xi-c(t-T_\eps)-c_1T_\eps)}
\]
which, by \eqref{aux-new-thm4-eq1}, satisfies for all $t\geq T_\eps$ and $x\in\R^N$,
\begin{align*}
\p_t u_\xi=c\sqrt{a}  u_\xi =(2a+ |\chi|\eps(\sqrt{a}+1))u_\xi \ge \Delta u_\xi-\chi \nabla\cdot( u_\xi \nabla v)+a u_\xi-bu_\xi^2.
\end{align*}

At $t=T_\eps$, using \eqref{5.3}, we have
\begin{align*}
    u(T_\eps, x; u_0, v_0)\le  M_1 e^{-\sqrt{a}(x\cdot\xi-c_1 T_\eps)} = u_\xi(T_\eps,x)\qquad\forall\, x\in\R^N\text{ and }\xi\in\mathbb{S}^{N-1}.
\end{align*}
}
Thus it follows from the comparison principle that
\[
u(t,x;u_0,v_0)\leq u_\xi(t,x)\qquad\forall\, x\in\R^N,\, t\geq T_\eps\text{ and }\xi\in\mathbb{S}^{N-1}.
\]
This implies that
$$
c_{\rm up}^*(u_0,v_0)\le 2\sqrt {a}+|\chi|\eps (1 +1/\sqrt{a})\quad \forall\, \eps>0.
$$
Passing $\eps\to 0$ and applying Theorem \ref{speed-lower-bound-thm}(1) show
$
c_{\rm low}^*(u_0,v_0)=c_{\rm up}^*(u_0,v_0)=2\sqrt a.
$
\end{proof}

\subsection{Proofs of Theorem \ref{speed-thm}(2)  and   Theorem \ref{speed-thm1}}\lb{S.6.2}

In this subsection,  we investigate the spreading speeds when $|\chi|\ll 1$ and when $v_0$ is not small for large $|x|$,
and prove Theorem \ref{speed-thm}(2) and Theorem \ref{speed-thm1}.  We focus on the case that $\tau=1$, and we consider \eqref{main-Eq1} with $\sigma\in (0,1]$.
Throughout this subsection, we assume that $(u,v)$ is a solution to \eqref{main-Eq1} for all time with initial data $(u_0,v_0)$ satisfying the conditions in Theorem \ref{speed-thm}(2).

\medskip

As mentioned in Remark \ref{speed-rk} (4),
to prove Theorem \ref{speed-thm}(2) and Theorem \ref{speed-thm1}, we make the following change of variable,
$
\zeta:=1-v.
$
Then  $(u,\zeta)$ satisfies
\begin{equation*}
\begin{cases}
u_{t}=\Delta u + \chi\nabla \cdot (u \nabla \zeta)+ u(a-bu^\sigma),\quad  &x\in\R^N, \\
{\zeta_t}=\Delta \zeta +u(1-\zeta),\quad &x\in\R^N
\end{cases}
\end{equation*}
with initial data $u_0$ and $\zeta_0:=1-v_0$.
We will estimate $w:={\zeta^p}/{u}$ for some $p>1$. 

\smallskip

To this end, we first  discuss the inf- and sup- convolution technique.
For $T>0$, suppose $\rho_1,\rho_2\in C^\infty((0,T)\times\bbR^N)$ and let $r(t)\in C^\infty((0,T))$ be non-negative. Define
\begin{equation}\label{convl}
\un{\rho}(t,x):=\sup_{y\in B(x,r(t))}
\rho_1(t,y),\qquad
\ba{\rho}(t,x):=\inf_{y\in B(x,r(t))}
\rho_2(t,y).
\end{equation}
Then $\un{\rho} $ and $\ba{\rho} $ are Lipschitz continuous.
Let $y_{1,t}=y_{1,t}(x)\in\ba{B(x,r(t))}$ be such that $\un{\rho} (\cdot,t)=\rho_1(t,y_{1,t}(\cdot))$.
Then the following holds:
\beq\lb{6.1}
(\Delta \un{\rho} )(t,x)\geq (\Delta\rho_1)(t,y_{1,t}(x)),\quad (\nabla \un{\rho} )(t,x)=(\nabla\rho_1)(t,y_{1,t}(x))
\eeq
and
\beq\lb{6.2}
(\partial_t \un{\rho} )(t,x)=(\partial_t \rho_1)(t,y_{1,t}(x))+r'(t)|\nabla \rho_1|(t,y_{1,t}(x)).
\eeq
The first inequality in \eqref{6.1} needs to be understood in the viscosity sense. The proof can be found in Lemmas 5.2, 5.3 \cite{KZ2022} and Lemma 5.4 \cite{kimzhang21} for a more general case.
Similarly, assuming $y_{2,t}=y_{2,t}(x)\in\ba{B(x,r(t))}$ to satisfy that $\ba{\rho} (\cdot,t)=\rho_2(t,y_{2,t}(\cdot))$, we have
\[
(\Delta \ba{\rho})(t,x)\leq (\Delta\rho_2)(t,y_{2,t}(x)),\quad (\nabla \ba{\rho})(t,x)=(\nabla\rho_2)(t,y_{2,t}(x)),
\]
and
\[
(\partial_t \ba{\rho})(t,x)=(\partial_t \rho_2)(t,y_{2,t}(x))-r'(t)|\nabla \rho_2|(t,y_{2,t}(x)).
\]

Let us specify the selection of parameters. For given $r_0>0$, set $\rho_1$ to be the unique solution to
\[
\partial_t\rho_1=\Delta \rho_1,\quad
\rho_1(0,x)=\inf_{y\in  B(x,r_0)}u_0(y),\quad \forall\, x\in\R^N,
\]
and $\rho_2$ the unique solution to
\[
\partial_t\rho_2=\Delta \rho_2,\quad \rho_2(0,x)=\sup_{y\in B(x,r_0)}u_0(y),\quad \forall\, x\in\R^N.
\]
Then, let $\un{\rho}$ and $\ba{\rho}$ be defined in \eqref{convl} with $r(t):=r_0(1- t/(4\beta))$ for some $\beta\in (0,1]$. The construction immediately yields that
\beq\lb{6.6}
\un{\rho}(0,x)= \sup_{y\in B(x,r_0)}\inf_{y'\in B(y,r_0)} u_0(y')\leq u_0(x)\leq  \inf_{y\in B(x,r_0)}\sup_{y'\in B(y,r_0)}u_0(y')=\ba{\rho}(0,x).
\eeq
It follows from \eqref{6.1} and \eqref{6.2} that
\beq\lb{6.3}
\un{\rho}_t\leq \Delta \un{\rho} - (4\beta)^{-1} r_0|\nabla\un{\rho}|\quad {\rm for}\quad (t,x)\in [0,4\beta)\times\R^N.
\eeq
Similarly, we have
\beq\lb{6.4}
\ba{\rho}_t\geq \Delta \ba{\rho} +(4\beta)^{-1} r_0|\nabla\ba{\rho}|\quad {\rm for}\quad (t,x)\in [0,4\beta)\times\R^N.
\eeq
Let us comment that here and below, inequalities involving derivatives of sup- or inf- convolutions are understood in the viscosity sense.
So \eqref{6.6}--\eqref{6.4} and the comparison principle yield that $\un{\rho}\leq \ba{\rho}$ in $ [0,4\beta)\times\R^N$. 

By the assumption, if $r_0>0$ is sufficiently small depending on $u_0$, we have that  $\rho_i(0,x)\not\equiv 0$ with $i=1,2$. We fix one such $r_0\in (0,1)$. We claim that for any $p>1$ there exists $C=C( p)>0$ such that
\beq\lb{6.5}
\ba\rho(t,x)^p\leq C\un\rho(t,x)\quad \forall\, (t,x)\in [\beta,2\beta]\times\R^N.
\eeq
Let $R\geq 1$ be such that $u_0(\cdot)$ is supported inside $B(0,R)$. 
Then $\rho_i(0,\cdot)$ with $i=1,2$ are supported in $B(0,R+1)$. Note that
\[
\rho_i(t,x):=(4\pi t)^{-N/2}\int_{\R^N} e^{-\frac{|x-y|^2}{4t}}\rho_i(0,y)dy,
\]
and  when $|x|\geq LR$ for $L\geq 4$ and $|y|\leq R+1$, we have
\[
e^{-\frac{|x|^2(1+2/L)^2}{4t}}\leq e^{-\frac{|x-y|^2}{4t}}\leq e^{-\frac{|x|^2(1-2/L)^2}{4t}}.
\]
This implies that for $t\in [\beta,2\beta]$, we have
$$
\rho_1 (t, x)  \ge  (4\pi t)^{-\frac{N}2}e^{-\frac{|x|^2(1+2/L)^2}{4t}} \int_{|y|\le R+1}\rho_1(0, y) \, dy
$$
and
\[
 \rho_2 (t, x)^p \le (4\pi t)^{-\frac{Np}2}e^{-\frac{p|x|^2(1-2/L)^2}{4t}} \Big(\int_{|y|\le R+1}\rho_2(0, y) \, dy\Big)^p.
\]
Thus, by picking $L\gg1$ such that $p\geq (\frac{L+2}{L-2})^2$, we get for all $|x|\geq LR$ and $t\in [\beta,2\beta]$ that
\begin{align*}
\rho_2(t,x)^p&\leq  Ce^{-\frac{|x|^2(1+2/L)^2}{4t}} \Big(\int_{|y|\le R+1}\rho_2(0, y) \, dy\Big)^p\\
&\leq Ce^{-\frac{|x|^2(1+2/L)^2}{4t}} \int_{|y|\le R+1}\rho_1(0, y) \, dy\leq C \inf_{y\in B(x,r)}\rho_1(t,x) 
\end{align*}
for some $C>1$ depending on $r_0$, $u_0$ (so, also $\rho_1,\rho_2$), $\beta$ and $p$.

If $|x|\leq LR$, since $\rho_i$ are strictly positive, the same holds with possibly a larger $C$ in the compact set $[\beta,2\beta]\times B(0,LR)$.
Overall, we can find $C>0$ such that for $t\in [\beta,2\beta]$,
\[
\ba\rho(t,x)^{p}=\inf_{y\in B(x,r(t))}
\rho_2(t,y)^{p}\leq  C\sup_{y\in B(x,r(t))}
\rho_1(t,y)= C\un\rho(t,x)
\]
which yields \eqref{6.5}. 

Similarly, since $\zeta(0,\cdot)$ is compactly supported, the same argument yields that
\beq\lb{6.7}
\widehat\zeta(t,x)^{p}\leq C \un\rho(t,x)\quad\text{ for }(t,x)\in [\beta,2\beta]\times \R^N,
\eeq
where $\widehat\zeta$ is the unique solution to the heat equation with initial data $\zeta_0$.

\medskip

Below we use $\un{\rho}$ and $\ba{\rho}$ to show that $\zeta^p\lesssim u$ in a positive finite time interval. In the proof, we need to estimate $u$ from above. So, as a by-product, we also obtain that $u^p\lesssim \zeta$ in the short time. 

Since $u_0\in X_1^+$ and $v_0\in X_1^+\cap C_{\rm unif}^{2+\alpha, b}(\R^N)$, by Lemma  \ref{derivative-boundedness-lm} and Remark \ref{lm2.2-rk}, there exist $\chi_1\in (0,1)$ and $A=A(\|v_0\|_{X_1},\|u_0\|_\infty)>0$ such that as long as $|\chi|\leq \chi_1$, we have
\beq\lb{999}
\sup_{t\ge 0}\Big( \|\nabla \zeta (t,\cdot)\|_\infty
+\|\Delta\zeta(t,\cdot)\|_\infty\Big)=\sup_{t\geq 0} \Big(\|\nabla v(t,\cdot)\|_\infty +\|\Delta v(t,\cdot)\|_\infty\Big)\leq A.
\eeq

\begin{lem}\lb{L.8.1}
Assume $|\chi|\le \chi_1$.  There exists $\beta\in (0,1]$ such that for any $p>1$ we can find $L\geq 4$ such that for all $(t,x)\in [\beta,2\beta]\times \R^N$ we have
\[
\zeta(t,x)^{p}\leq Lu(t,x)\quad\text{and}\quad u(t,x)^{p}\leq L\zeta(t,x).
\]
\end{lem}
\begin{proof}
Take $A$ and $r_0$ as the above, and let 
\beq\lb{5.10}
 \beta=\min\{1,r_0/(4A)\}.
\eeq

Let $\un{\rho} $ and $\ba{\rho}$ be defined as the above in $[0,4\beta)\times\R^N$. Let $\ba\varphi(t,x):=e^{Mt}\ba\rho(t,x)$ for some $M>0$ to be determined. Then, by \eqref{6.4},
\[
\ba\varphi_t-\Delta\ba\varphi\geq M\ba\varphi+(4\beta)^{-1} r_0|\nabla\ba\varphi|.
\]
It follows that
\[
\ba\varphi_{t}-\Delta\ba\varphi - \chi\nabla\cdot (\ba\varphi \nabla\zeta)- \ba\varphi(a-b\ba\varphi^\sigma)\geq M\ba\varphi+((4\beta)^{-1}r_0-A|\chi|)|\nabla\ba\varphi|-A\chi\ba\varphi-a\ba\varphi,
\]
and so, using \eqref{5.10} and taking $M\geq A+a\ge A |\chi|+a $,
$\ba\varphi$ is a supersolution to the equation satisfied by $u$. Also recall \eqref{6.6},  the comparison principle yields $\ba\varphi\geq u$.

On the other hand, let $\un\varphi(t,x):=e^{-Mt}\un\rho(t,x)$ and we view $-(a-bu^\sigma)$ as a function of $(t,x)$ that is bounded from above by  $C>0$. Using \eqref{6.3}, we have
\[
\un\varphi_{t}-\Delta\un\varphi - \chi\nabla\cdot(\un\varphi\nabla\zeta)- \un\varphi(a-bu^\sigma)\leq -M\un\varphi-((4\beta)^{-1}r_0-A|\chi|)|\nabla\un\varphi|+A|\chi|\un\varphi+C\un\varphi\leq 0,
\]
after further assuming $M\geq A+C\ge A|\chi|+C)$.  Then $\un\varphi$ is a subsolution to the above linear equation and the comparison principle yields $\un\varphi\leq u$.
Overall, we obtain for all $t\in [0,2\beta]$,
\beq\lb{8.3}
 e^{-2M}\un\rho(t,x)\leq u(t,x)\leq e^{2M} \ba \rho(t,x).
\eeq

Now we estimate $\zeta$. Let us start with the upper bound. Recall that $\widehat\zeta$ is defined as the solution to the heat equation with initial data $\zeta_0$. We claim that $\zeta\leq \widehat\zeta+te^{2M}\ba\rho=:\ba\zeta$ in $[0,2\beta]\times\R^N$. This is because $\zeta(0,\cdot)=\widehat\zeta(0,\cdot)=\ba\zeta(0,\cdot)$ and, by \eqref{6.4} and \eqref{8.3},
\[
\ba\zeta_t-\Delta\ba\zeta\geq e^{2M}\ba\rho\geq u,
\]
and $\zeta$ satisfies $\zeta_t-\Delta\zeta=u(1-\zeta)\leq u$.

For the lower bound, take
\[
\un\zeta(t,x):=\alpha t \un\rho(t,x)\qquad\text{with $\alpha\in (0,2^{-1}e^{-2M})$}.
\]
Then $\un\zeta(0,\cdot)\equiv 0\leq \zeta_0$. Let us fix $\alpha>0$ to be sufficiently small such that  $\un\zeta(t,\cdot)\leq \frac12$ for $t\in [0,2\beta]$. Within the time interval, by \eqref{6.3}, \eqref{8.3} and $\alpha<e^{-2M}/2$, we get
\[
\un\zeta_t-\Delta\un\zeta\leq \alpha\un\rho\leq 2^{-1}u\leq u(1-\un\zeta).
\]
Therefore, the comparison principle yields
$\zeta(t,\cdot)\geq \un\zeta$ for  $(t,x)\in [0,2\beta]\times\R^N$. Overall, we obtain for $(t,x)\in [0,2\beta]\times\R^N$,
\[
t\alpha\un\rho(t,x)\leq \zeta(t,x)\leq \widehat\zeta(t,x)+te^{2M}\ba\rho(t,x),
\]
which, combining with \eqref{6.5} and \eqref{6.7}, yields 
\[
C^{-1}\un\rho(t,x)^{p}\leq \zeta(t,x)^{p}\leq C\un\rho(t,x)\quad\forall  (t,x)\in [\beta,2\beta]\times\R^N,\,\text{ for some }C>1.
\]

Finally, since \eqref{6.5} and \eqref{8.3} imply
\[
C^{-1}\un\rho(t,x)^{p}\leq u(t,x)^{p}\leq C\un\rho(t,x)\quad\forall  (t,x)\in [\beta,2\beta]\times\R^N,\,\text{ for some }C>1,
\]
the conclusion follows immediately.
\end{proof}

Below we prove that $\zeta^{p}\lesssim u$ for all $t\geq\frac12$. 


\begin{lem}\lb{L.8.3}   
Let $\beta,L=L(\beta,p)$ from Lemma \ref{L.8.1} with some $p>1$, and let $\chi_1,A$ from \eqref{999}. 
If  $|\chi|\leq \min\{\frac{a}{2A},\chi_1\}$ and $-\chi\leq p-1$,
then for all $(t,x)\in [ \beta,\infty)\times\bbR^N$,
\beq\lb{8.10}
\zeta(t,x)^p\leq Mu(t,x)\quad \text{where } M:=\max\left\{L,\,{4p}/{a},\,2(4b/a)^{1/\sigma}\right\}.
\eeq
\end{lem}

\begin{proof}
Let us consider $w:=\frac{\zeta^{p}}{u}$, which is well-defined and $\leq L$ at least for $t\in [\beta,2\beta]$ by Lemma \ref{L.8.1}. 
Since the solutions are smooth for all positive times, it suffices to prove a priori estimate that $w$ stays uniformly bounded for all $t\geq \beta$.

By direct computation,
\begin{align*}
\nabla w&=\nabla\left(\frac{\zeta^{p}}{u}\right)=\frac{u\nabla (\zeta^{p})-\zeta^{p} \nabla u}{u^2},
\\
\Delta w&=\Delta\left(\frac{\zeta^{p}}{u}\right)=\frac{u\Delta (\zeta^{p})-\zeta^{p} \Delta u}{u^2}-\frac{2\nabla u(u\nabla (\zeta^{p})-\zeta^{p}\nabla u)}{u^3}\\
&=\frac{ pu\zeta^{p-1}\Delta \zeta-\zeta^{p}\Delta u}{u^2}+\frac{ p(p-1)w|\nabla \zeta|^2}{\zeta^2}-\frac{2\nabla u\cdot \nabla w}{u},\\
\frac{\zeta^{p}\nabla\cdot (u\nabla \zeta)}{u^2}&= \frac{\zeta^p\Delta\zeta}u+\frac{\zeta^p\nabla u\cdot\nabla \zeta -u \nabla (\zeta^p)\cdot \nabla \zeta }{u^2}+\frac{u \nabla (\zeta^p)\cdot \nabla \zeta }{u^2}\\
&=w\Delta \zeta-\nabla w \cdot \nabla \zeta +\frac{pw|\nabla \zeta| ^2}{\zeta},
\end{align*}
and
\[
\frac{\zeta^p(a-bu^{\sigma})-pu\zeta^{p-1}(1-\zeta)}{u}=aw-bu^\sigma w-p\zeta^{p-1}(1-\zeta).
\]
Also by the equations, we have
\begin{align*}
w_{t}&=\frac{\zeta^{p}_t u-\zeta^{p} u_t}{u^2}=\frac{pu\zeta^{p-1}(\Delta \zeta +u(1-\zeta))-\zeta^{p}(\Delta u+\chi\nabla \cdot(u\nabla \zeta )+u(a-bu^\sigma))}{u^2}\\
&=\frac{pu\zeta^{p-1}\Delta \zeta -\zeta^{p}\Delta u}{u^2}-\frac{\chi\zeta^{p}\nabla \cdot(u\nabla \zeta)}{u^2}-\frac{ \zeta^{p}(a-bu^{\sigma})-p\zeta^{p-1}u(1-\zeta)}{u}.
\end{align*}
Putting these together, we obtain that $w$ satisfies
\begin{align*}
w_{t}&= \Delta w  +\frac{2\nabla u\cdot\nabla w }{u} -\frac{ p(p-1)w|\nabla \zeta|^2}{\zeta^2}-\chi\left(w\Delta \zeta-\nabla \zeta\cdot  \nabla w +\frac{pw |\nabla \zeta |^2}{\zeta}\right)\\
&\qquad\qquad-(a-bu^\sigma) w+p\zeta^{p-1}(1-\zeta).
\end{align*}
Due to $|\chi \Delta \zeta|\leq\frac{a}{2}$ and $-\chi\leq p-1$ by the assumption, and $\zeta\in [0,1]$, we get
\begin{equation}\lb{8.5}
w_{t}\leq  \Delta w  +\frac{2\nabla u\cdot\nabla w }{u} - \chi\nabla \zeta\cdot  \nabla w -\left(\frac{a}2-bu^\sigma\right) w+p.
\end{equation}

Note that if $bu^\sigma\geq \frac{a}4$, we have $u\geq (\frac{a}{4b})^{1/\sigma}=:c_*$ and $w=\frac{\zeta^p}{u}\leq \frac{1}{c_*}$, and otherwise, we have $\left(\frac{a}2-bu^\sigma\right)\geq \frac{a}{4}$.
Therefore,
\[
\left({a}/2-bu^\sigma\right) w\geq {aw}/{4}-bu^\sigma w \psi(w),
\]
where $\psi$ is Lipschitz continuous and satisfies $\psi(w)\in [\mathbbm{1}_{\{w\leq 1/c_*\}},\mathbbm{1}_{\{w\leq 2/c_*\}}]$ with $\mathbbm{1}$ denoting the characteristic function.
We deduce from \eqref{8.5} that
\[
w_{t}\leq  \Delta w  +\frac{2\nabla u\cdot\nabla w }{u} - \chi\nabla \zeta\cdot  \nabla w -\frac{aw}{4}+bu^\sigma w \psi(w)+p.
\]
Recall that $w(\beta,\cdot)\leq L$  by Lemma \ref{L.8.1} and $u$ is uniformly finite. We can compare $w$ with the solution to the following ODE ($z=z(t)$)
\[
\frac{d}{dt}z=-\frac{az}4+ b\|u\|_\infty^\sigma z \psi(z) +p,\quad z(\beta)=L,
\]
to get for all $x\in\R^N$ and $t\geq\beta$,
\[
w(t,x)\leq z(t)\leq \max\left\{L,{4p}/a,2/c_*\right\}=M.
\]
which implies \eqref{8.10}.
\end{proof}

After obtaining the estimate $\zeta^p\leq Mu$, we are able to bound $|\nabla\zeta|$ and $|D^2\zeta|$ in terms of $u$.

\begin{lem}\lb{L.8.2}
Assume $|\chi|\leq \chi_1$. For any $p'>1$,  there exists  $C=C(p')>0$ such that the following holds. For all $(t,x)\in [1,\infty)\times\bbR^N$ we have
\[
|\nabla u(t,x)|\leq C u(t,x)^\frac{1}{p'}.
\]

Under the assumptions of Lemma \ref{L.8.3} and for $p,M$ from the lemma,
we have for all $(t,x)\in [3,\infty)\times\bbR^N$,
\[
|\nabla\zeta(t,x)|,\quad|D^2\zeta(t,x)|\leq CM^\frac1p u(t,x)^\frac{1}{p'p}.
\]
\end{lem}
\begin{proof}
Since $|\chi|\leq \chi_1$, the first claim follows from Lemma \ref{asymptotic-lm1}.

The estimate for $|\nabla\zeta|$ follows  similarly. 
Indeed, by the equation of $\zeta$, the local $L^q$-parabolic estimates (see for e.g., \cite[Theorem 7.22]{Lieberman}) yields for $t\geq2$ and any $x\in\R^N$,
\[
\begin{aligned}
    \|D^2\zeta\|_{L^{N+3}([t-\frac12,t]\times B(x,\frac12))}&\leq C\left(\|u(1-\zeta)\|_{L^{N+3}([t-1,t]\times B(x,1))}+\|\zeta\|_{L^{N+3}([t-1,t]\times B(x,1))}\right)\\
&
\leq C\left(\|u\|_{L^{N+3}([t-1,t]\times B(x,1))}+\|\zeta\|_{L^{N+3}([t-1,t]\times B(x,1))}\right).
\end{aligned}
\]
The anisotropic Sobolev embedding (\cite[Lemma A3]{Engler}) yields for any $(t,x)\in [2,\infty)\times\R^N$,
\beq\lb{8.4}
|\nabla\zeta(t,x)|
\leq C\left(\|u\|_{L^{N+3}([t-1,t]\times B(x,1))}+\|\zeta\|_{L^{N+3}([t-1,t]\times B(x,1))}\right).
\eeq
Since $t\geq 2$, applying \eqref{8.10} and Lemma \ref{asymptotic-lm1} (with $R=1,s_0\in[0,1]$ and $p=p'$), we get for some $C=C(p')$,
\[
|\nabla\zeta(t,x)|\leq  C\left(\|u\|_{L^{N+3}([t-1,t]\times B(x,1))}+M^\frac1p\|u^\frac1p\|_{L^{N+3}([t-1,t]\times B(x,1))}\right)
\leq CM^\frac1p  u(t,x)^\frac1{p'p}.
\]


For the last claim, note that $q:=\nabla\zeta$ satisfies
\[
q_t-\Delta q+u q=(\nabla u)(1-\zeta).
\]
So Theorem 7.22 in \cite{Lieberman}, \eqref{8.10} and \eqref{8.4} yield for $t\in [3,\infty)$,
\begin{align*}
&\|D^2q\|_{L^{N+3}([t-\frac12,t]\times B(x,\frac12))}\leq C\left(\|\nabla u\|_{L^{N+3}([t-1,t]\times B(x,1))}+\|q \|_{L^{N+3}([t-1,t]\times B(x,1))}\right)\\
&\qquad\qquad\leq C\left(\|\nabla u\|_{L^{N+3}([t-1,t]\times B(x,1))}+ \|u\|_{L^{N+3}([t-2,t]\times B(x,2))}+M^\frac1p \|u^\frac{1}{p}\|_{L^{N+3}([t-2,t]\times B(x,2))}\right).
\end{align*}
Again by the anisotropic Sobolev embedding and Lemma \ref{asymptotic-lm1}, there exists $C=C(p')$ such that for $t\in [3,\infty)$,
\[
|D^2\zeta(t,x)|\leq C\|D^2q\|_{L^{N+3}([t-\frac12,t]\times B(x,\frac12))}\leq CM^\frac{1}{p}  u(t,x)^\frac{1}{p'p}.
\]
\end{proof}

Now, we prove Theorem \ref{speed-thm}(2) and we recall that here $\sigma=1$ and $\chi<0$.

\begin{proof}[Proof of Theorem \ref{speed-thm}(2)]

First,  let
$$
w=u-\frac{\chi}{2}|\nabla \zeta |^2.
$$
Note that
$$
\nabla \zeta \cdot \nabla (\Delta \zeta)=\frac{1}{2}\Delta |\nabla \zeta|^2 -|D^2\zeta|^2.
$$
Then
$w$ satisfies
\begin{equation*}
w_t=\Delta w+\chi u\Delta \zeta +\chi|D^2\zeta|^2+\chi\zeta  \nabla u\cdot\nabla \zeta+\chi u|\nabla \zeta|^2 +u(a-b u).
\end{equation*}
Write $\tilde\chi=-\chi>0$, and by Young's inequality we get
\begin{align*}
w_t&=\Delta w-\tilde \chi u\Delta \zeta-\tilde \chi| D^2\zeta|^2-\tilde  \chi\zeta  \nabla u \cdot \nabla \zeta -\tilde \chi u |\nabla \zeta|^2 +u(a-b u)\\
&\le \Delta w+\frac{N\tilde \chi }{4} u^2 -\tilde  \chi\zeta  \nabla u \cdot \nabla \zeta -\tilde \chi u |\nabla \zeta|^2 +a w-\frac{  a \tilde\chi}{2}|\nabla \zeta|^2 -b u^2\\
&\le  \Delta w+\frac{N\tilde \chi }{4}u^2 +\frac{\tilde\chi}{2a}\zeta^2|\nabla u|^2 -\tilde \chi u |\nabla \zeta|^2 +a w-b u^2.
\end{align*}

By the assumptions in Theorem \ref{speed-thm}(2) and  Lemmas \ref{L.8.3} and \ref{L.8.2},  there are $C>0$  and $\chi_2\in (0,\chi_1)$ such that if $|\chi|\le \chi_2$,
$$
|\nabla u|\le Cu^{\frac{1}{2}}\quad\text{and}\quad \zeta\leq Cu^\frac12\quad \text{ in }[1,\infty)\times\bbR^N.
$$
With these, $w$ satisfies
\[
w_t\leq \Delta w+\frac{N\tilde \chi}{4}u^2 +\frac{C\tilde\chi}{2a} u^{2} +a w-b u^2.
\]
which implies that if 
$\tilde \chi\le \chi_3:=\frac{2ab}{aN+2C}$,  then
$
w_t\le  \Delta w+aw-bw^2/2.
$
This shows that the spreading speed of $u$ is $c^*=2\sqrt a$ when $0<\-\chi\le \chi_0:=\min\{\chi_2,\chi_3\}$.
\end{proof}

Finally, we prove Theorem \ref{speed-thm1}. Here we consider the case of $\sigma\in (0,1)$, while allowing $\chi>0$.

\begin{proof}[Proof of Theorem \ref{speed-thm1}]

Since $\sigma<1$, it follows from the assumptions in  Theorem \ref{speed-thm1} and  Lemma \ref{L.8.2}  that there exist $C>0$ and $\chi_2\in (0,\chi_1)$ such that as long as $|\chi|\le \chi_2$, 
\[
|\nabla u|,\quad |\nabla\zeta|\leq C u^{\frac{1+\sigma}{2}}\quad \text{and}\quad |D^2\zeta|\leq C u^{\sigma} \quad\text{ in } [3,\infty)\times\bbR^N.
\]
Hence, from the equation, we get for $t\geq 3$,
\begin{align*}
u_t&=\Delta u + \chi \nabla u\cdot\nabla \zeta+\chi u\Delta \zeta+ u(a-bu^\sigma)\\
&\leq\Delta u+au-bu^{1+\sigma}+C|\chi|u^{1+\sigma}  \leq \Delta u +au,
\end{align*}
provided that $|\chi|\leq \chi_0:=\min\{\chi_2, \frac{b}{C}\}$.
This implies that the spreading speed of $u$ is  $c^*=2\sqrt{a}$. 
\end{proof}

\section{Numerical Simulation}
\label{simulation}

In this section, we present the numerical results on the spreading speed of \eqref{main-Eq} in the one dimensional setting, i.e., $N=1$, with the initial data having compact support. 
This investigation is accomplished through numerical simulations of the solutions to \eqref{num03}+\eqref{BC}. All the numerical simulations were implemented using the Python programming language.

We start by defining the scheme for \eqref{num03} and \eqref{BC} as follows: We divide the space interval $[-L, L]$  into $M$ subintervals with equal length and divide the time interval $[0, T ]$ into $N$ subintervals with equal length. Then the space step size is $h = \frac{2L}{M}$ and  the time step size
is $\tau^*= \frac{T}{N}$. For simplicity, we denote the approximate value of $\tilde u(t_j , x_i), \tilde v(t_j, x_i)$ by $\tilde u(j, i), \tilde v(j , i)$ respectively, with $t_j = (j - 1)\tau^*, \, 1 \le j \le N + 1$ and $x_i = -L + (i - 1)h, \,
1 \le i \le M + 1$.
Using the central approximation for the spatial derivatives $\tilde v_{xx} (t_j , x_i), \tilde v_{x} (t_j , x_i) $ and $\tilde u_{xx} (t_j , x_i), \tilde u_{x} (t_j , x_i)$:
\begin{align*}
\tilde v_{x} (t_j , x_i ) &\approx \frac{\tilde v(j , i + 1) - \tilde v(j , i - 1)}{2h},\\    \tilde v_{xx} (t_j , x_i ) &\approx \frac{\tilde v(j , i - 1) - 2\tilde v(j , i) + \tilde v(j , i + 1)}{h^2},\\
     \tilde u_{x} (t_j , x_i ) &\approx\frac{\tilde u(j , i + 1) - \tilde u(j , i - 1)}{2h},\\
     \tilde u_{xx} (t_j , x_i ) &\approx\frac{\tilde u(j , i - 1) - 2\tilde u(j , i) + \tilde u(j , i + 1)}{h^2}.
\end{align*}
The forward approximation of the time derivative yields
\begin{align*}
    \tilde v_{t} (t_j , x_i ) \approx\frac{\tilde v(j +1, i) - \tilde v(j , i)}{\tau^*},\qquad \tilde u_{t} (t_j , x_i ) \approx \frac{\tilde u(j +1, i) - \tilde u(j , i)}{\tau^*}.
\end{align*}
By the boundary conditions in \eqref{BC}, we set
\begin{equation}\label{u-BC}
   \tilde u(j, 1) = \tilde u(j, M+1) = 0.
\end{equation}
Using the forward approximation for the Neumann boundary condition at $-L$ and backward approximation at $L$,  we get
$$
    \frac{\partial\tilde v}{\partial x}(t_j , x_1 ) \approx \frac{\tilde v(j, 2) - \tilde v(j , 1)}{\tau^*} = 0, \qquad
     \frac{\partial\tilde v}{\partial x}(t_j , x_{M+1} ) \approx\frac{\tilde v(j, M+1) - \tilde v(j , M)}{\tau^*} = 0,
$$
and hence, we set
\begin{equation}
\label{v-BC}
  v(j, 2) = v(j, 1),\quad
     v(j, M+1) = v(j, M).
\end{equation}
The equation \eqref{num03} can thus be discretized as: for $ 1 \le j \le N, 2\le i\le M$,
\begin{align*}
&    \frac{\tilde u(j +1, i) - \tilde u(j , i)}{\tau^*} =\frac{\tilde u(j , i - 1) - 2\tilde u(j , i) + \tilde u(j , i + 1)}{h^2} \\
    &\qquad\qquad+ \left(c - \chi \frac{\tilde v(j , i + 1) - \tilde v(j , i - 1)}{2h}\right)  \frac{\tilde u(j , i + 1) - \tilde u(j , i - 1)}{2h}\\
    &\qquad\qquad - \chi \tilde u(j, i)\frac{\tilde v(j , i - 1) - 2\tilde v(j , i) + \tilde v(j , i + 1)}{h^2} + \tilde u(j, i)(1-\tilde u(j, i)),
\end{align*}
and
\[
\begin{aligned}
   {\tau}\frac{\tilde v(j +1, i) - \tilde v(j , i)}{\tau^*} &=\frac{\tilde v(j , i - 1)  - 2\tilde v(j , i) + \tilde v(j , i + 1)}{h^2} + c \tau \frac{\tilde v(j , i + 1)- \tilde v(j , i - 1)}{2h}\\
   &\quad- \tilde u(j, i)\tilde v(j, i).
\end{aligned}
\]
Simplifying and reordering of the two equations, we get for $1 \le j \le N, 2\le i\le M$,
\begin{align}
\label{u-scheme}
   \tilde u(j +1, i)  &=\tilde u(j , i)\left[1 + \tau^* - \frac{2\tau^*}{h^2} - \frac{\tau^*\chi}{h^2}(v(j , i - 1) - 2\tilde v(j , i) + \tilde v(j , i + 1)) \right]\nonumber \\
    &\quad +\tilde u(j , i - 1)\left[\frac{\tau^*}{h^2} - \frac{\tau^*}{2h}\left(c - \chi \frac{\tilde v(j , i + 1) - \tilde v(j , i - 1)}{2h}\right) \right] \nonumber\\
    &\quad +  \tilde u(j , i + 1) \left[\frac{\tau^*}{h^2} + \frac{\tau^*}{2h}\left(c - \chi \frac{\tilde v(j , i + 1) - \tilde v(j , i - 1)}{2h}\right) \right] - \tau^*\tilde u(j, i)^2 ,
\end{align}
and
\begin{align}
\label{v-scheme}
&\tilde v(j +1, i)\nonumber\\
&=\tilde v(j , i) (1- \frac{\tau^*}{\tau} u(j, i) - \frac{2\tau^*}{{\tau}h^2}) + \tilde v(j , i - 1)\left( \frac{\tau^*}{{\tau} h^2} -\frac{c\tau^*}{2h}\right)+  \tilde v(j , i + 1)\left( \frac{\tau^*}{{\tau} h^2} +\frac{c\tau^*}{2 h}\right).
\end{align}

We use \eqref{u-BC}, \eqref{v-BC}, \eqref{u-scheme} and \eqref{v-scheme} in implementing the scheme on python, we apply the same space step size $h = 0.1$ and the same time step size $\tau^* = 0.002.$ We carry out the numerical experiments for ${\tau = 1}$ and for different values of $\chi$ and $c$ in subsections \ref{simulation-1}-\ref{simulation-3}, and carry out the numerical experiments for different values of $\tau$, $\chi$, and $c$ in subsection \ref{simulation-4}.    In all the simulations we take $L=20$ and $T=500$.

\subsection{Numerical experiments with positive but not large $\chi$ and observations}
\label{simulation-1}
We  carry out experiments with $\chi=1$, $\chi=1.25$ and $\chi = 1.5$. For each of these $\chi$, we take $c$ to be of the following values: $c=1,1.99,2.01,3$.  Throughout this subsection, $\tau=1$.
Figures \ref{chi-1-c-1-1.99} and \ref{chi-1-c-1.99-3} are  the graphs of the $u$-component of  the numerical solutions of  \eqref{num03}+\eqref{BC}+\eqref{u0-v0}
with $\chi=1$, and $c=1,1.99,2.01, 3$  at some fixed times. Figure \ref{chi-1-c-1-v} are the graphs of the $v$-component of  the numerical solutions of  \eqref{num03}+\eqref{BC}+\eqref{u0-v0}
with $\chi=1$ and $c=1,3$  at some fixed times.

\begin{figure}[H]
    \centering
    \begin{subfigure}[b]{0.4\textwidth}
        \includegraphics[width=\textwidth]{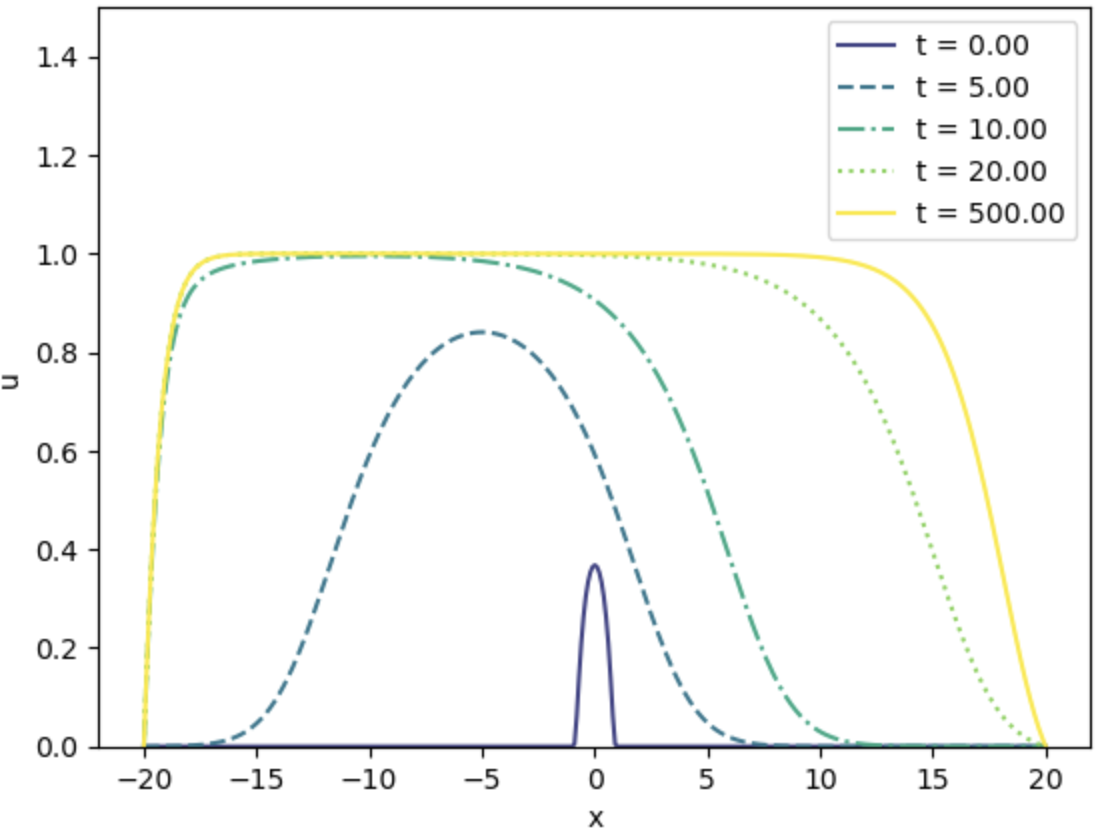}
       \caption{$c = 1$}
    \end{subfigure}
    \hfill
   \begin{subfigure}[b]{0.4\textwidth}
        \includegraphics[width=\textwidth]{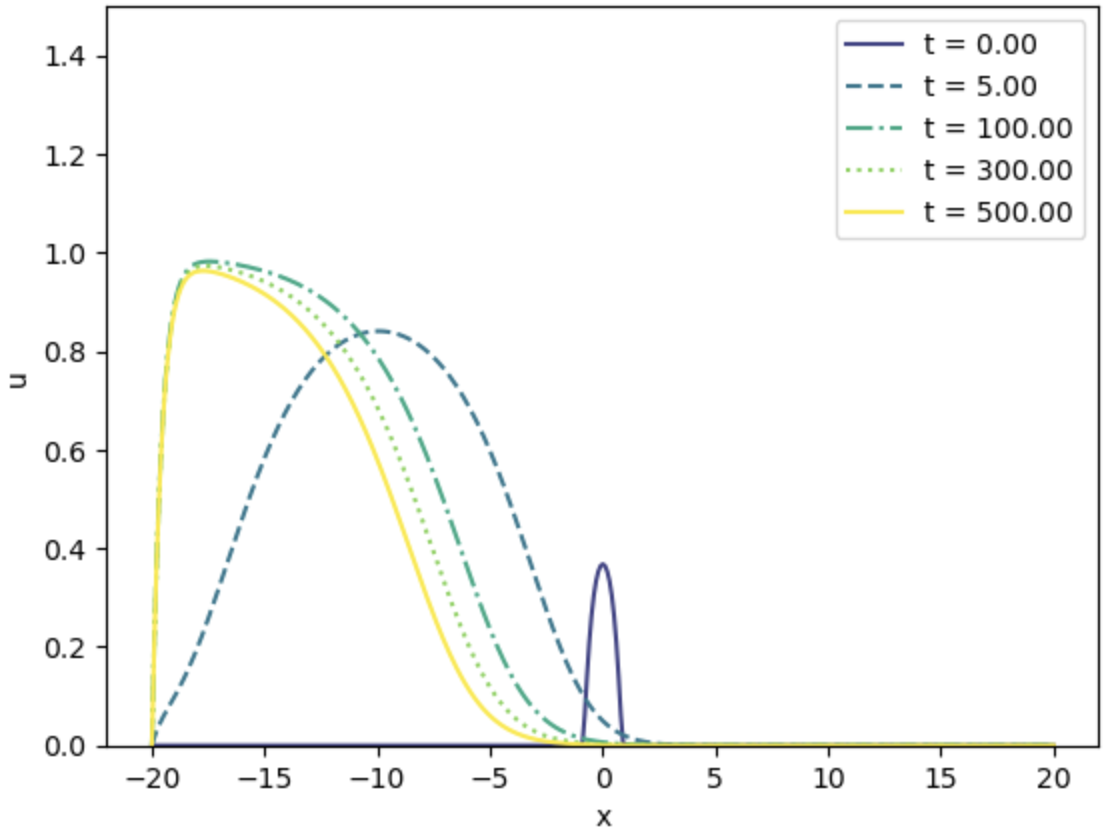}
       \caption{$ c = 1.99$}
    \end{subfigure}
    \caption{$\chi=1, \tau=1, u(t,x)$}
    \label{chi-1-c-1-1.99}
\end{figure}

\begin{figure}[H]
    \centering
    \begin{subfigure}[b]{0.4\textwidth}
      \includegraphics[width=\textwidth]{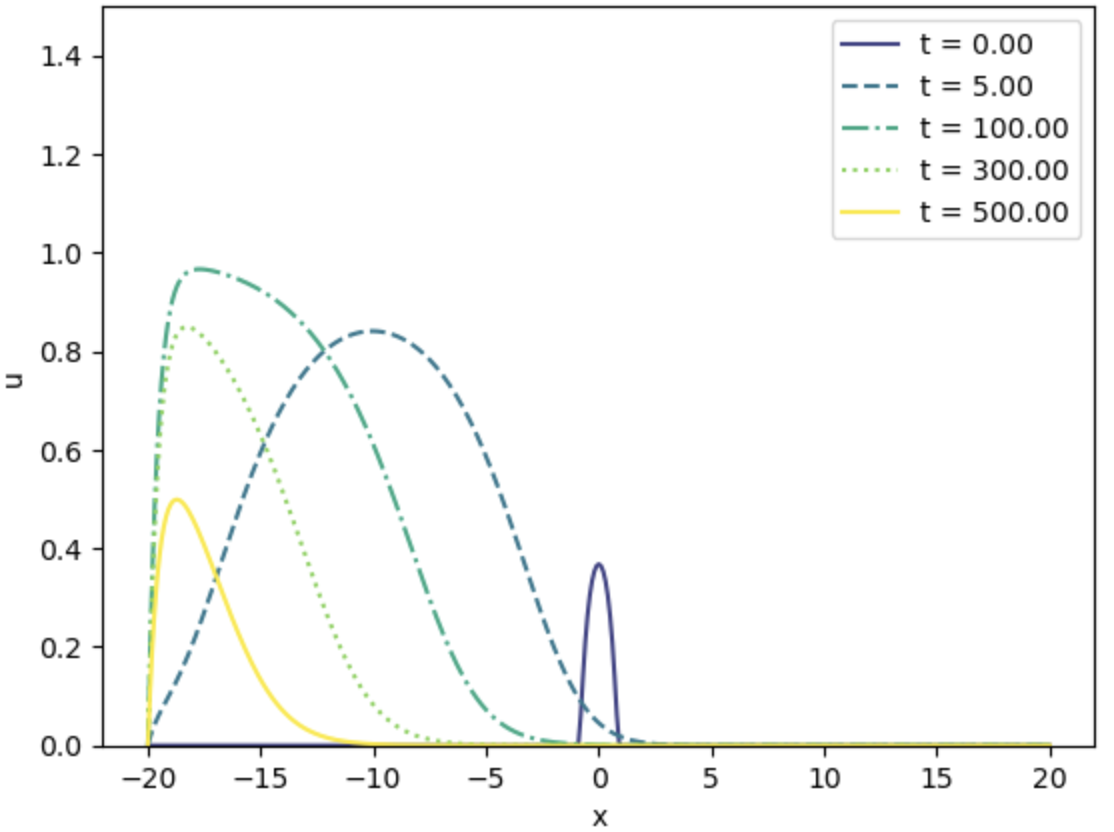}
       \caption{$c = 2.01$}
    \end{subfigure}
    \hfill
   \begin{subfigure}[b]{0.4\textwidth}
        \includegraphics[width=\textwidth]{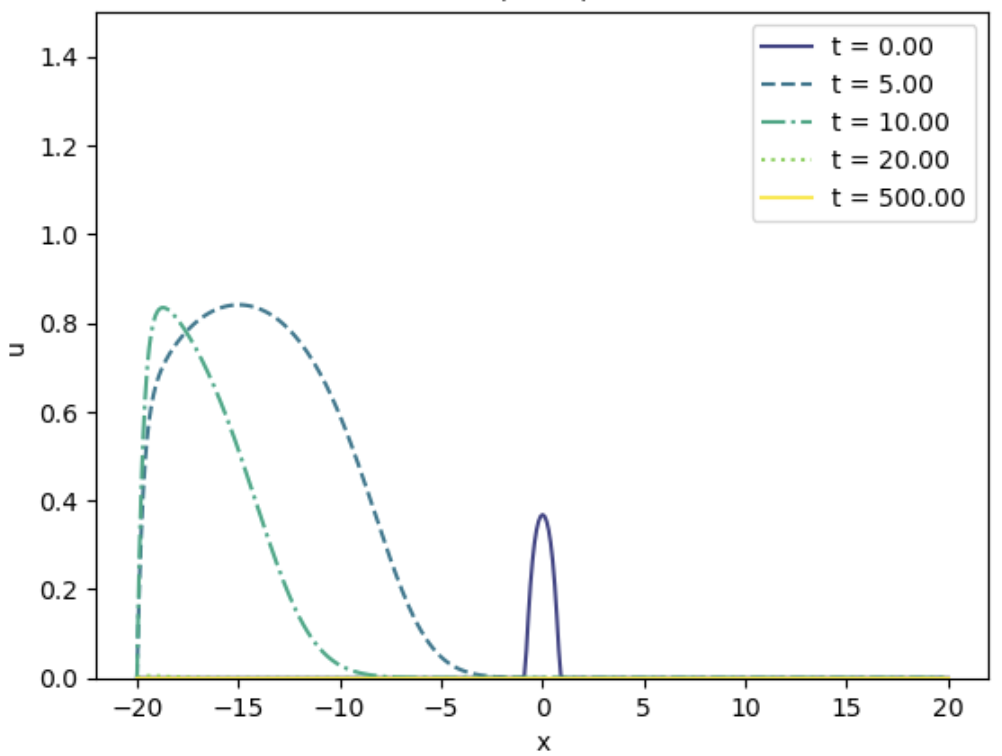}
       \caption{$ c = 3$}
    \end{subfigure}
    \caption{$\chi = 1, \tau=1, u(t,x)$}
    \label{chi-1-c-1.99-3}
\end{figure}

\begin{figure}[H]
    \centering
    \begin{subfigure}[b]{0.4\textwidth}
      \includegraphics[width=\textwidth]{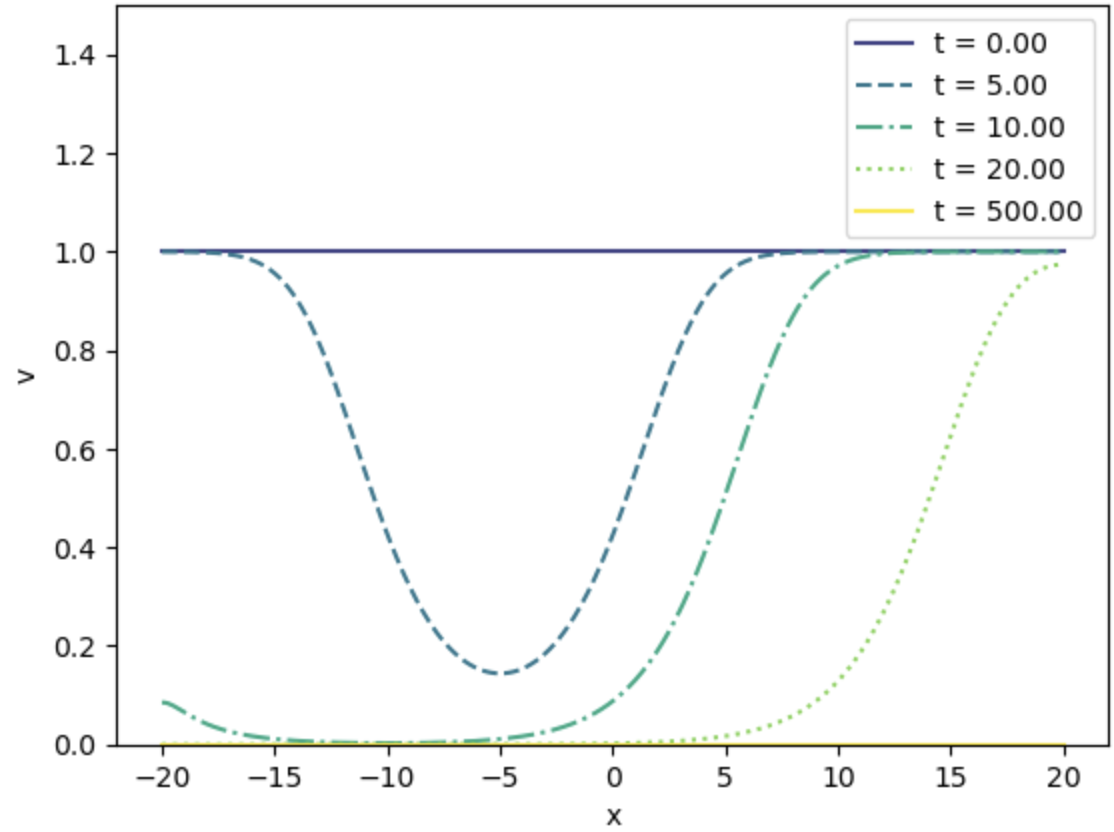}
       \caption{$c = 1$}
    \end{subfigure}
    \hfill
   \begin{subfigure}[b]{0.4\textwidth}
        \includegraphics[width=\textwidth]{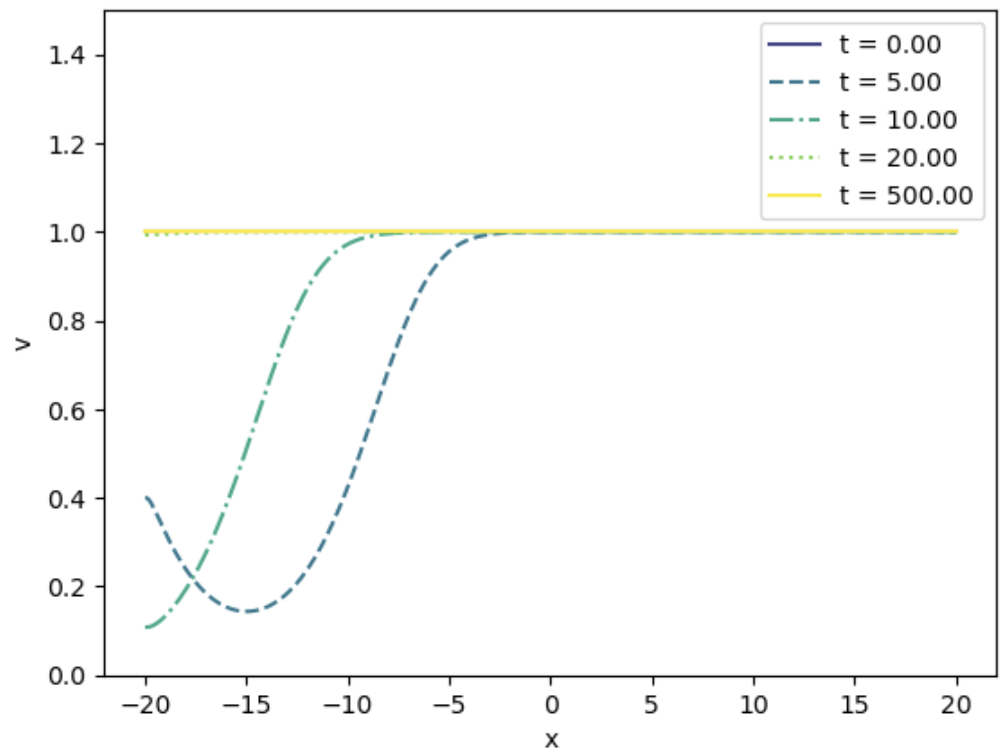}
       \caption{$ c = 3$}
    \end{subfigure}
    \caption{$\chi = 1, \tau=1, v(t,x)$}
     \label{chi-1-c-1-v}
\end{figure}

We observe the following scenarios via the numerical simulations: The results of the numerical simulations with $\chi=1.25$ and $\chi=1.5$ are similar to those with $\chi=1$. The results show that when $\tilde u(t,x)$ stays positive, $\tilde v(t,x)\to 0$ as $t\to\infty$, and when $\tilde u(t,x)\to 0$, $\tilde v(t,x)$ stays positive as $t\to\infty$ (we do not include all the graphs of the $v$-component since we can infer the behavior of $v$ from $u$). For $0<c<2$, $u(t,x)$ stays positive, and for $c>2$, $u(t,x)$ goes to $0$ as $t\to\infty$, which indicates that chemotaxis with a positive but not large sensitivity coefficient $\chi$ does not speed up the spreading.

\subsection{Numerical experiments with positive and large $\chi$ and observations}
\label{simulation-2}

We carry out some numerical experiments with $\chi=2.5, 5$ and $\chi=10$.
For each of these $\chi$, we do simulations for the following values of $c$;
 $c=1, 1.99, 2.01, 3$. Figures \ref{chi-10-c-1} and \ref{chi-10-c-3} are the graphs of the $u$-component of the numerical solutions of  \eqref{num03}+\eqref{BC}+\eqref{u0-v0}
with $\chi=10$ and $c=1,1.99,2.01, 3$.

We
observe that the results of numerical simulations with $\chi=2.5$ and $\chi = 5$ are similar to those with $\chi=10$.
We also observe that $\tilde u(t,x)$ stays positive  and $\tilde v(t,x)\to 0$ for $c<2$ as well as for $c=2.01$ as $t\to\infty$, which indicates that chemotaxis with positive large sensitivity coefficient $\chi$ speeds up the spreading.

\begin{figure}[H]
    \centering
    \begin{subfigure}[b]{0.4\textwidth}
        \includegraphics[width=\textwidth]{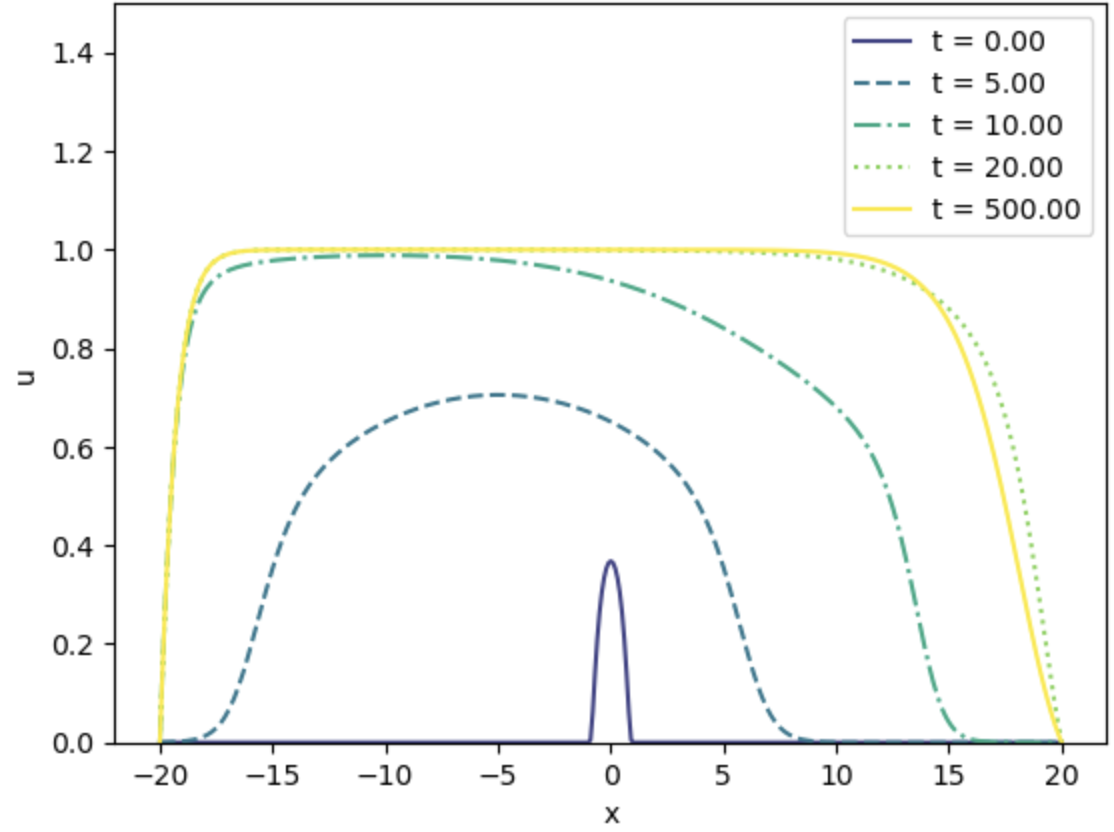}
       \caption{$c = 1$}
    \end{subfigure}
    \hfill
   \begin{subfigure}[b]{0.4\textwidth}
        \includegraphics[width=\textwidth]{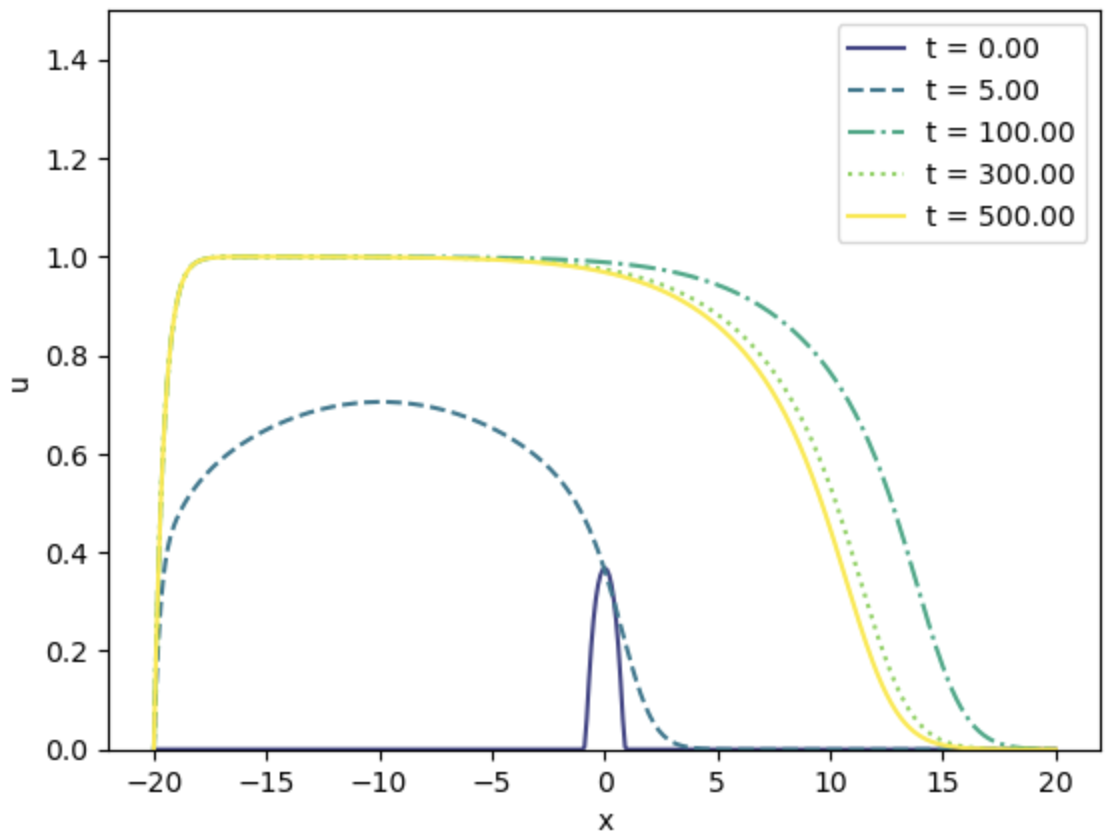}
       \caption{$ c = 1.99$}
    \end{subfigure}
    \caption{$\chi = 10, \tau=1, u(t,x)$}
    \label{chi-10-c-1}
\end{figure}

\begin{figure}[H]
    \centering
    \begin{subfigure}[b]{0.4\textwidth}
        \includegraphics[width=\textwidth]{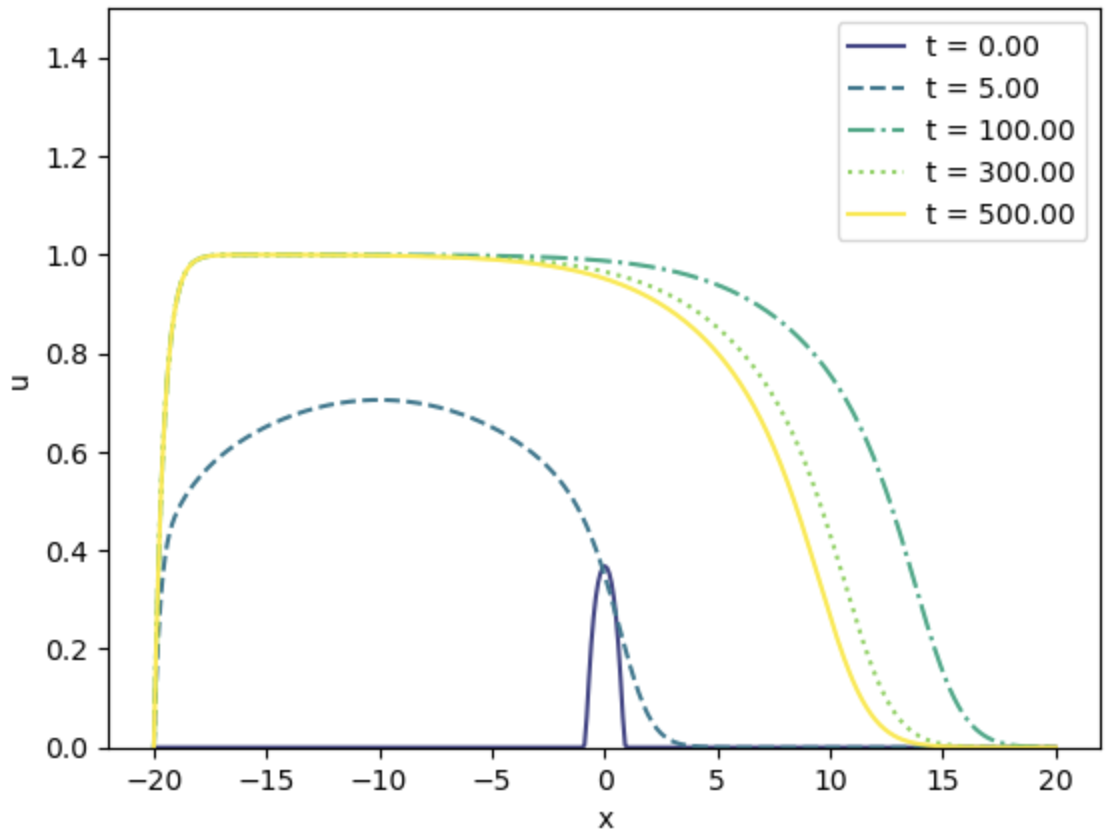}
       \caption{$c = 2.01$}
    \end{subfigure}
    \hfill
   \begin{subfigure}[b]{0.4\textwidth}
        \includegraphics[width=\textwidth]{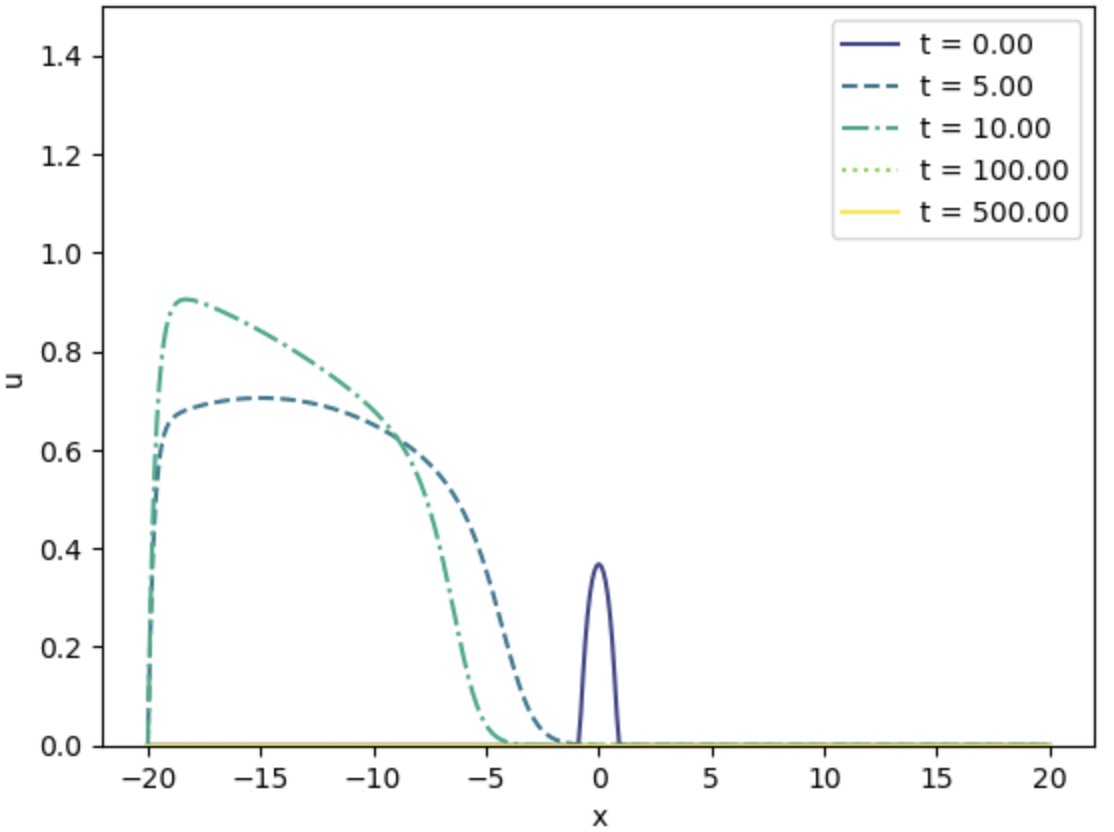}
       \caption{$c = 3$}
    \end{subfigure}
    \caption{$\chi = 10, \tau=1, u(t,x)$}
    \label{chi-10-c-3}
\end{figure}

\subsection{Numerical experiments with negative $\chi$ and observations}
\label{simulation-3}

We carry out numerical simulations with $\chi=-1$ and $\chi=-10$. For each of these $\chi$, we do simulations
for the following values of $c$, $c=1, 1.99, 2.01, 3$. Figures \ref{chi--1-c-1.99} and \ref{chi--1-c-3} are the graphs of the $u$-component of the  numerical solutions of  \eqref{num03}+\eqref{BC}+\eqref{u0-v0} for $\chi = -1$, and Figures \ref{chi--10-c-1.99} and \ref{chi--10-c-3} are   the numerical solutions of \eqref{num03}+\eqref{BC}+\eqref{u0-v0} for $\chi=-10$.

We observe that $\tilde u(t,x)$ stays positive and $\tilde v(t,x)\to 0$ as $t\to\infty$ for $0<c<2$
and $\tilde u(t,x)\to 0$ and $\tilde v(t,x)$ stays positive as $t\to\infty$, which indicates that chemotaxis with negative sensitivity coefficient or chemo-repeller does not speed up the spreading.

\smallskip

We give the following remarks on all experiments discussed in sections \ref{simulation-1}--\ref{simulation-3}.
\begin{rk}

\begin{enumerate}
    \item The simulations indicate a critical point for $\chi$, i.e., there exists $\chi^*$ {between 1.5 and 1.9} such that for $\chi<\chi^*$, the chemotaxis does not speed up the rate of spread, while for $\chi>\chi^*$, The chemotaxis speeds up the spreading.
    \item For $c = 3$, $u(t,x)\to 0$ for all the values of $\chi$. This also supports the fact that $c^*_{\rm up}<\infty$.
    \item In the above numerical experiments, we used the same space step size $h =0.1$ and the same time step size $\tau^* = 0.002$ as they satisfy the numerical stability condition $\tau^*/h^2 < 0.5$. We do not give an accuracy analysis of the simulations in this paper. To assess the reliability of the numerical results, we used different values of $h$ and $\tau^*$ to simulate the spreading speed of the chemotaxis system in $\R$. We repeat the above experiments for $h= 0.1$ and $\tau^* = 0.002, 0.004$ and for $\tau^* = 0.002$ and $h = 0.2$. The observed  outcomes are close for the different values of $h$ and $\tau^*$.
\end{enumerate}
\end{rk}

\begin{figure}[H]
    \centering
    \begin{subfigure}[b]{0.4\textwidth}
        \includegraphics[width=\textwidth]{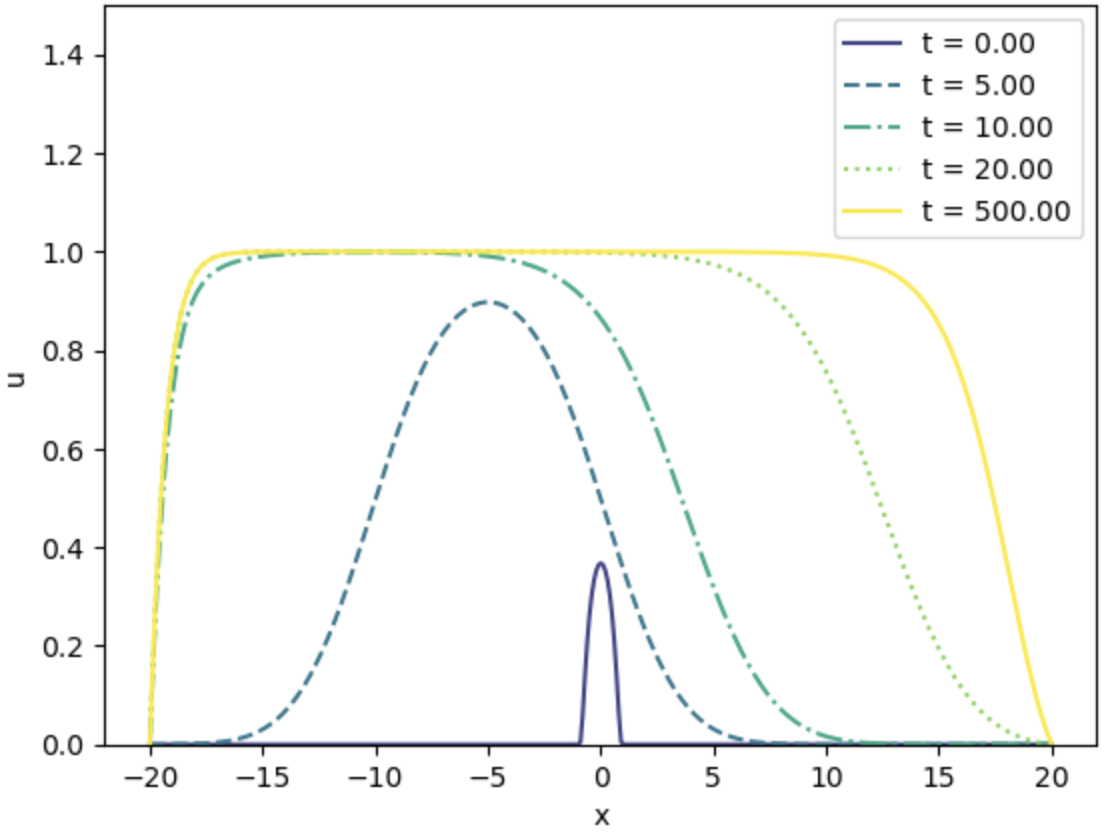}
        \caption{$c=1$}
    \end{subfigure}
    \hfill
    \begin{subfigure}[b]{0.4\textwidth}
        \includegraphics[width=\textwidth]{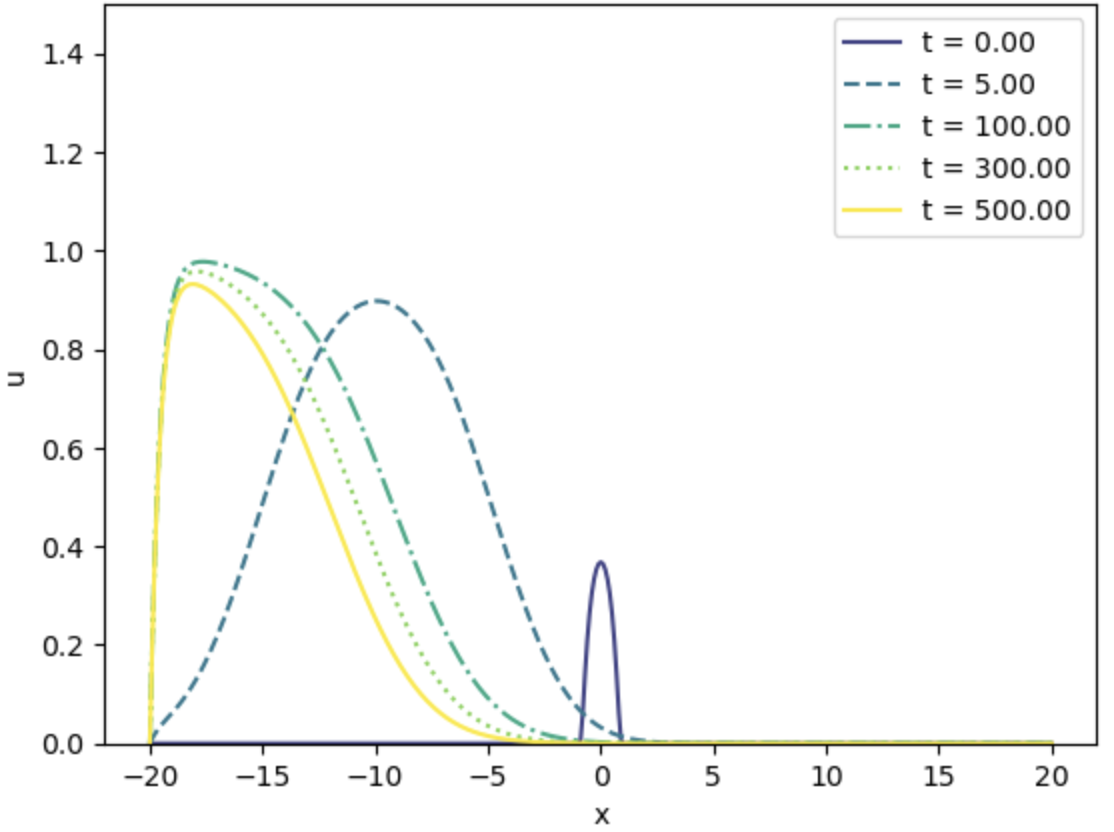}
        \caption{$c=1.99$}
    \end{subfigure}
   \caption{$\chi = -1, \tau=1, u(t,x)$}
    \label{chi--1-c-1.99}
\end{figure}

\begin{figure}[H]
    \centering
    \begin{subfigure}[b]{0.4\textwidth}
        \includegraphics[width=\textwidth]{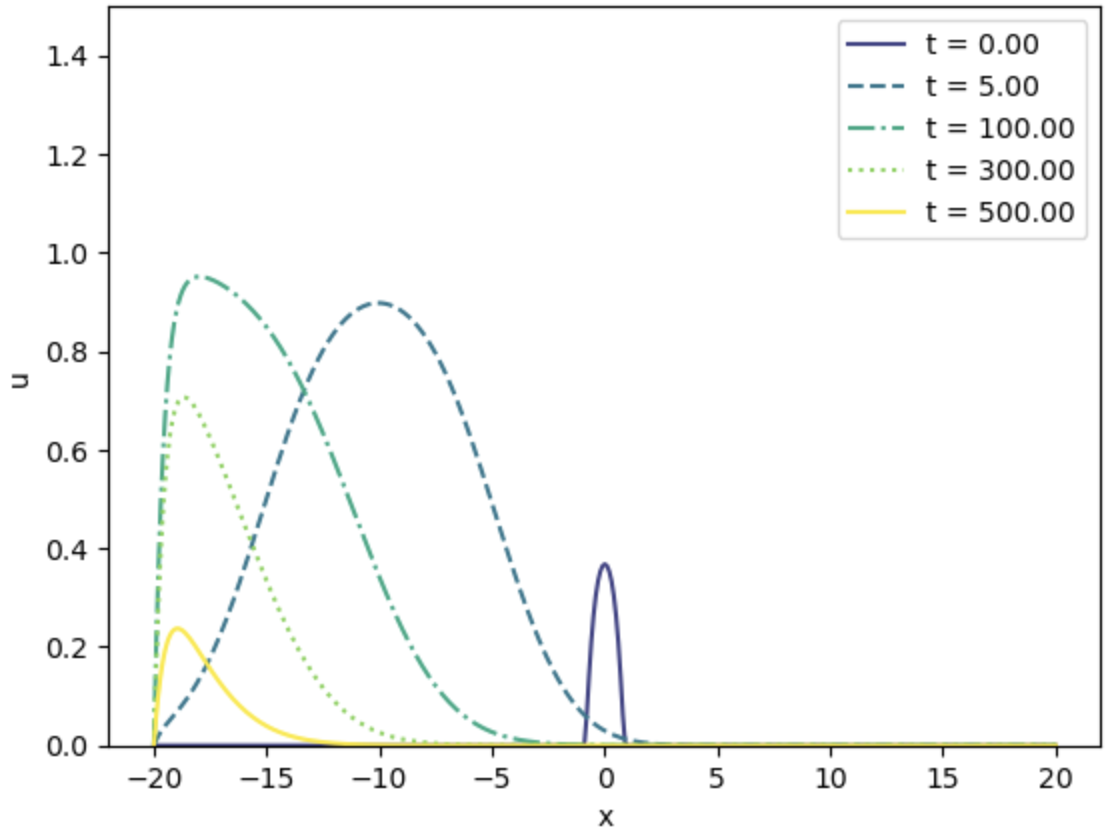}
       \caption{$c = 2.01$}
    \end{subfigure}
    \hfill
   \begin{subfigure}[b]{0.4\textwidth}
        \includegraphics[width=\textwidth]{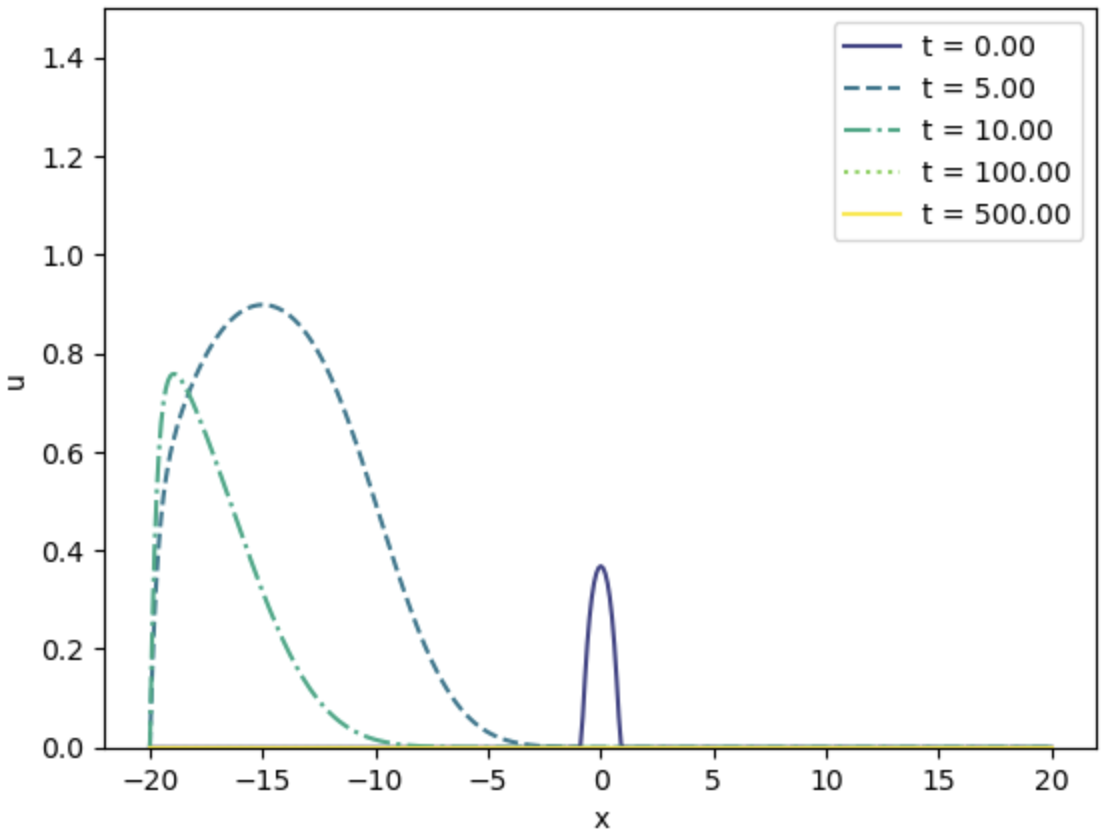}
       \caption{$c = 3$}
    \end{subfigure}
    \caption{$\chi = -1, \tau=1, u(t,x)$}
    \label{chi--1-c-3}
\end{figure}

\begin{figure}[H]
    \centering
    \begin{subfigure}[b]{0.4\textwidth}
        \includegraphics[width=\textwidth]{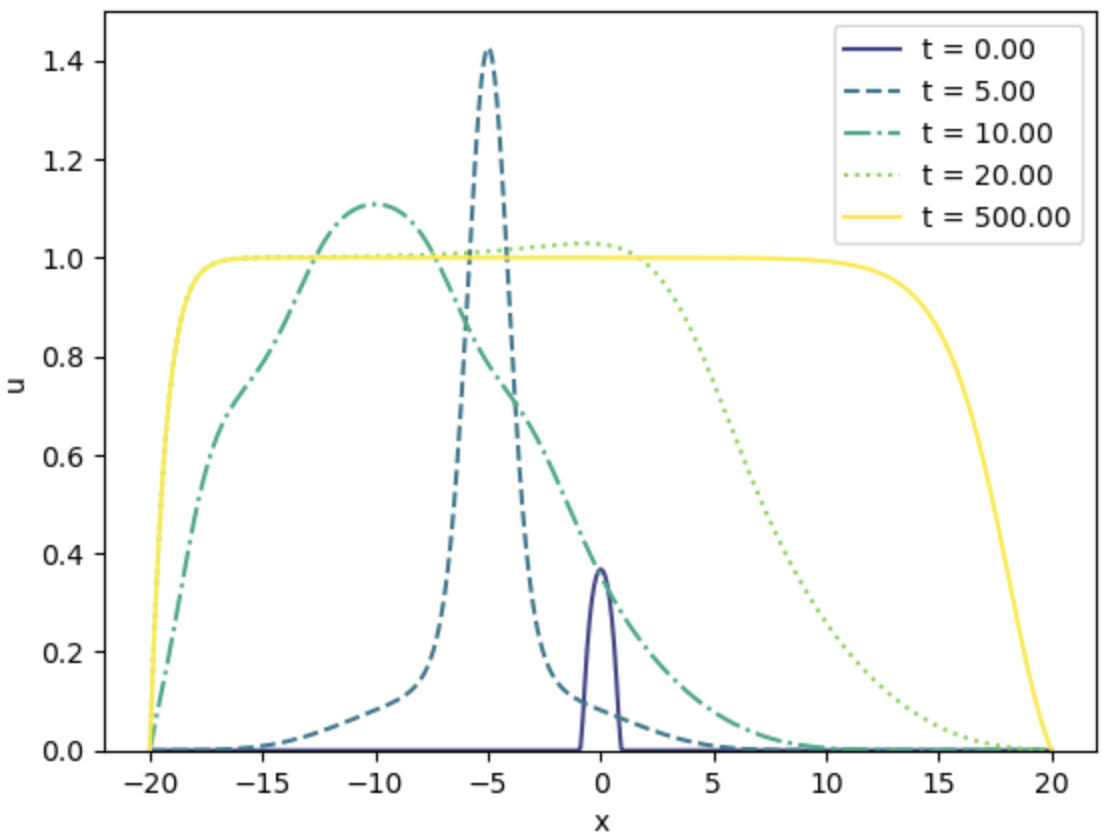}
       \caption{$ c = 1$}
    \end{subfigure}
    \hfill
   \begin{subfigure}[b]{0.4\textwidth}
        \includegraphics[width=\textwidth]{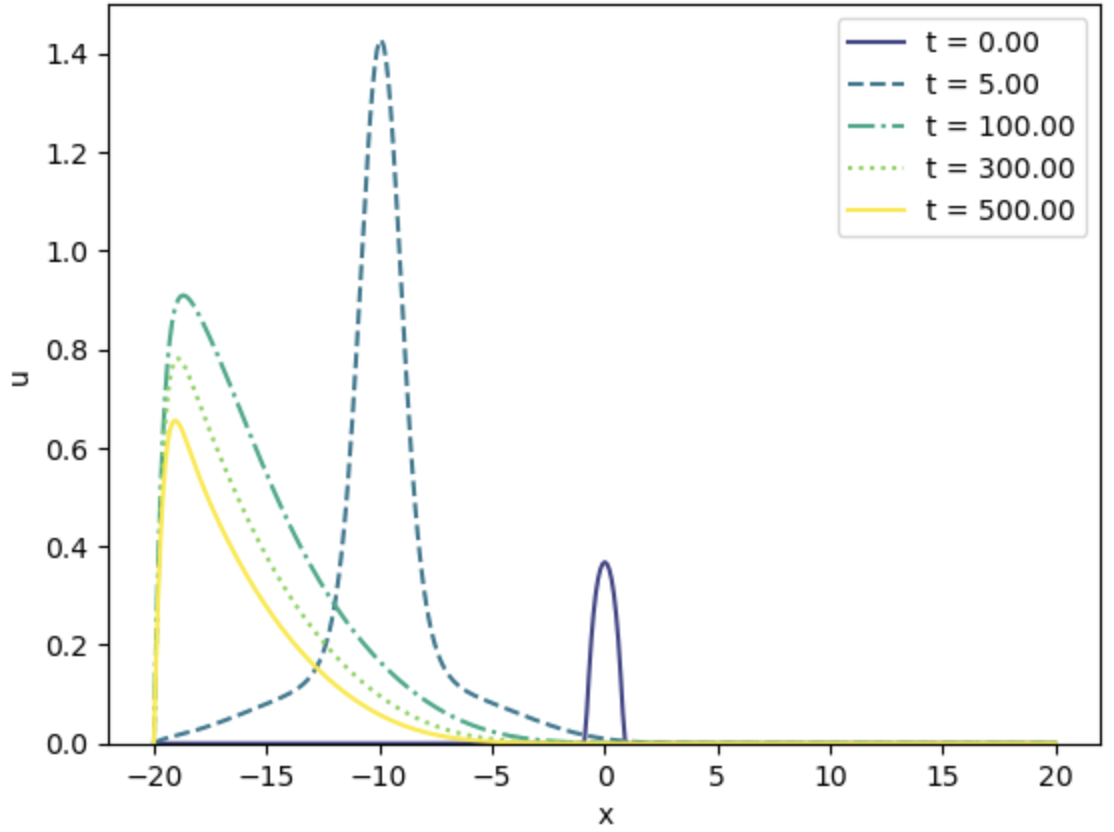}
       \caption{$c = 1.99$}
    \end{subfigure}
    \caption{$\chi = -10, \tau=1, u(t,x)$}
    \label{chi--10-c-1.99}
\end{figure}

\begin{figure}[H]
    \centering
    \begin{subfigure}[b]{0.4\textwidth}
        \includegraphics[width=\textwidth]{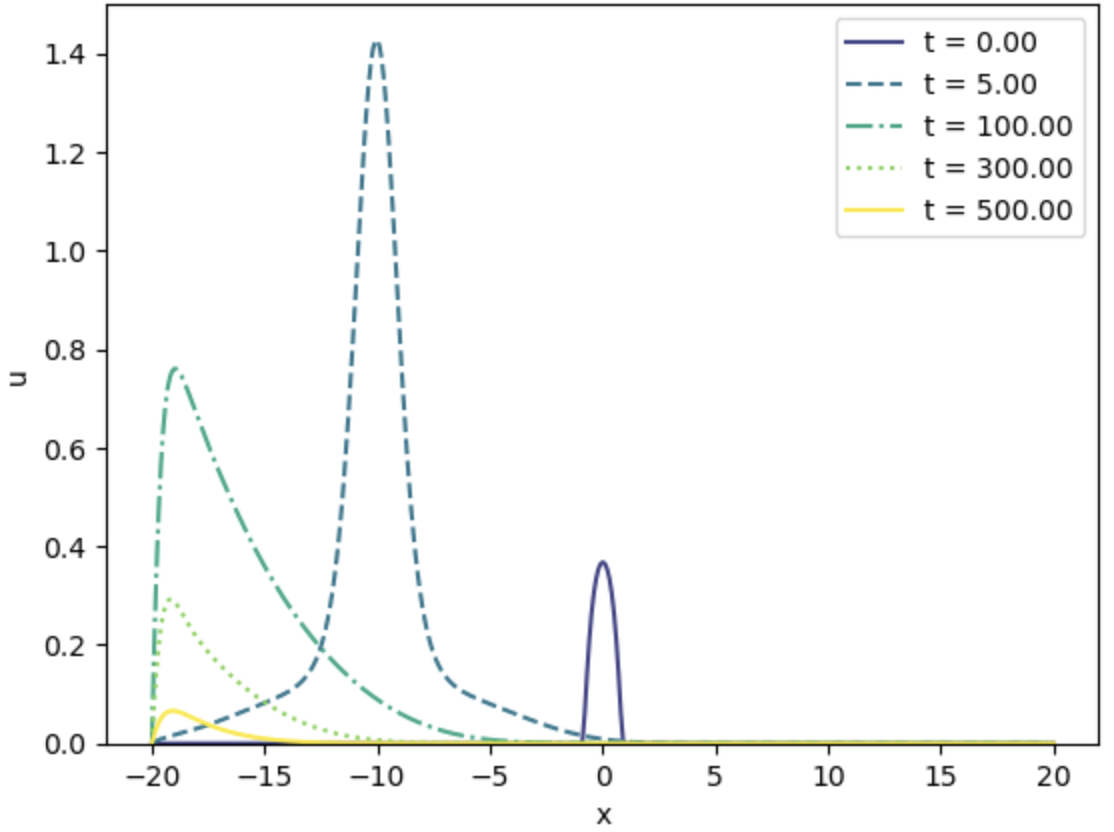}
       \caption{$ c = 2.01$}
    \end{subfigure}
    \hfill
   \begin{subfigure}[b]{0.4\textwidth}
        \includegraphics[width=\textwidth]{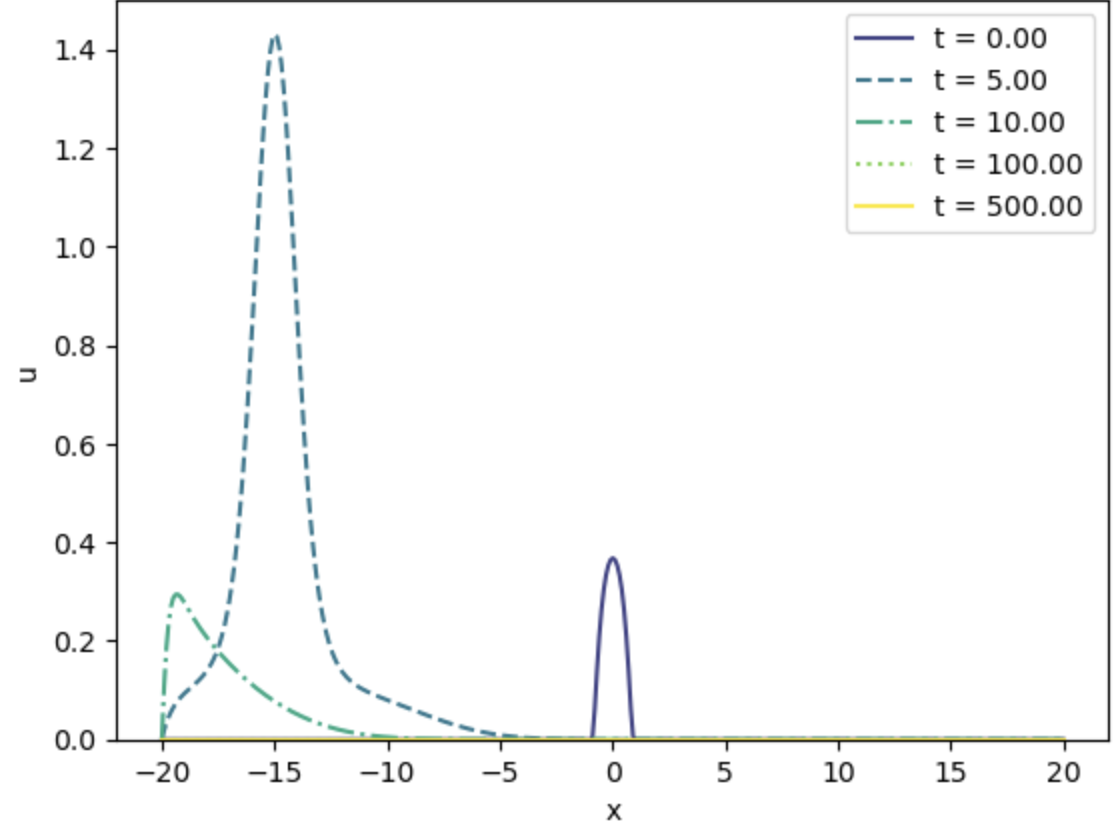}
       \caption{$ c = 3$}
    \end{subfigure}
    \caption{$\chi = -10, \tau=1,  u(t,x)$}
    \label{chi--10-c-3}
\end{figure}

\subsection{Numerical Simulations for different $\tau$}\label{simulation-4}
We carry out some numerical experiments, to see the effect of $\tau$ on the spreading speed. We choose $\tau = 0.5,1, 4$, and $c= 1, 2.01, 3$. The results of the 2D simulations of the $u$-component  are shown in subsections 6.4.1--6.4.3 for different time steps. In each of the images, (a) is the result for $\tau = 0.5,$ (b) is the result for $\tau = 1,$ and $(c)$ is for $\tau = 4.$ All the images are for time $t = 0, 5, 10, 20, 500$ except for $c=2.01, \chi =1$, which is for time $0, 5, 20, 500, 1000.$

For $c = 2.01,$ we gave the result of the simulation with $\chi = 1.9, 5, 10$. While for $c = 1$ and $c=3,$ we show the result of the simulation for $\chi = 1.9, 5, 10$

\subsubsection{Results of simulation for $c=1$}
For $c=1$, the solution all stays positive for all $\tau.$ This is expected from Theorem \ref{speed-lower-bound-thm}. The results of the simulations for $\chi = 1.9, 5, 10$ are shown in Figures \ref{f10}, \ref{f11} and \ref{f12}.

\begin{figure}[H]
    \centering
    \begin{subfigure}[b]{0.3\textwidth}
        \includegraphics[width=\textwidth]{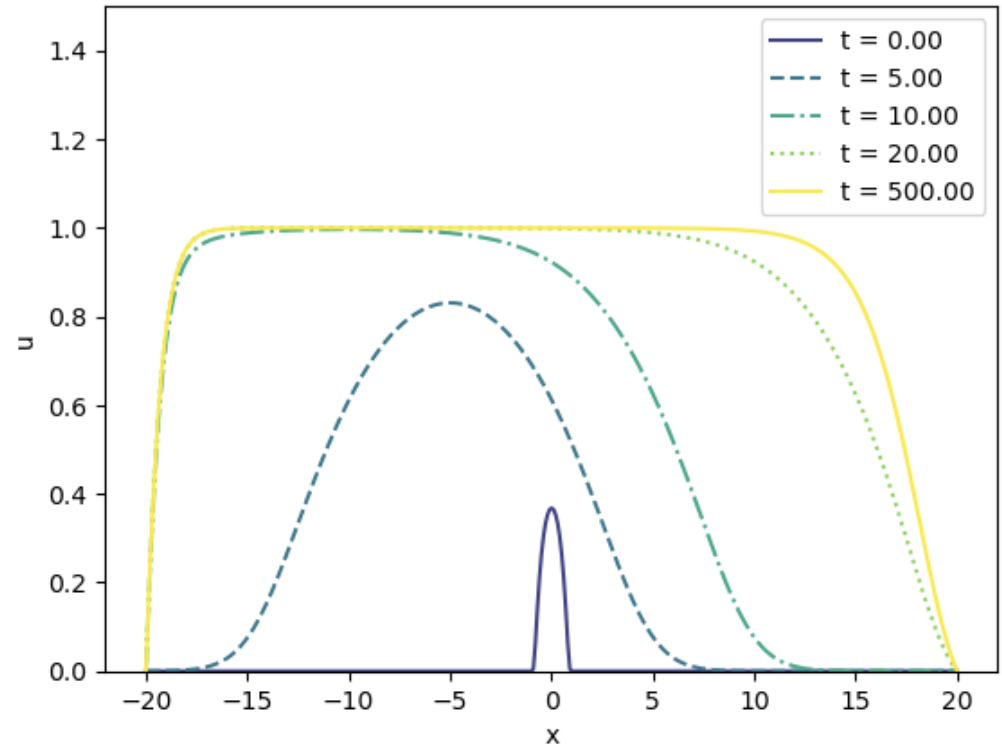}
        \caption{$\tau=0.5$}
    \end{subfigure}
    \begin{subfigure}[b]{0.3\textwidth}
        \includegraphics[width=\textwidth]{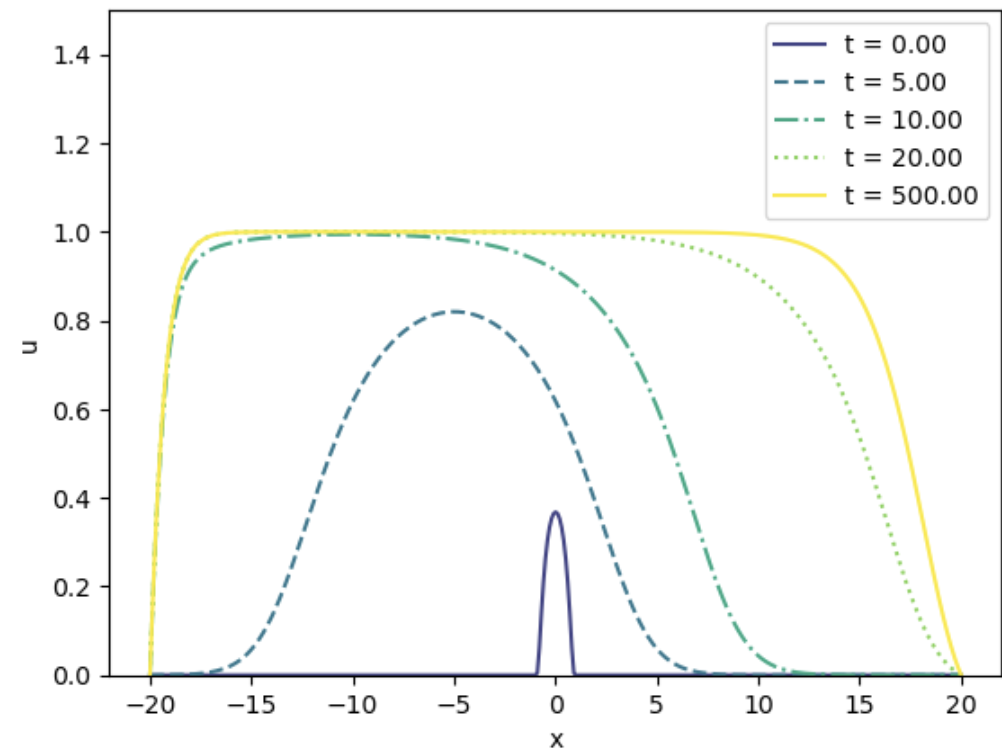}
        \caption{$\tau=1$}
    \end{subfigure}
    \begin{subfigure}[b]{0.3\textwidth}
        \includegraphics[width=\textwidth]{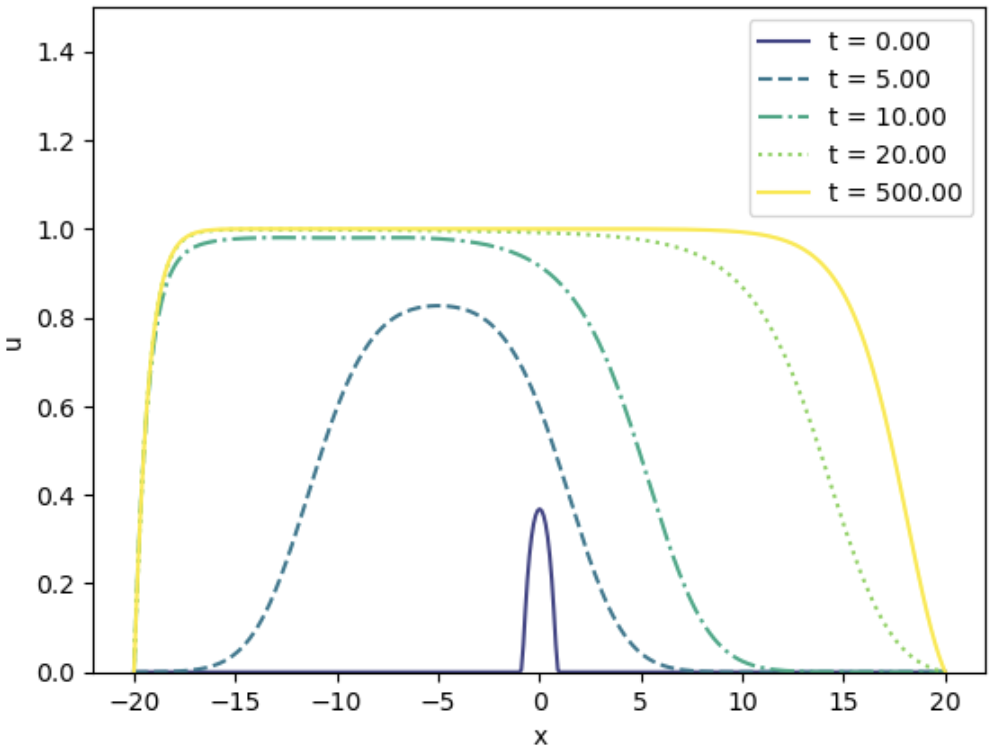}
        \caption{$\tau=4$}
    \end{subfigure}
    \caption{$\chi = 1.9,  c = 1$}
    \label{f10}
\end{figure}

\begin{figure}[H]
    \centering
    \begin{subfigure}[b]{0.3\textwidth}
        \includegraphics[width=\textwidth]{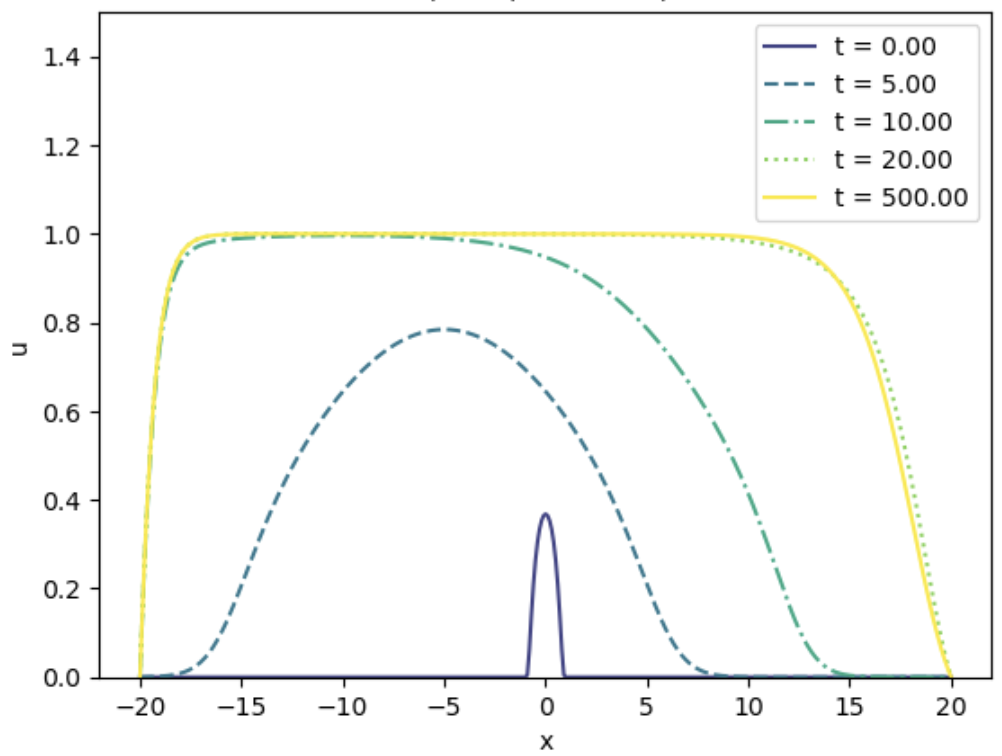}
        \caption{$\tau=0.5$}
    \end{subfigure}
    \begin{subfigure}[b]{0.3\textwidth}
        \includegraphics[width=\textwidth]{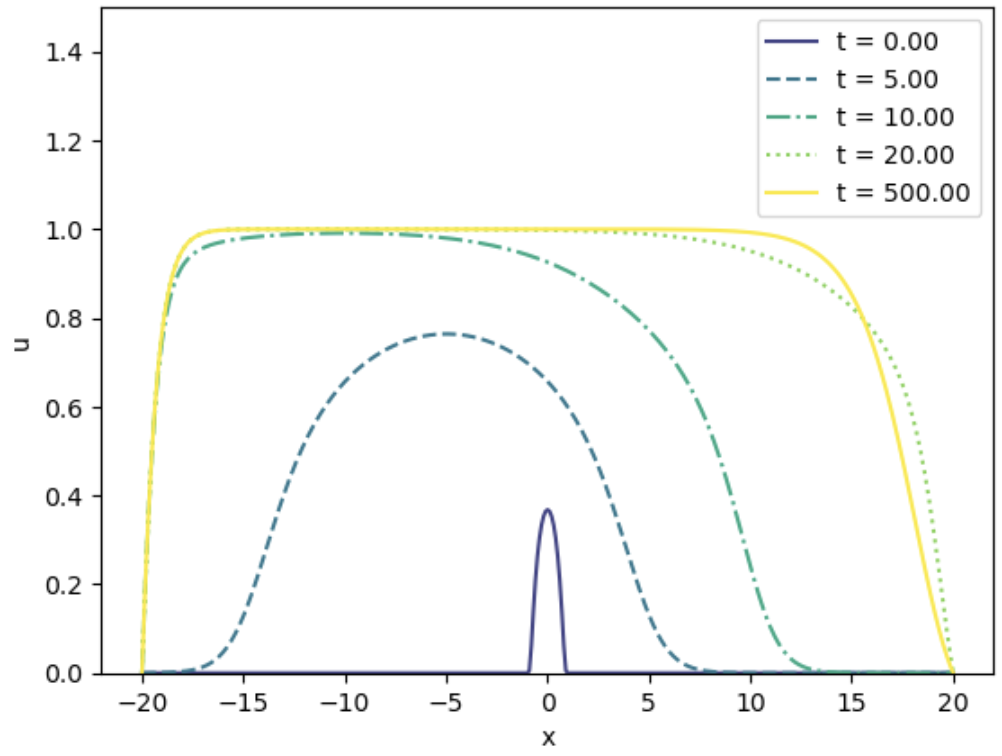}
        \caption{$\tau=1$}
    \end{subfigure}
    \begin{subfigure}[b]{0.3\textwidth}
        \includegraphics[width=\textwidth]{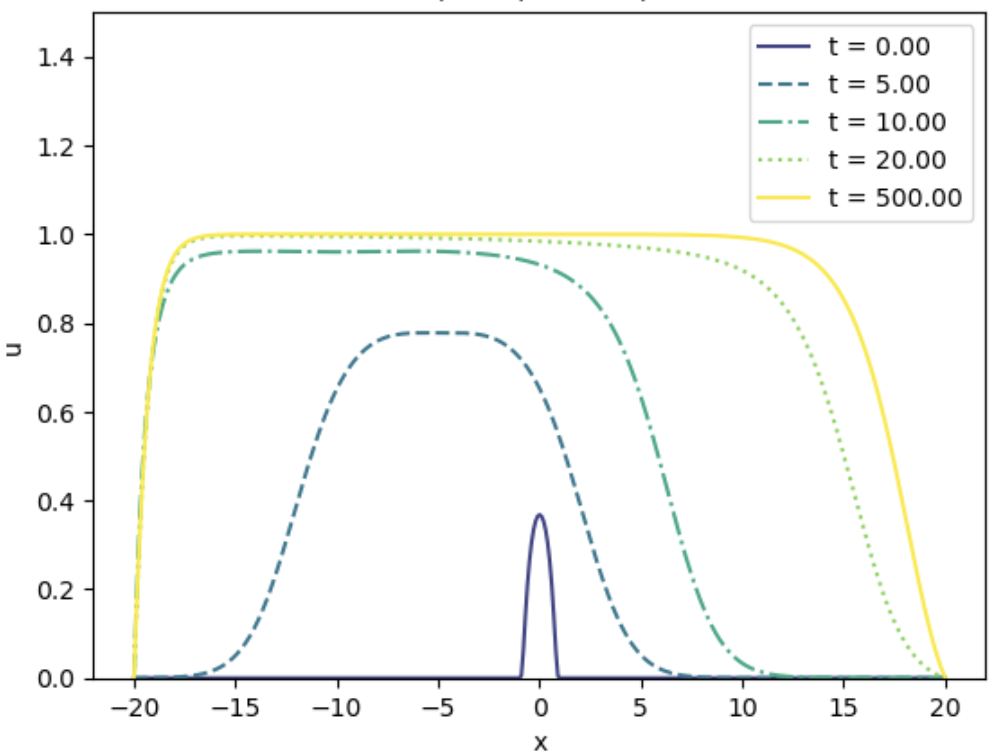}
        \caption{$\tau=4$}
    \end{subfigure}
    \caption{$\chi = 5, c = 1$}
    \lb{f11}
\end{figure}

\begin{figure}[H]
    \centering
    \begin{subfigure}[b]{0.3\textwidth}
        \includegraphics[width=\textwidth]{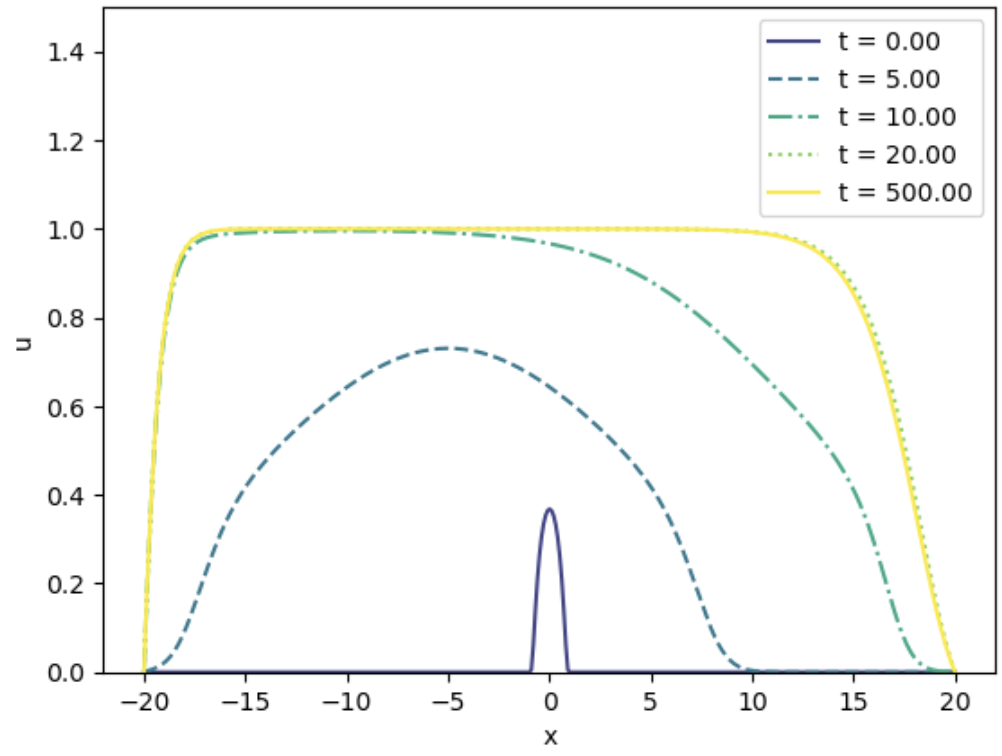}
        \caption{$\tau=0.5$}
    \end{subfigure}
    \begin{subfigure}[b]{0.3\textwidth}
        \includegraphics[width=\textwidth]{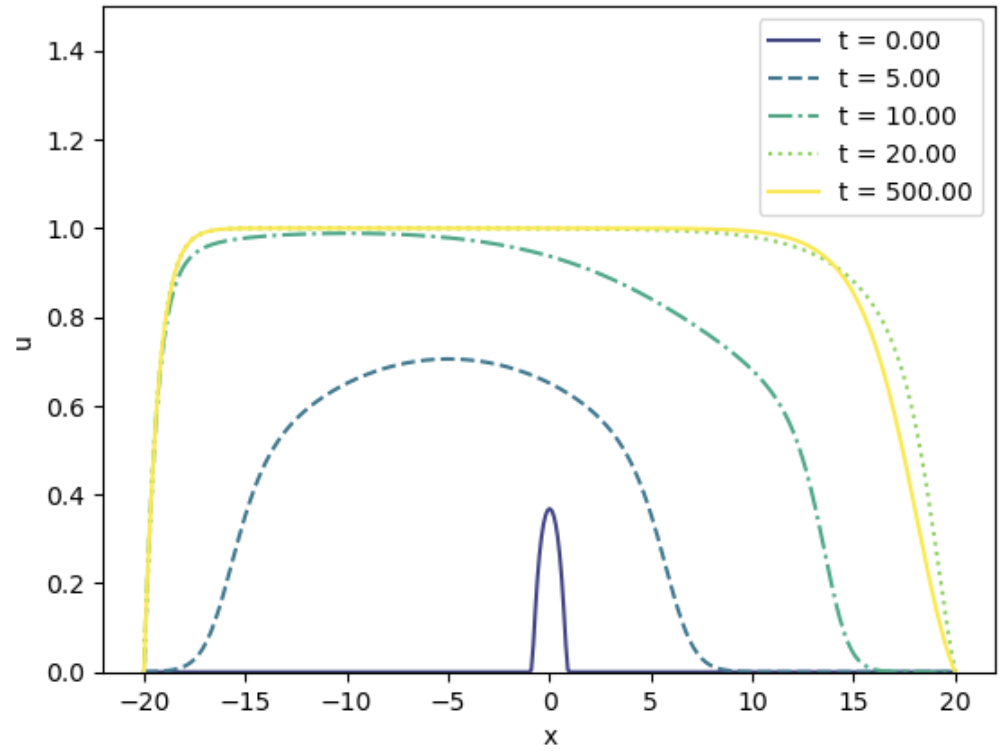}
        \caption{$\tau=1$}
    \end{subfigure}
    \begin{subfigure}[b]{0.3\textwidth}
        \includegraphics[width=\textwidth]{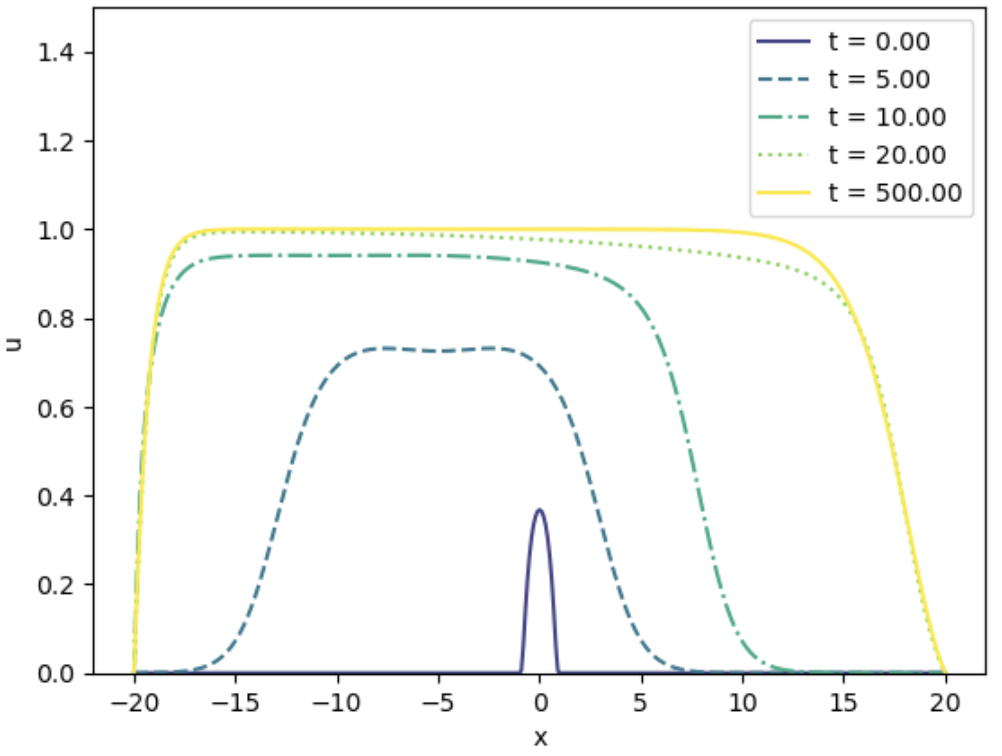}
        \caption{$\tau=4$}
    \end{subfigure}
    \caption{$\chi = 10, c = 1$}
    \lb{f12}
\end{figure}

\subsubsection{Results of simulation for $c=2.01$}
The result of the simulations shows that the diffusion rate of $v$ has an effect on the spreading speed. We see that as $\tau$ gets bigger, then $u$ goes to zero faster, see Figures \ref{f13}--\ref{f17}. We also see that for large $\chi$ there might be speed up for all the $\tau$, as seen for $\chi=5$ and $\chi =10$ in Figures \ref{c2.01chi5} and \ref{c2.01chi10}.

\begin{figure}[H]
    \centering
    \begin{subfigure}[b]{0.3\textwidth}
        \includegraphics[width=\textwidth]{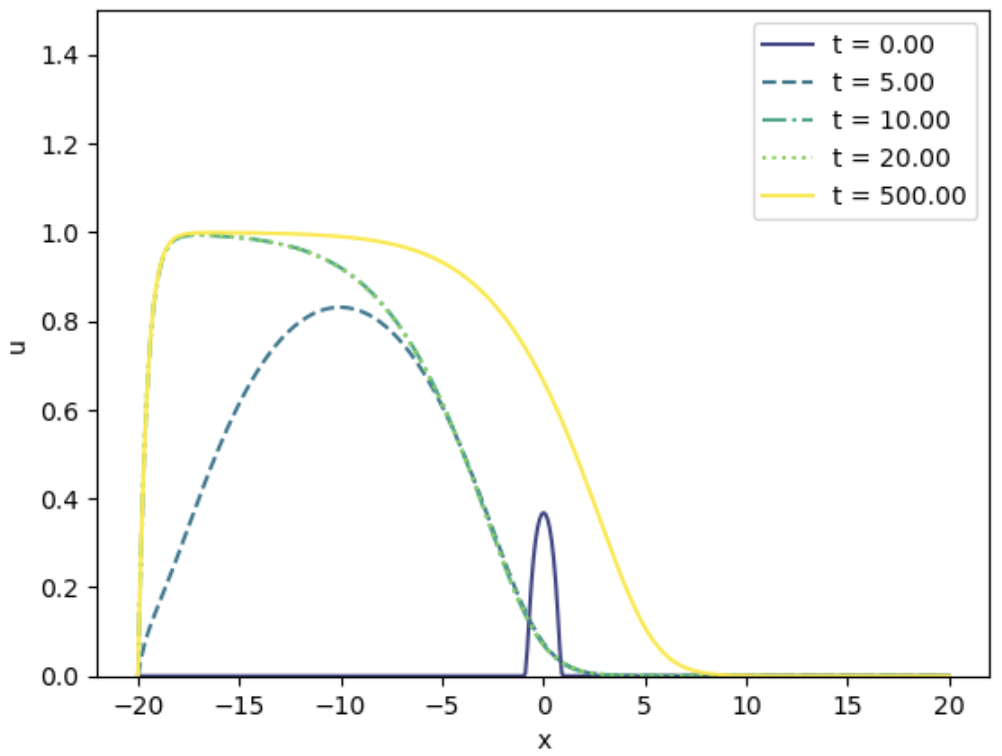}
        \caption{$\tau=0.5$}
    \end{subfigure}
    \begin{subfigure}[b]{0.3\textwidth}
        \includegraphics[width=\textwidth]{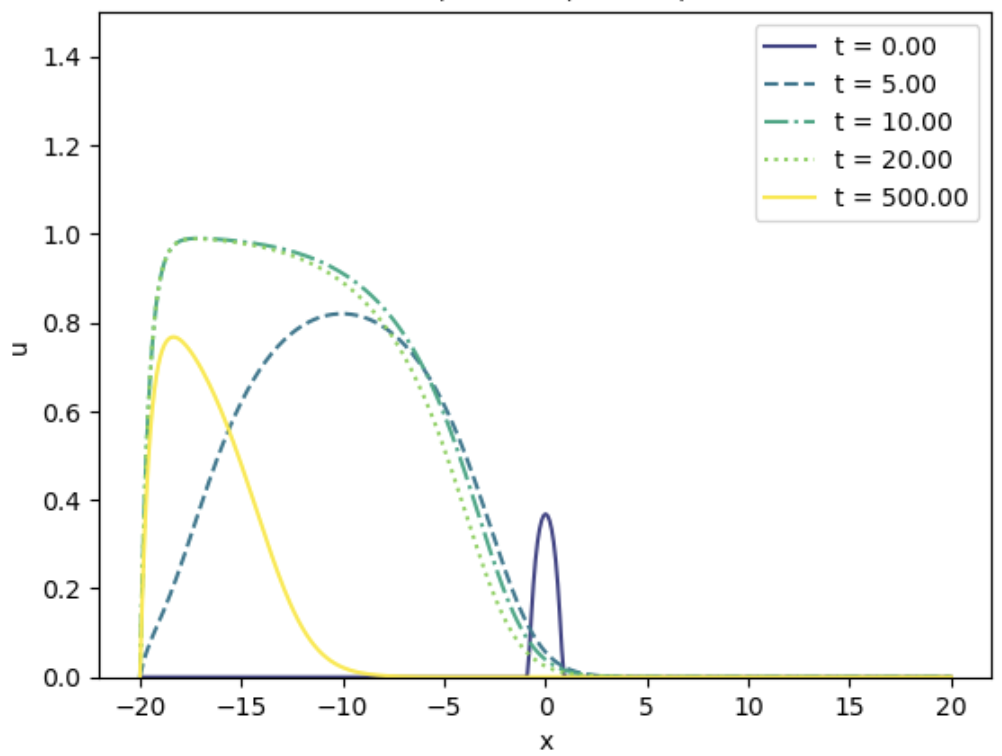}
        \caption{$\tau=1$}
    \end{subfigure}
    \begin{subfigure}[b]{0.3\textwidth}
        \includegraphics[width=\textwidth]{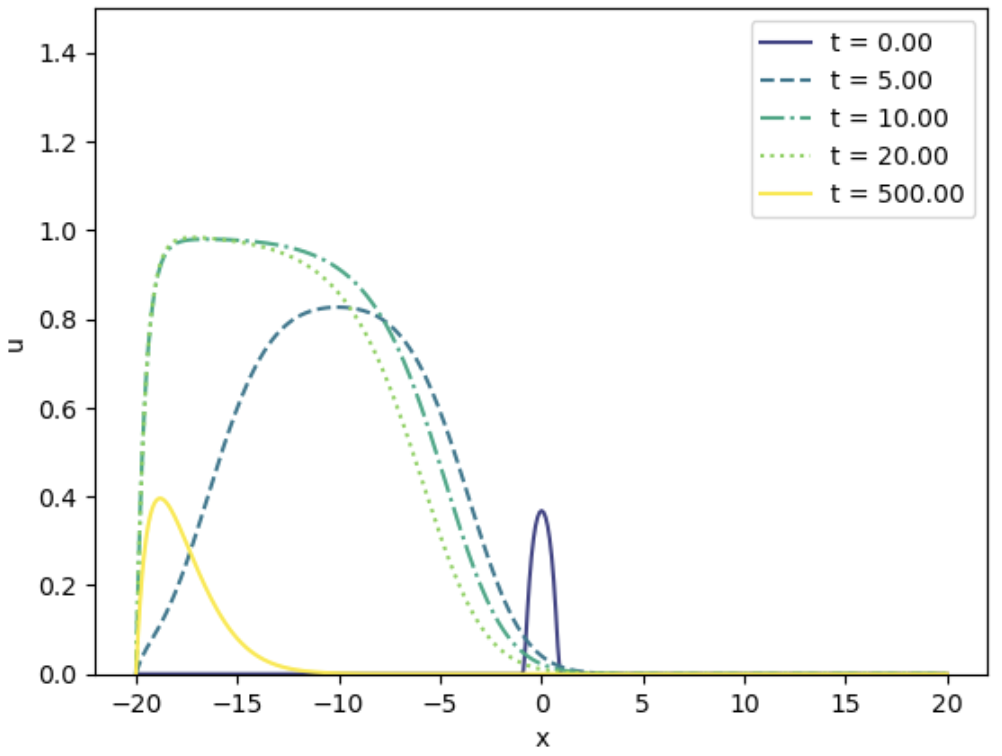}
        \caption{$\tau=4$}
    \end{subfigure}
    \caption{$\chi = 1.9, c = 2.01$}
    \lb{f13}
\end{figure}


\begin{figure}[H]
    \centering
    \begin{subfigure}[b]{0.3\textwidth}
        \includegraphics[width=\textwidth]{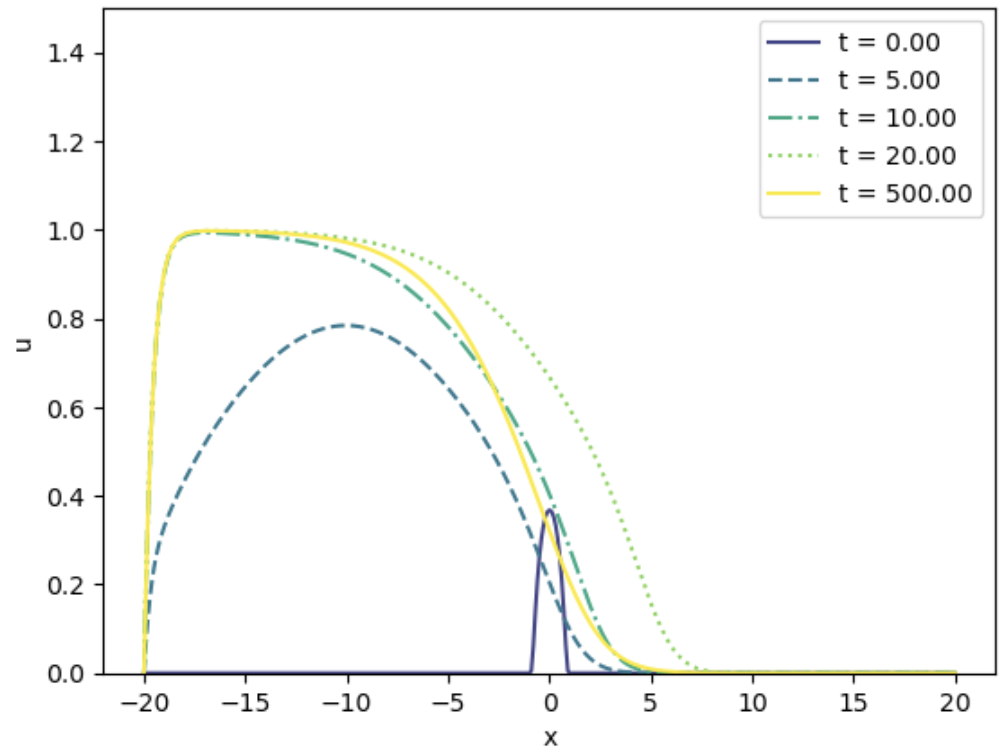}
        \caption{$\tau=0.5$}
    \end{subfigure}
    \begin{subfigure}[b]{0.3\textwidth}
        \includegraphics[width=\textwidth]{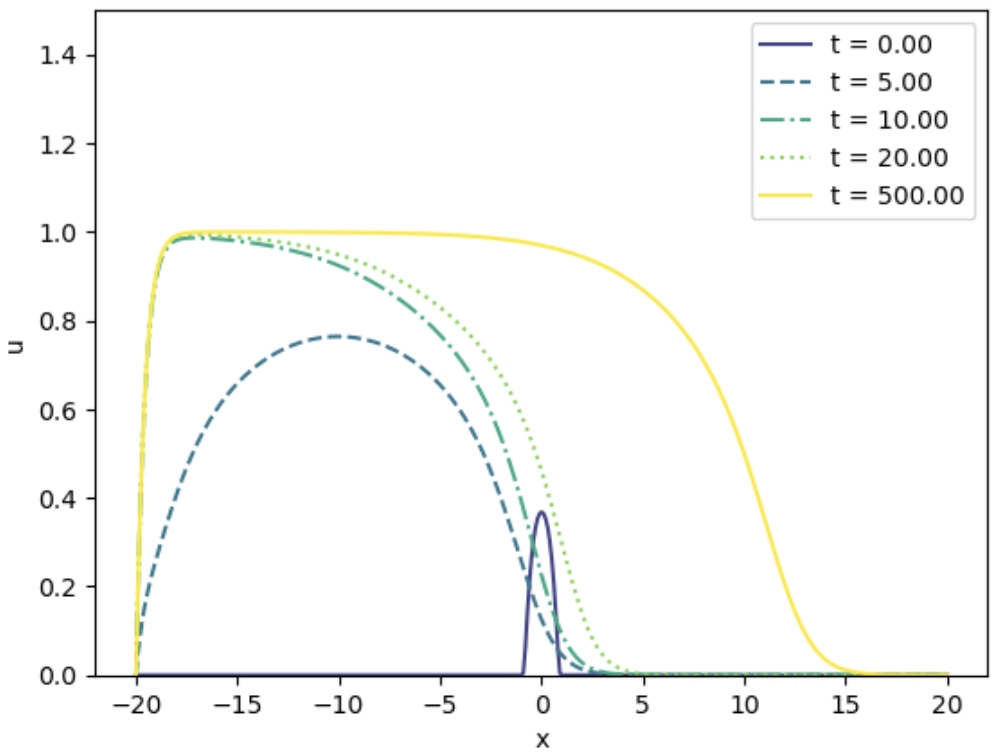}
        \caption{$\tau=1$}
    \end{subfigure}
    \begin{subfigure}[b]{0.3\textwidth}
        \includegraphics[width=\textwidth]{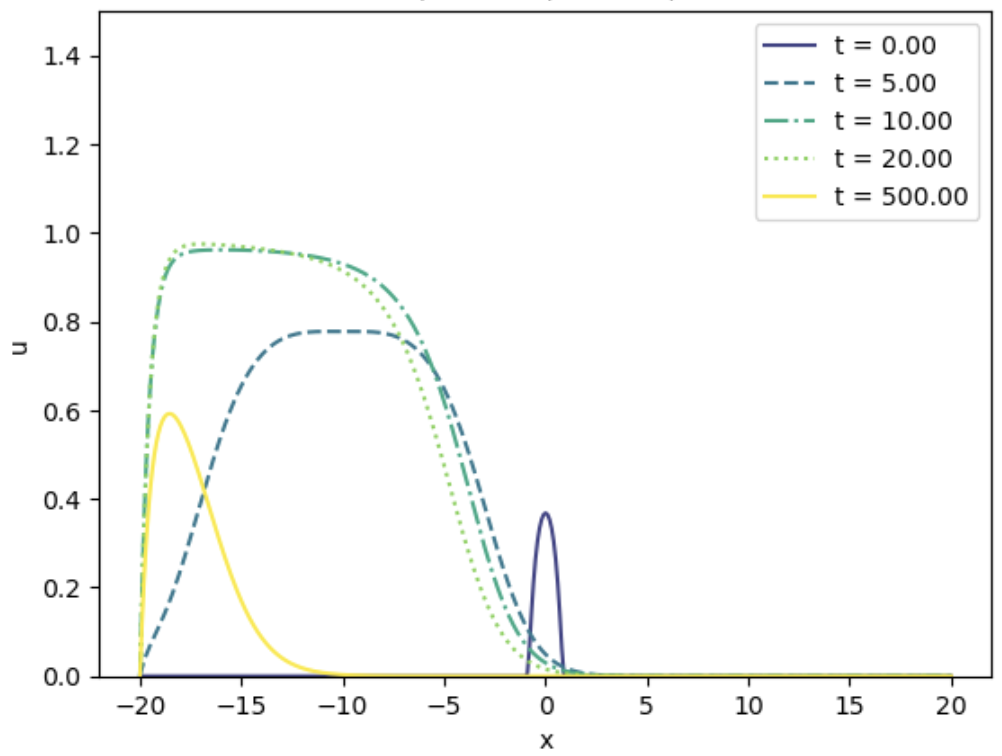}
        \caption{$\tau=4$}
    \end{subfigure}
    \caption{$\chi = 5, c = 2.01$}
    \label{c2.01chi5}
\end{figure}

\begin{figure}[H]
    \centering
    \begin{subfigure}[b]{0.3\textwidth}
        \includegraphics[width=\textwidth]{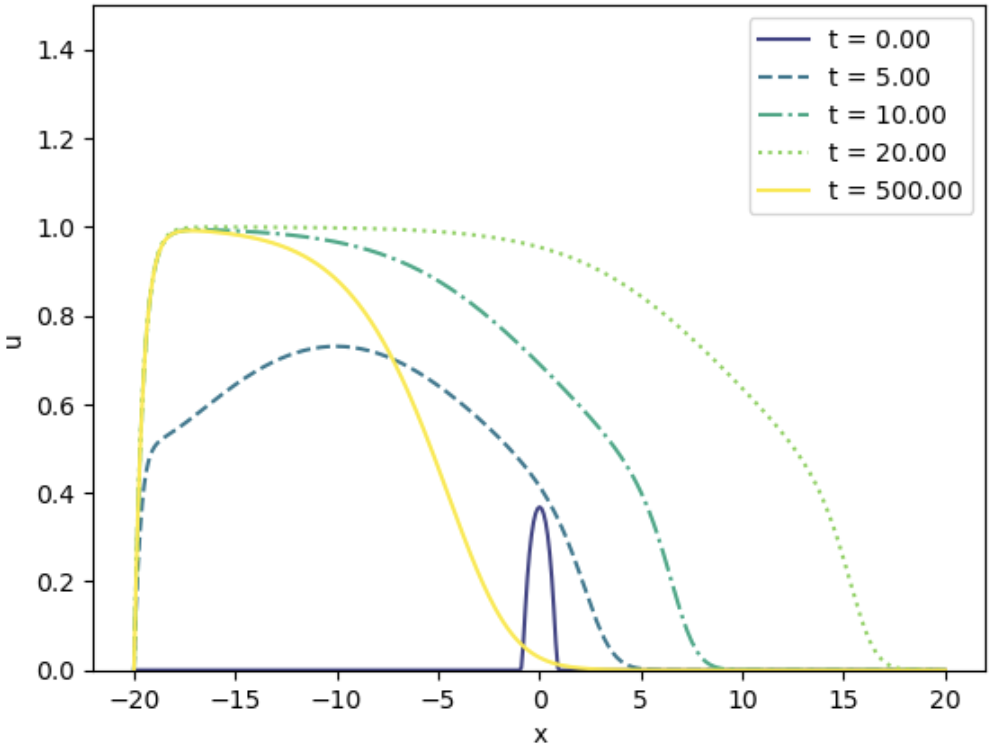}
        \caption{$\tau=0.5$}
    \end{subfigure}
    \begin{subfigure}[b]{0.3\textwidth}
        \includegraphics[width=\textwidth]{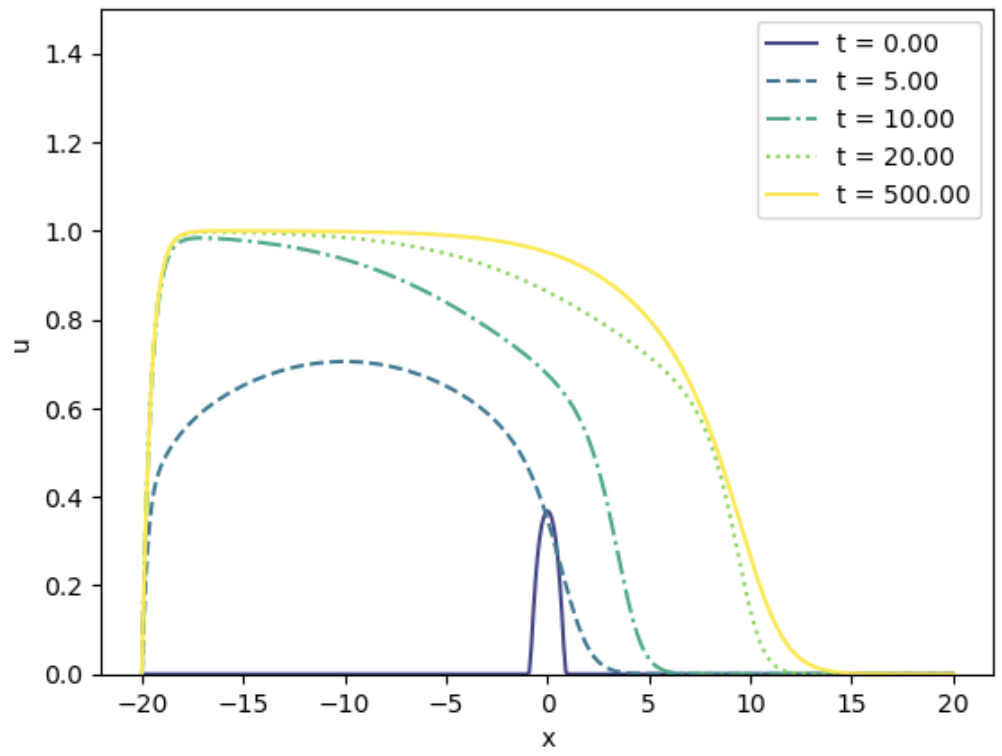}
        \caption{$\tau=1$}
    \end{subfigure}
    \begin{subfigure}[b]{0.3\textwidth}
        \includegraphics[width=\textwidth]{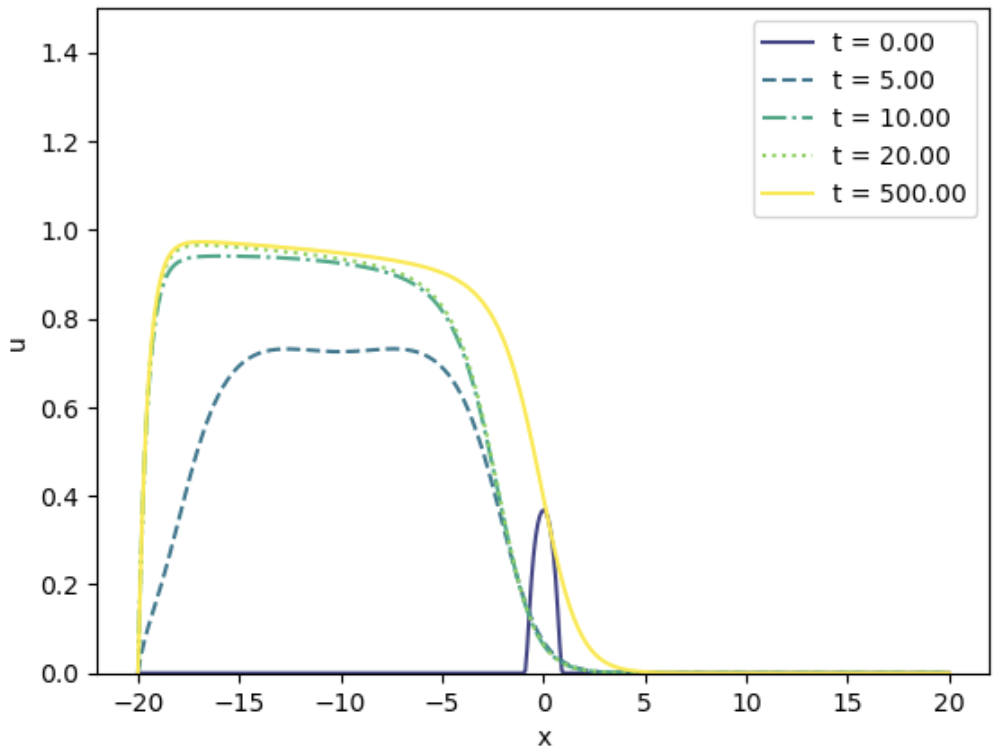}
        \caption{$\tau=4$}
    \end{subfigure}
    \caption{$\chi = 10, c = 2.01$}
    \label{c2.01chi10}
    \lb{f17}
\end{figure}
\subsubsection{Results of simulation for $c=3$}
We carry out the simulation for $c=3$ and $\tau = 0.5, 1, 4$. The results show that the solution goes to zero for large time and for all the $\tau$, see Figures \ref{f18}--\ref{f20}.
\begin{figure}[H]
    \centering
    \begin{subfigure}[b]{0.3\textwidth}
        \includegraphics[width=\textwidth]{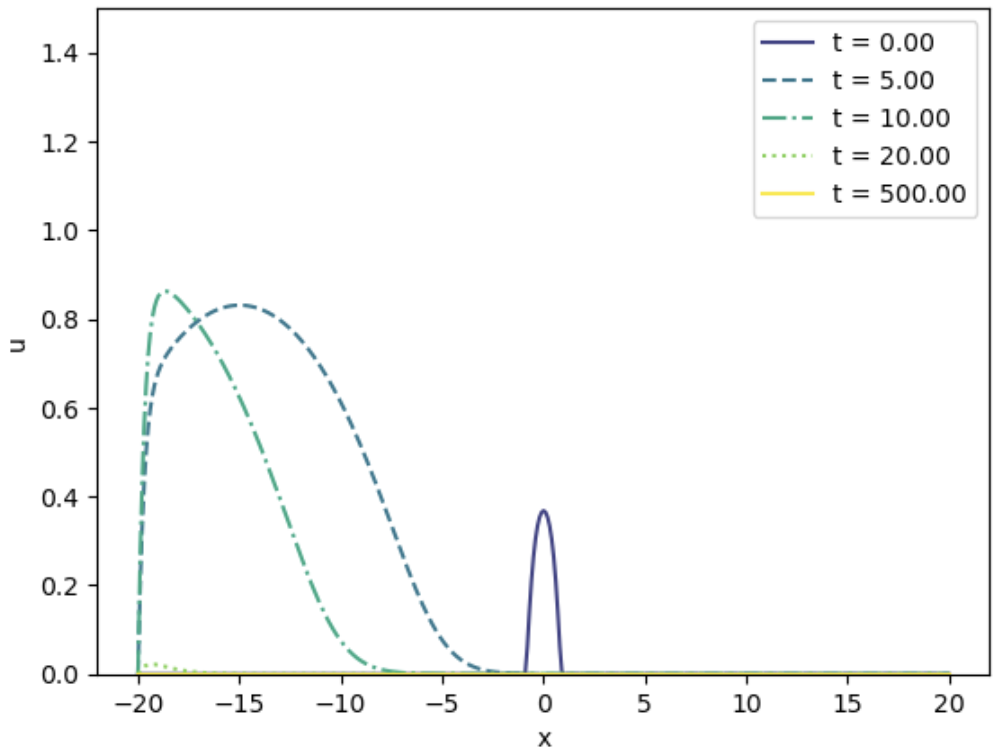}
        \caption{$\tau=0.5$}
    \end{subfigure}
    \begin{subfigure}[b]{0.3\textwidth}
        \includegraphics[width=\textwidth]{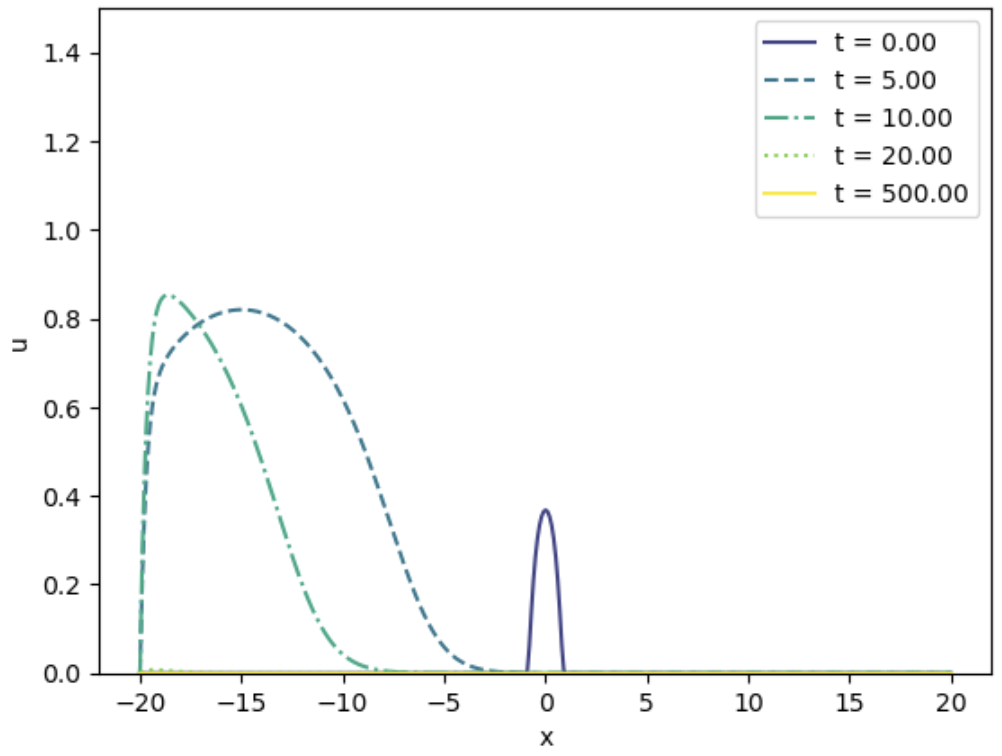}
        \caption{$\tau=1$}
    \end{subfigure}
    \begin{subfigure}[b]{0.3\textwidth}
        \includegraphics[width=\textwidth]{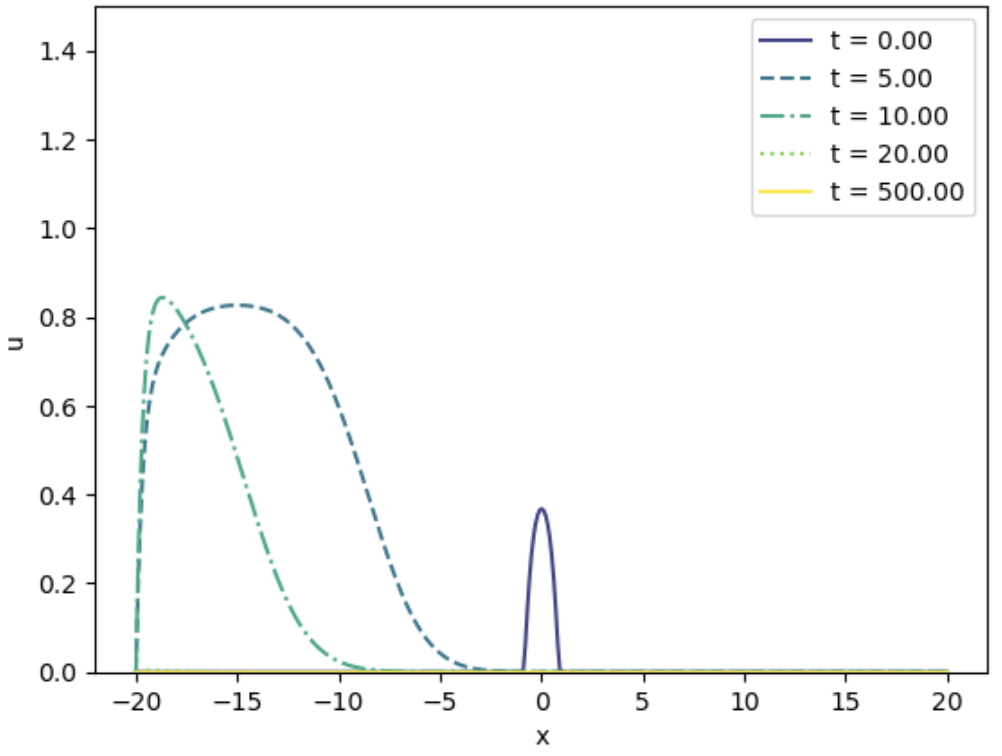}
        \caption{$\tau=4$}
    \end{subfigure}
    \caption{$\chi = 1.9, c = 3$}
    \lb{f18}
\end{figure}

\begin{figure}[H]
    \centering
    \begin{subfigure}[b]{0.3\textwidth}
        \includegraphics[width=\textwidth]{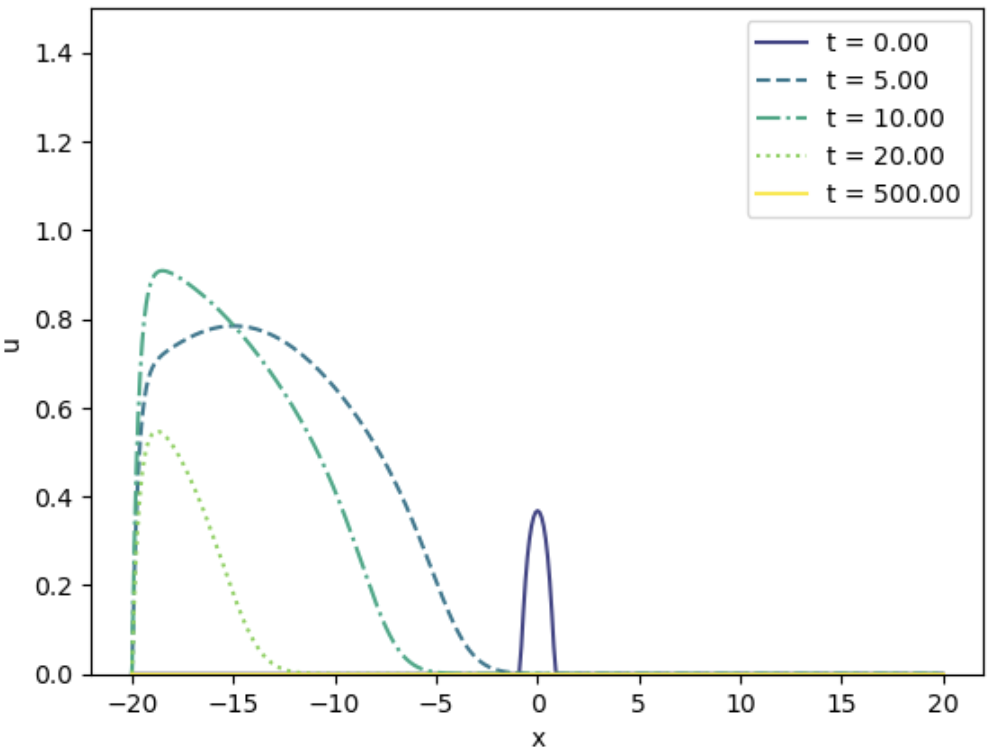}
        \caption{$\tau=0.5$}
    \end{subfigure}
    \begin{subfigure}[b]{0.3\textwidth}
        \includegraphics[width=\textwidth]{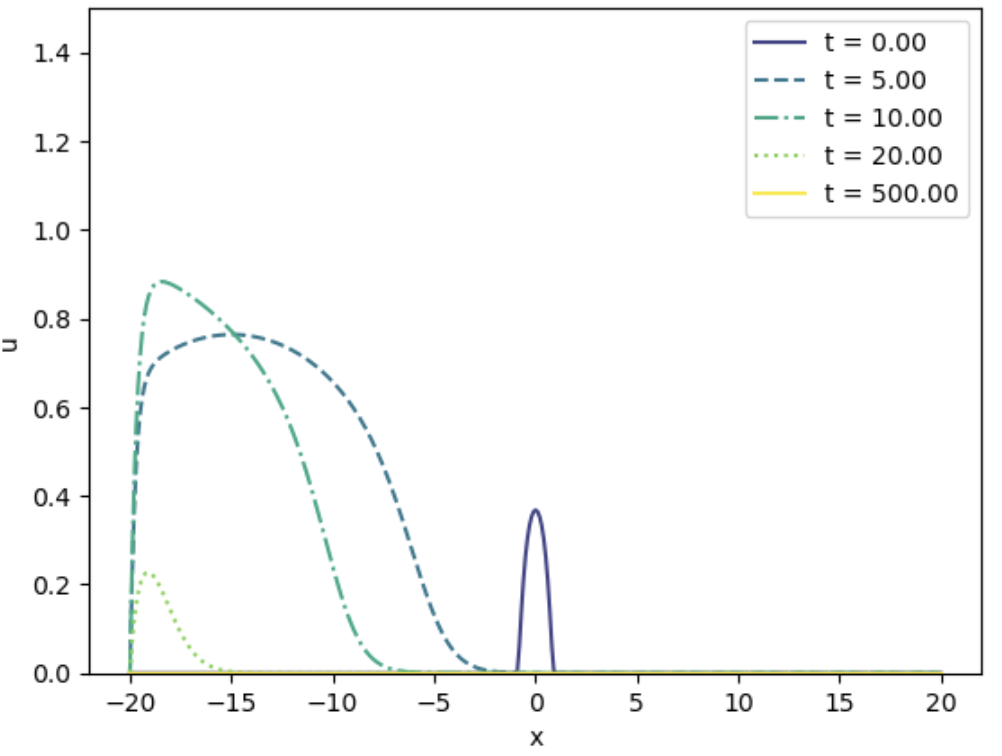}
        \caption{$\tau=1$}
    \end{subfigure}
    \begin{subfigure}[b]{0.3\textwidth}
        \includegraphics[width=\textwidth]{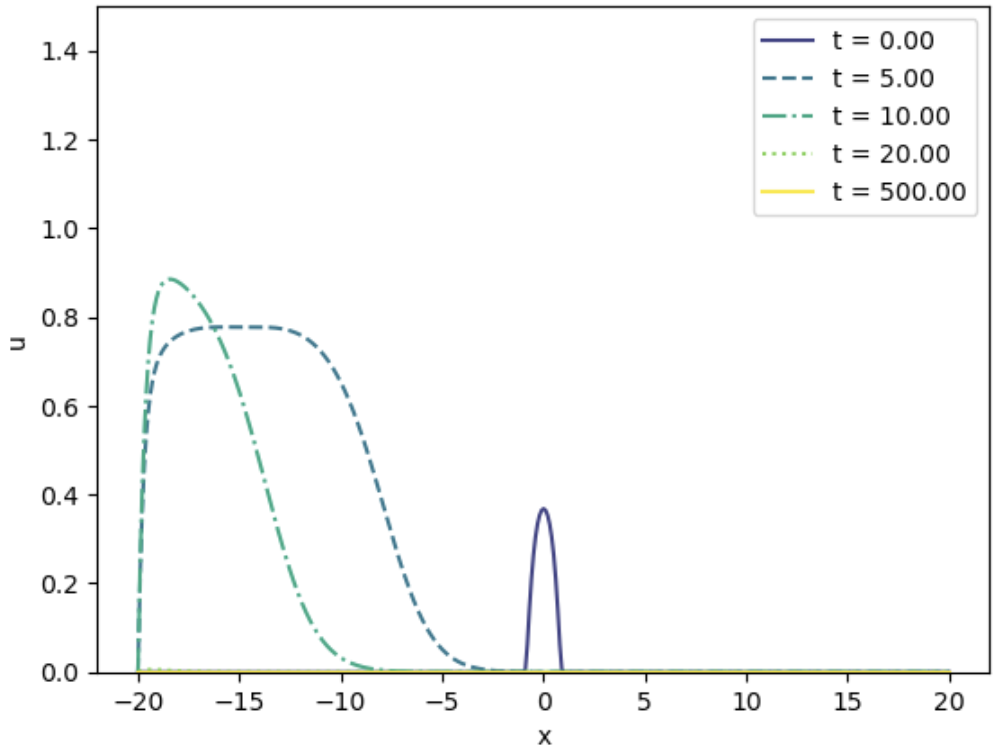}
        \caption{$\tau=4$}
    \end{subfigure}
    \caption{$\chi = 5, c = 3$}
\end{figure}

\begin{figure}[H]
    \centering
    \begin{subfigure}[b]{0.3\textwidth}
        \includegraphics[width=\textwidth]{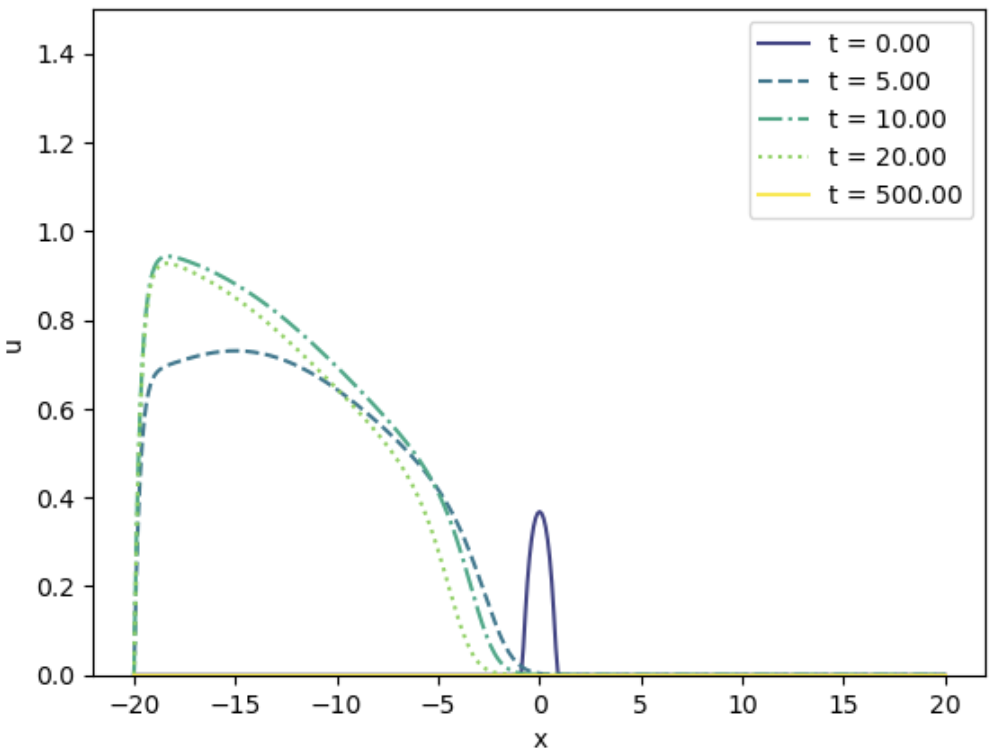}
        \caption{$\tau=0.5$}
    \end{subfigure}
    \begin{subfigure}[b]{0.3\textwidth}
        \includegraphics[width=\textwidth]{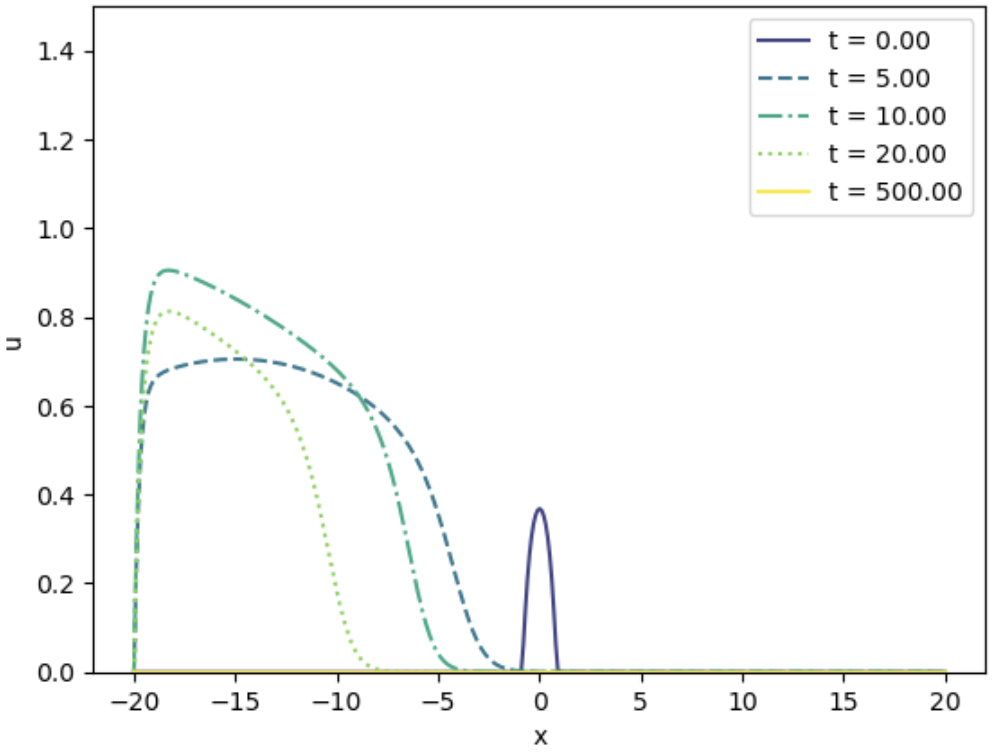}
        \caption{$\tau=1$}
    \end{subfigure}
    \begin{subfigure}[b]{0.3\textwidth}
        \includegraphics[width=\textwidth]{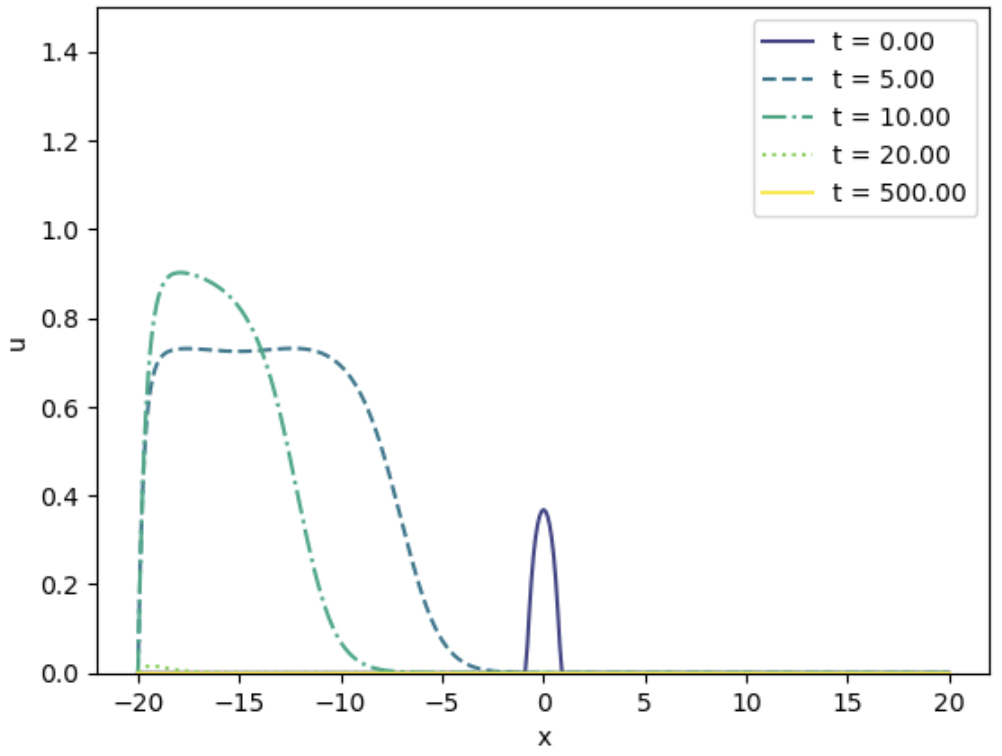}
        \caption{$\tau=4$}
    \end{subfigure}
    \caption{$\chi = 10, c = 3$}
    \lb{f20}
\end{figure}
\appendix

\section{Appendix: Proofs of Lemmas \ref{derivative-boundedness-lm} and \ref{asymptotic-lm1}}

In this appendix, we give  proofs of Lemmas \ref{derivative-boundedness-lm} and \ref{asymptotic-lm1}.

\begin{proof}[Proof of Lemma \ref{derivative-boundedness-lm}]
First of all, for simplicity in notation, we put
$$
u(t)=u(t,\cdot;u_0,v_0),\quad v(t)=v(t,\cdot;u_0,v_0).
$$
Let $T(t)$ be the analytic semigroup generated by $-A$ on $X=C_{\rm unif}^b(\R^N)$,
where $A=I-\Delta$.
Then we have
\begin{equation}
\label{v-formula-eq}
v(t)=T(t)v_0+\int_0^t T(t-s) (v(s)-u(s)v(s))ds
\end{equation}
and
\begin{equation}
\label{u-formula-eq}
u(t)=\underbrace{T(t)u_0}_{I_0(t)}-\chi \underbrace{\int_0^t T(t-s)\nabla \cdot (u(s)\nabla v(s))ds}_{I_1(t)}+\underbrace{\int_0^t T(t-s) u(s)\big(1+a-bu(s)\big)ds}_{I_2(t)}
\end{equation}
for all $t>0$. 
For $\beta\in [0,\infty)$, let $X^\beta = {\rm Dom}(A^{\beta})$ be the fractional power spaces associated with $A $  on $X = C^b_{\rm unif}(\R^N)$. We denote $\|\cdot\|_{\beta} = \|\cdot\|_{X^\beta}$.  Then  for $0<\delta<1$ and $\beta\ge 0$, there is constant $C_{\delta,\beta}$, 
 \begin{equation}
    \label{L-infinity-estimate-1}
     \|A^\beta T(t)u \|_{C_{\rm unif}^b(\mathbb{R}^{N})}\leq  C_{\delta, \beta}  t^{-\beta }e^{-(1-\delta)t}
    \|u\|_{C_{\rm unif}^b (\mathbb{R}^{N})}
     \end{equation}
     for every $u\in C_{\rm unif}^b(\R^N)$ and  $t>0$ (see \cite[Theorem 1.4.3]{Hen}).
  For every $t>0$, the operator $T(t)\nabla \cdot$ has a unique bounded extension on
    $(C_{\rm unif}^b(\R^N))^N$ satisfying
    \begin{equation}
    \label{L-infinity-estimate-2}
    \|T(t)\nabla \cdot p\|_{C_{\rm unif}^b(\R^{N})}\leq \frac{N}{\sqrt\pi}t^{-\frac{1}{2}}e^{-t}\|p\|_{C_{\rm unif}^b (\R^{N})} \ \ \ \forall \ p\in\big( C_{\rm unif}^b (\R^{N})\big)^N, \,\, \ \forall \ t>0
    \end{equation}
(see \cite[Lemma 3.2]{SaSh}).

Next, suppose that $u$ is bounded by a constant $M>0$. Since $\|v\|_\infty\leq \|v_0\|_\infty$,
by \eqref{v-formula-eq},  there is $C>0$ such that
\begin{align*}
    \|\nabla v(t)\|_{\infty} &\leq  \|\nabla T(t)v_0\|_{\infty} +  \int_0^t \|\nabla T(t-s)(1-u(s))v(s)\|_{\infty} ds \nonumber\\
    &\leq  e^{-t}\|\nabla v_0\|_{\infty} + C\int_0^t (t-s)^{-\frac{1}{2}}e^{-(t-s)}\|v(s)(1-u(s))\|_{\infty}ds \nonumber\\
    &\leq  \|\nabla v_0\|_{\infty}  +  C (1+M)\|v_0\|_{\infty}:= M_1,
\end{align*}
where $C$ depends only on the dimension.
By \eqref{u-formula-eq},  \eqref{L-infinity-estimate-1}, and \eqref{L-infinity-estimate-2},   for $\beta \in (0,\frac12)$ and any $\delta\in (0,1)$, there is $C=C(\beta,\delta)>0$ such that
\begin{align*}
\|u(t)\|_{\beta}  &\le \|T(t)u_0\|_{\beta} +|\chi| \int_0^t \left\|T(t-s)\nabla \cdot (u(s)\nabla v(s))\right\|_{\beta}ds + \int_0^t\left\| T(t-s)u(s)(1+a-bu(s))\right\|_{\beta}ds\\
  &\le  C t^{-\beta}\|u_{0}\|_{\infty} +| \chi| C\int_0^t \left(\frac{t-s}{2}\right)^{-\beta}e^{-(1-\delta)(t-s)}\left\|T\left(\frac{t-s}{2}\right)\nabla \cdot (u(s)\nabla v(s))\right\|_{\infty}ds \\
  & \qquad +C\int_0^t (t-s)^{-\beta}e^{-(1-\delta)(t-s)} \|u(s)(1+a-bu(s))\|_{\infty}ds\\
  &\le  C t^{-\beta}\|u_{0}\|_{\infty} + |\chi| C\int_0^t \left(\frac{t-s}{2}\right)^{-\beta-\frac{1}{2}}e^{-(1-\delta)(t-s)}\| (u(s)\nabla v(s))ds\|_{\infty} \\
  & \qquad +C\int_0^t (t-s)^{-\beta}e^{-(1-\delta)(t-s)} \|u(s)(1+a-bu(s))\|_{\infty}ds\\
&\le C_{ t_0} \|u_{0}\|_{\infty}+|\chi| C MM_1 + C M(1+a+bM),
\end{align*}
where the last inequality holds for $t\geq t_0>0$. So for these $t$, $u(t)$ is uniformly bounded in $X^{\beta}$. Then, the uniform H\"{o}lder continuity of $u(t,\cdot)$ for each $t\geq t_0$ follows from the following continuous embeddings (see  \cite[Exercise 9]{Hen})
\begin{eqnarray}\label{Fractional power Imbedding-0}
X^{\beta} \subset C_{\rm unif}^{\lfloor \alpha \rfloor, \alpha -\lfloor \alpha \rfloor,b }(\R^N) \quad \text{if} \quad 0\leq \alpha < 2\beta.
\end{eqnarray}

Now, for any fixed $0<\beta<\frac{1}{2}$,  choose $\tilde\beta\in (\beta,\frac{1}{2})$. Recall $I_i$ with $i=0,1,2$ from \eqref{u-formula-eq}.
Then, by \cite[Theorem 1.4.3]{Hen},  \eqref{L-infinity-estimate-1} and \eqref{L-infinity-estimate-2},    there exist $C=C(\beta,\delta)$ such that for any $t\ge t_0$ and $h>0$, we have
\begin{align*}
\|I_0(t+h)-I_0(t)\|_\beta=\|\big(T(h)-I\big) T(t)u_0\|_\beta\le C h^{\tilde \beta-\beta}\|T(t) u_0\|_{\tilde \beta}\le C h^{\tilde\beta-\beta} t^{-\tilde\beta}\|u_0\|_\infty, 
\end{align*}
\begin{align*}
&\quad \,   \|I_1(t+h)-I_1(t)\|_\beta\nonumber\\
& \le  \int_0^t \|\big(T(h)-I\big) T(t-s)\nabla \cdot (u(s)\nabla v(s))\|_\beta ds 
+ \int_t^{t+h} \|T(t+h-s)\nabla \cdot (u(s)\nabla v(s))\|_\beta ds\nonumber\\
&\le C h^{\tilde \beta-\beta}\int_0^t \|T(t-s)\nabla \cdot (u(s)\nabla v(s))\|_{\tilde \beta} ds+ C\int_t ^{t+h} (t+h-s)^{-\beta-\frac{1}{2}}\|u(s)\|_\infty\|\nabla v(s)\|_\infty ds\nonumber\\
&\le C h^{\tilde\beta-\beta}\int_0^ t (t-s)^{-\tilde \beta-\frac{1}{2}} e^{-(1-\delta)(t-s)}\|u(s)\|_\infty\|\nabla v(s)\|_\infty ds+  C   M M_1 h^{\frac{1}{2}-\beta}\nonumber\\
&\le C  M M_1  h^{\tilde \beta-\beta}+ C  M M_1 h^{\frac{1}{2}-\beta},
\end{align*}
and
\begin{align*}
\|I_2(t+h)-I_2(t)\|_\beta & \le Ch^{\beta}\int_{0}^{t}e^{-(1-\delta)(t-s)}(t-s)^{-\beta}\|u(s)\|_{\infty}\left[1+ a + \|u(s)\|_{\infty}\right]ds \nonumber\\
&  \quad+C\int_{t}^{t+h}(t+h-s)^{-\beta}\|u(s)\|_{\infty}\left[1+a+ \|u(s)\|_{\infty}\right]ds\nonumber\\
&\le  C M(1+a+M)(h^{\beta}+h^{1-\beta}).
\end{align*}
Hence, the mapping  $[t_0, \infty) \ni t \mapsto u(t) \in X^{\beta}$ is  uniformly H\"{o}lder continuous.
This, together with \eqref{Fractional power Imbedding-0}, implies that $u$ is H\"older continous in both space and time. Furthermore, the H\"older norm is uniformly finite for all $t\geq t_0$. 

Now, since $u$ is H\"{o}lder continuous with uniform bound for $t\geq t_0$, we can apply the classical parabolic estimates (see Theorems 5 and 10 in Chapter 3 of \cite{Friedman}) and use the equation of $v$ to get that \eqref{bound-on-holder-norm-eq} holds
for  $w= v(t,x )$, $\partial_{x_i}v(t,x )$,
$\partial_t v(t,x)$ or $\partial^2_{x_ix_j}v(t,x )$ for $1\le i,j\le N$.
 Then applying the classical parabolic estimates  and using the equation of $u$, we have  that \eqref{bound-on-holder-norm-eq} holds for $w=u(t,x )$, $\partial_{x_i}u(t,x )$, $\partial_t u(t,x )$, $\partial^2_{x_i x_j}u(t,x )$.

Finally,   assume that $v_0$ and $u_0$ are, respectively, uniformly bounded in $C_{\rm unif}^{2+\alpha, b}$ and $X_1$. Then $\|T(t)u_0\|_{\beta}\leq C_\beta$ uniformly for all $t\geq 0$ when $\beta<\frac12$. Then the above argument yields that $u(t,x)$ is H\"older continuous uniformly for all $t\in [0,\infty)$ and $x\in\R^N$. The boundary Schauder estimates (see e.g., \cite{LaSoUr}  and \cite[Theorem 3.2]{Lihe}) then yield that \eqref{bound-on-holder-norm-eq} holds
for  $w= v(t,x )$, $\partial_{x_i}v(t,x )$,
$\partial_t v(t,x)$ or $\partial^2_{x_ix_j}v(t,x )$ for $1\le i,j\le N$.
\end{proof}

\begin{proof}[Proof of Lemma \ref{asymptotic-lm1}]
{ By   Lemma  \ref{derivative-boundedness-lm}},  we have
$$
A=A(\chi,\|v_0\|_{X_1},\|u\|_\infty): =\sup_{t>1}\left(\|\nabla v(t,\cdot)\|_{\infty}+\|\Delta v(t,\cdot)\|_{\infty}\right)<\infty.
$$
We first  prove \eqref{asymptotic-eq1}.  To this end, for given $t_0\ge 1$, let $t_1:=t_0-\frac12$ and let   $\bar u(t, x)$ be the solution of the equation
\[\begin{cases}
	\bar u_{t} =  \Delta \bar u - \chi\nabla v(t,x)\cdot \nabla \bar u, \qquad & (t,x)\in  (t_1,\infty) \times \R^N,\\
	\bar u(t_1,\cdot) = u(t_1,\cdot), & x\in \R^N.
\end{cases}\]
It follows from the comparison principle for parabolic equations  that \
$$
0\le\bar {u}(t,x) \le  \|u(t_1,\cdot)\|_\infty\le \|u\|_\infty:=\sup_{s\ge 0}\|u(s,\cdot)\|_\infty \quad \forall\, (t,x)\in (t_1,\infty)\times\R^N.
$$

Let $\psi(t,x)$ and $\phi(t,x)$  be defined by
\begin{equation*}
\psi(t,x)=\min\big\{1,({a}/{b})^{1/\sigma}/\|u\|_{\infty}\big\}\,e^{-|\chi| A(t-t_1)}\,\bar u(t,x)
\end{equation*}
and
\begin{equation*}
\phi(t,x)=e^{(a+A|\chi|)(t-t_1)}\,\bar u(t,x),
\end{equation*}
respectively.
 Then, we have for $(t,x)\in (t_1,\infty)\times\R^N$,
\[
\begin{aligned}
    \psi_{t} = -|\chi| A \psi + \Delta \psi - \chi\nabla v\cdot\nabla\psi\le  \Delta \psi - \chi\nabla\cdot(\psi\nabla v) \le \Delta \psi - \chi\nabla\cdot(\psi\nabla v) + \psi(a-b\psi^\sigma),
\end{aligned}
\]
and
\[
\begin{aligned}
    \phi_{t} &= (a+|\chi| A) \phi + \Delta \phi - \chi\nabla v\cdot\nabla\phi\ge  \Delta \phi - \chi\nabla\cdot(\phi\nabla v)+ a\phi\\
    &\ge \Delta \phi - \chi\nabla\cdot(\phi\nabla v) + \phi(a-b\phi^\sigma).
\end{aligned}
\]
Since
$$
\psi(t_1,\cdot)\leq \phi(t_1,\cdot)=\bar u(t_1,\cdot)=u(t_1,\cdot),
$$
the comparison principle yields
$$
\psi(t,\cdot)\le u(t,\cdot)\le \phi(t,\cdot)\quad \forall\, t\ge t_1.
$$

Note    that
\begin{equation*}
\bar u(t, x) = \int_{\R^N} \Gamma(t, t_1, x, z)u(t_1, z)dz,
\end{equation*}
where $\Gamma(t,t_1,x,z)$ is the fundamental solution of the equation
\begin{equation*}
\begin{cases}
 \Gamma_t = \Delta_x \Gamma - \chi\nabla_x v\cdot\nabla_x\Gamma, \qquad \qquad \text{ for }(t,x)\in (s,\infty)\times \R^N,\\
\Gamma(t=t_1,t_1,x,z) = \delta_z(x).
	\end{cases}
\end{equation*}
Hence
\beq\label{new-added-eq1}
\begin{aligned}
&\min\left\{1,({a}/{b})^{1/\sigma}/\|u\|_{\infty}\right\}\,e^{-|\chi| A(t-t_1)} \int_{\R^N} \Gamma(t, t_1, x, z)u(t_1, z)dz\le u(s',x)\\
&\qquad\qquad\le  e^{(a+A|\chi|)(t-t_1)} \int_{\R^N} \Gamma(t, t_1, x, z)u(t_1, z)dx\quad \forall\, t\ge t_1,\,\, x\in\R^N.
\end{aligned}
\eeq

Let  $p>1$, $q$ be the conjugate exponent to $p$ and let $\alpha:=(1+p)/(2p)$. From the proof of \cite[(4.25)]{HaHe}, for any fixed $s_0,R>0$,  there exists a constant $C_0>0$ depending only on  $s_0$, $R$, $p$, $A$, $N$ and an upper bound of $|\chi|$ such that
\begin{equation}\label{comparable_kernel}
	\Gamma(t_0,t_1, x, z)^{\alpha p} \leq C_0\,\Gamma(t,t_1, y, z)\,\,\forall\, t\in[t_0,t_0+s_0],\,\, |x-y|\le R,\,\, z\in\R^N.
\end{equation}
By \cite[Theorem 10]{Aronson}, there exists $C_1\ge1$, depending only $A$, $N$ and an upper bound of $|\chi|$,  such that for any $x,y \in \R^N$ and  $0 \leq s < t$ with  $t - s \leq \frac12$,
\begin{equation*}
	({C_1(t-s)^{N/2}})^{-1} e^{-C_1\frac{|x-y|^2}{t-s}}
		\leq \Gamma(t,s, x,y)
		\leq {C_1}{(t-s)^{-N/2}} e^{-\frac{|x-y|^2}{C_1(t-s)}}.
\end{equation*}
This and $t_0=t_1+\frac12$ imply that  there is  $C_2(q,\chi,N)$ such that
\begin{equation}\label{gamma-lq}
    \|\Gamma(t_0,t_1,x,\cdot)^{1-\alpha}\|_q \le C_2\quad\forall\, x\in\R^N.
\end{equation}

By \eqref{new-added-eq1} and H\"older's inequality, we have
\begin{align*}
   u(t_0,x)&\le e^{(a+A|\chi|)/2} \int_{\R^N} \Gamma(t_0,t_1,x,z)\,u(t_1,z)\,dz\\
    &\le e^{(a+A|\chi|)/2}\,\|u\|_\infty^{1/q} \int_{\R^N} \left(u(t_1,z)\,\Gamma(t_0,t_1,x,z)^{\alpha p}\right)^{1/p}  \Gamma(t_0,t_1,x,z)^{(1-\alpha) }dz\\
    &\le  e^{(a+A|\chi|)/2}\,\|u\|_\infty^{1/q}\,\|\Gamma(t_0,t_1,x,\cdot)^{1-\alpha}\|_q \left(\int_{\R^N} u(t_1,z) \Gamma(t_0,t_1,x,z)^{\alpha p} dz\right)^{1/p}.
\end{align*}
Using \eqref{comparable_kernel} and \eqref{gamma-lq} yields for $t\in [t_0,t_0+s_0]$, 
\begin{align*}
   u(t_0,x)&\le  C_2 C_0^{1/p} e^{(a+A|\chi|)/2}\,\|u\|_\infty^{1/q}\, \left(\int_{\R^N} u(t_1,z) \Gamma(t,t_1,y,z) dz\right)^{1/p}\\
    &\le  C_2 C_0^{1/p} e^{(a+A|\chi|)/2 + |\chi| A(\frac12+s_0)/p}\max\left\{\|u\|_\infty^{1/q}, \|u\|_\infty\left({b}/{a}\right)^{1/(\sigma p)}\right\}u(t, y)^{1/p},
\end{align*}
where in the last inequality, we applied \eqref{new-added-eq1}. We can conclude with \eqref{asymptotic-eq1}.

\smallskip

 Next, we prove  \eqref{asymptotic-eq2}.  By a priori estimate for parabolic equations (see \cite[Theorem 7.22]{Lieberman}) and
the anisotropic Sobolev embedding for $q = N+3$  (see \cite[Lemma A3]{Engler}),  there is  $C_3= C_3(A, \chi, R, N)>0$ such that for any $(t,x) \in (1,\infty)\times \R^n$,
\[\begin{split}
	|\nabla u(t,x)|
		&\leq C_3\,\left(\|u\|_{L^q([t-\frac12,t]\times B_R(x))} + \|u(a-bu)\|_{L^q([t-\frac12,t]\times B_R(x))} \right)\\
		&\leq C_3\,\|u\|_{L^\infty([t-\frac12,t]\times B_R(x))}\,(1+a + b\|u\|_\infty).
\end{split}\]
By  \eqref{asymptotic-eq1},  for any $p\in (1,\infty)$, there exists a constant $C>0$ that depends only on {$a$, $b$,} $R$, $p$, $A$, $N$,  $\|u\|_\infty$, and an upper bound of $|\chi|$  such that
\[
|\nabla u(t,x)|
		\leq { C \,u(t,y)^{1/p}}
\]
for all $|x-y|\le R$, which proves \eqref{asymptotic-eq2}.
\end{proof}

\medskip

\noindent{\bf Acknowledgments.}
The authors would like to thank the referee for  his/her careful reading of the original manuscript and for valuable comments and suggestions which improved the presentation of this paper considerably.

\medskip

\noindent {\bf Conflict of Interest} On behalf of all authors, the corresponding author states that there is no conflict of interest.

\medskip

\noindent {\bf Ethics Approval}
This research did not involve human participants, animals, or any other ethical considerations that would require formal ethics approval.

\medskip

\noindent {\bf Funding} The work of Y. P. Zhang is partially supported by Simons Foundation, Travel Support for Mathematicians MPS-TSM-00007305, and a start-up grant at Auburn University.

\medskip

\noindent {\bf Data Availability} Data sharing is not applicable to this article as no datasets were generated or analysed during
the current study.

\end{document}